\documentclass[a4paper,12pt]{amsart}
\usepackage{bm}
\usepackage[nice]{nicefrac}
\usepackage{mathrsfs}
\usepackage[utf8]{inputenc}
\usepackage[T1]{fontenc}
\usepackage[english]{babel}
\usepackage{algorithm}
\usepackage{algorithmicx}
\usepackage[noend]{algpseudocode}
\usepackage{soul}
\usepackage{fullpage}
\usepackage{adjustbox}
\usepackage{amsfonts}
\usepackage{amsthm}
\usepackage{amssymb}
\usepackage{wasysym}
\usepackage{mathcomp}
\usepackage{mathrsfs}
\usepackage{booktabs}
\usepackage{graphicx}
\usepackage{epstopdf}
\usepackage[centertags]{amsmath}
\usepackage[pdftex]{hyperref}
\usepackage{tikz}
\usetikzlibrary{positioning}
\usepackage{tikz-cd}
\usepackage{cases}
\usepackage{mathtools, stmaryrd}
\usepackage{shuffle}
\usepackage{adjustbox}
\usepackage{ebgaramond}
\usepackage{epigraph}
\usepackage{quoting}
\usepackage[all, pdf]{xy}
\usepackage[normalem]{ulem}

\usepackage{graphicx}     

\newcommand{\mycomment}[1]{}

\newcommand{\calB}{\mathcal{B}}



\allowdisplaybreaks
\hypersetup{
    colorlinks,
    linkcolor=[RGB]{164,120,100},
    citecolor=[RGB]{164,120,100},
    urlcolor=[RGB]{164,120,100}
}

\title{The algebraic structures of social organizations: \\ the operad of cooperative games}
\author{Dylan Laplace Mermoud \quad and \quad Victor Roca i Lucio}
\thanks{This research was performed while the second author, in alphabetical order, benefited from the support of the ANR project SHoCos ANR-22-CE40-0008.}

\address{Dylan Laplace Mermoud, UMA, ENSTA, Institut Polytechnique de Paris, Palaiseau, France, CEDRIC, Conservatoire National des Arts et M{\'e}tiers, 75003 Paris, France}
\email{\href{mailto:dylan.laplace.mermoud@protonmail.com}{dylan.laplace.mermoud@protonmail.com}}

\address{Victor Roca i Lucio, Université Paris Cité and Sorbonne Université, CNRS, IMJ-PRG, F-75013 Paris, France}
\email{\href{mailto:rocalucio@imj-prg.fr}{rocalucio@imj-prg.fr}}

\date{\today}



\newtheorem*{corollaryintro}{Corollary}

\newtheorem{theoremintro}{Theorem}

\begin{document}

\theoremstyle{plain}
\newtheorem{theorem}{Theorem}
\newtheorem*{theorem*}{Theorem}
\newtheorem{lemma}{Lemma}
\newtheorem{proposition}{Proposition}
\newtheorem{assumption}{Assumption}
\newtheorem{corollary}{Corollary}

\theoremstyle{definition}
\newtheorem{definition}{Definition}
\newtheorem{hypothesis}{Hypothesis}

\theoremstyle{remark}
\newtheorem{remark}{\sc Remark}
\newtheorem{example}{\sc Example}
\newtheorem*{notation}{\sc Notation}


\newcommand{\qi}{\xrightarrow{ \,\smash{\raisebox{-0.65ex}{\ensuremath{\scriptstyle\sim}}}\,}}
\newcommand{\lqi}{\xleftarrow{ \,\smash{\raisebox{-0.65ex}{\ensuremath{\scriptstyle\sim}}}\,}}

\newcommand{\draftnote}[1]{\marginpar{\raggedright\textsf{\hspace{0pt} \tiny #1}}}
\newcommand{\ac}{{\scriptstyle \text{\rm !`}}}

\newcommand{\Ch}{\categ{Ch}}
\newcommand{\Catesmall}{\categ{Cat}_{\categ E, \mathrm{small}}}
\newcommand{\Operade}{\Operad_{\categ E}}
\newcommand{\Operadecprime}{\Operad_{\categ E}}
\newcommand{\Operadesmall}{\Operad_{\categ E, \mathrm{small}}}
\newcommand{\eII}{\mathcal{I}}
\newcommand{\ecateg}[1]{\mathcal{#1}}
\newcommand{\cmonlax}{\categ{CMon}_\lax}
\newcommand{\cmonoplax}{\categ{CMon}_\oplax}
\newcommand{\cmonstrong}{\categ{CMon}_\strong}
\newcommand{\cmonstrict}{\categ{CMon}_\strict}
\newcommand{\cat}{\mathrm{cat}}
\newcommand{\lax}{\mathrm{lax}}
\newcommand{\oplax}{\mathrm{oplax}}
\newcommand{\strong}{\mathrm{strong}}
\newcommand{\strict}{\mathrm{strict}}
\newcommand{\pl}{\mathrm{pl}}
\newcommand{\gr}{\mathrm{gr}}

\newcommand{\Cats}{\mathsf{Cats}}
\newcommand{\Functors}{\mathsf{Functors}}

\newcommand{\Ob}{\mathrm{Ob}}
\newcommand{\tr}{\mathrm{tr}}
\newcommand{\catch}{\mathsf{Ch}}
\newcommand{\categ}[1]{\mathsf{#1}}
\newcommand{\set}[1]{\mathrm{#1}}
\newcommand{\catoperad}[1]{\mathsf{#1}}
\newcommand{\operad}[1]{\mathcal{#1}}
\newcommand{\algebra}[1]{\mathrm{#1}}
\newcommand{\coalgebra}[1]{\mathrm{#1}}
\newcommand{\cooperad}[1]{\mathcal{#1}}
\newcommand{\ocooperad}[1]{\overline{\mathcal{#1}}}
\newcommand{\catofmod}[1]{{#1}\mathrm{-}\mathsf{mod}}
\newcommand{\catofcog}[1]{#1\mathrm{-}\mathsf{cog}}
\newcommand{\catcog}[1]{\Cog\left(#1\right)}

\newcommand{\catdgmod}[1]{\categ{dg}~#1\text{-}\categ{mod}}
\newcommand{\catpdgmod}[1]{\categ{pdg}~#1\text{-}\categ{mod}}
\newcommand{\catgrmod}[1]{\categ{gr}~#1\text{-}\categ{mod}}

\newcommand{\catdgalg}[1]{\categ{dg}~#1\text{-}\categ{alg}}
\newcommand{\catpdgalg}[1]{\categ{pdg}~#1\text{-}\categ{alg}}
\newcommand{\catgralg}[1]{\categ{gr}~#1\text{-}\categ{alg}}
\newcommand{\catcurvalg}[1]{\categ{curv}~#1\text{-}\categ{alg}}

\newcommand{\catdgcompalg}[1]{\categ{dg}~#1\text{-}\categ{alg}^{\mathsf{comp}}}
\newcommand{\catpdgcompalg}[1]{\categ{pdg}~#1\text{-}\categ{alg}^{\mathsf{comp}}}
\newcommand{\catgrcompalg}[1]{\categ{gr}~#1\text{-}\categ{alg}^{\mathsf{comp}}}
\newcommand{\catcurvcompalg}[1]{\categ{curv}~#1\text{-}\categ{alg}^{\mathsf{comp}}}

\newcommand{\catdgcog}[1]{\categ{dg}~#1\text{-}\categ{cog}}
\newcommand{\catpdgcog}[1]{\categ{pdg}~#1\text{-}\categ{cog}}
\newcommand{\catgrcog}[1]{\categ{gr}~#1\text{-}\categ{cog}}
\newcommand{\catcurvcog}[1]{\categ{curv}~#1\text{-}\categ{cog}}

\newcommand{\dgoperads}{\categ{dg}~\categ{Operads}}
\newcommand{\pdgoperads}{\categ{pdg}~\categ{Operads}}
\newcommand{\groperads}{\categ{gr}~\categ{Operads}}
\newcommand{\curvcooperads}{\categ{curv}~\categ{Cooperads}}
\newcommand{\grcooperads}{\categ{gr}~\categ{Cooperads}}

\newcommand{\pdgcooperads}{\categ{pdg}~\categ{Cooperads}}
\newcommand{\dgcooperads}{\categ{dg}~\categ{Cooperads}}

\newcommand{\catalg}[1]{\Alg\left(#1\right)}

\newcommand{\catofcolcomod}[1]{\mathsf{Col}\mathrm{-}{#1}\mathrm{-}\mathsf{comod}}
\newcommand{\catofcoalgebra}[1]{{#1}\mathrm{-}\mathsf{cog}}
\newcommand{\catofalg}[1]{\operad{#1}\mathrm{-}\mathsf{alg}}
\newcommand{\catofalgebra}[1]{{#1}\mathrm{-}\mathsf{alg}}
\newcommand{\mbs}{\mathsf{S}}
\newcommand{\catocol}[1]{\mathsf{O}_{\set{#1}}}
\newcommand{\catoftrees}{\mathsf{Trees}}
\newcommand{\cattcol}[1]{\catoftrees_{\set{#1}}}
\newcommand{\catcorcol}[1]{\mathsf{Corol}_{\set{#1}}}
\newcommand{\Einfty}{\mathcal{E}_{\infty}}
\newcommand{\nuEinfty}{\mathcal{nuE}_{\infty}}
\newcommand{\Fun}[3]{\mathrm{Fun}^{#1}\left(#2,#3\right)}
\newcommand{\III}{\operad{I}}
\newcommand{\treeoperad}{\mathbb{T}}
\newcommand{\treemodule}{\mathbb{T}}
\newcommand{\core}{\mathrm{Core}}
\newcommand{\forget}{\mathrm{U}}
\newcommand{\treemonad}{\mathbb{O}}
\newcommand{\cogcomonad}[1]{\mathbb{L}^{#1}}
\newcommand{\cofreecog}[1]{\mathrm{L}^{#1}}

\newcommand{\barfunctor}[1]{\mathrm{B}_{#1}}
\newcommand{\baradjoint}[1]{\mathrm{B}^\dag_{#1}}
\newcommand{\cobarfunctor}[1]{\mathrm{C}_{#1}}
\newcommand{\cobaradjoint}[1]{\mathrm{C}^\dag_{#1}}
\newcommand{\Operad}{\mathsf{Operad}}
\newcommand{\coOperad}{\mathsf{coOperad}}

\newcommand{\Aut}[1]{\mathrm{Aut}(#1)}

\newcommand{\verte}[1]{\mathrm{vert}(#1)}
\newcommand{\edge}[1]{\mathrm{edge}(#1)}
\newcommand{\leaves}[1]{\mathrm{leaves}(#1)}
\newcommand{\inner}[1]{\mathrm{inner}(#1)}
\newcommand{\inp}[1]{\mathrm{input}(#1)}

\newcommand{\field}{\mathbb{K}}
\newcommand{\mbk}{\mathbb{K}}
\newcommand{\mbn}{\mathbb{N}}

\newcommand{\id}{\mathrm{Id}}
\newcommand{\ii}{\mathrm{id}}
\newcommand{\unit}{\mathds{1}}

\newcommand{\Lin}{Lin}

\newcommand{\BijC}{\mathsf{Bij}_{C}}

\newcommand{\kk}{\Bbbk}
\newcommand{\PP}{\mathcal{P}}
\newcommand{\C}{\mathcal{C}}
\newcommand{\Sy}{\mathbb{S}}
\newcommand{\Tree}{\mathsf{Tree}}
\newcommand{\treemod}{\mathbb{T}}
\newcommand{\Dend}{\Omega}
\newcommand{\aDend}{\Omega^{\mathsf{act}}}
\newcommand{\cDend}{\Omega^{\mathsf{core}}}
\newcommand{\cDendpart}{\cDend_{\mathsf{part}}}

\newcommand{\build}{\mathrm{Build}}
\newcommand{\col}{\mathrm{col}}

\newcommand{\HOM}{\mathrm{HOM}}
\newcommand{\Hom}[3]{\mathrm{hom}_{#1}\left(#2 , #3 \right)}
\newcommand{\ov}{\overline}
\newcommand{\otimeshadamard}{\otimes_{\mathbb{H}}}

\newcommand{\I}{\mathcal{I}}
\newcommand{\Aa}{\mathcal{A}}
\newcommand{\BB}{\mathcal{B}}
\newcommand{\CC}{\mathcal{C}}
\newcommand{\DD}{\mathcal{D}}
\newcommand{\EE}{\mathcal{E}}
\newcommand{\FF}{\mathcal{F}}
\newcommand{\II}{\mathbb{1}}
\newcommand{\RR}{\mathcal{R}}
\newcommand{\UU}{\mathcal{U}}
\newcommand{\VV}{\mathcal{V}}
\newcommand{\WW}{\mathcal{W}}
\newcommand{\AAA}{\mathscr{A}}
\newcommand{\BBB}{\mathscr{B}}
\newcommand{\CCC}{\mathscr{C}}
\newcommand{\DDD}{\mathscr{D}}
\newcommand{\EEE}{\mathscr{E}}
\newcommand{\FFF}{\mathscr{F}}

\newcommand{\PPP}{\mathscr{P}}
\newcommand{\QQQ}{\mathscr{Q}}

\newcommand{\QQ}{\mathcal{Q}}

\newcommand{\KKK}{\mathscr{K}}
\newcommand{\KK}{\mathcal{K}}

\newcommand{\ra}{\rightarrow}

\newcommand{\Ai}{\mathcal{A}_{\infty}}
\newcommand{\uAi}{u\mathcal{A}_{\infty}}
\newcommand{\uEinfty}{u\mathcal{E}_{\infty}}
\newcommand{\uAW}{u\mathcal{AW}}

\newcommand{\uAlg}{\mathsf{Alg}}
\newcommand{\nuAlg}{\mathsf{nuAlg}}
\newcommand{\cAlg}{\mathsf{cAlg}}

\newcommand{\ucAlg}{\mathsf{ucAlg}}
\newcommand{\Cog}{\mathsf{Cog}}
\newcommand{\nuCog}{\mathsf{nuCog}}
\newcommand{\uAWcog}{u\mathcal{AW}-\mathsf{cog}}

\newcommand{\uCog}{\mathsf{uCog}}
\newcommand{\cCog}{\mathsf{cCog}}
\newcommand{\ucCog}{\mathsf{ucCog}}
\newcommand{\cNilCog}{\mathsf{cNilCog}}
\newcommand{\ucNilCog}{\mathsf{ucNilCog}}
\newcommand{\NilCog}{\mathsf{NilCog}}

\newcommand{\Cocom}{\mathsf{Cocom}}
\newcommand{\uCocom}{\mathsf{uCocom}}
\newcommand{\NilCocom}{\mathsf{NilCocom}}
\newcommand{\uNilCocom}{\mathsf{uNilCocom}}
\newcommand{\Liealg}{\mathsf{Lie}-\mathsf{alg}}
\newcommand{\cLiealg}{\mathsf{cLie}-\mathsf{alg}}
\newcommand{\Alg}{\mathsf{Alg}}
\newcommand{\Linfty}{\mathcal{L}_{\infty}}
\newcommand{\CMC}{\mathfrak{CMC}}
\newcommand{\Tfree}{\mathbb{T}}

\newcommand{\Hinich}{\mathsf{Hinich} -\mathsf{cog}}

\newcommand{\Ccomod}{\mathscr C -\mathsf{comod}}
\newcommand{\Pmod}{\mathscr P -\mathsf{mod}}

\newcommand{\cCoop}{\mathsf{cCoop}}

\newcommand{\Set}{\mathsf{Set}}
\newcommand{\sSet}{\mathsf{sSet}}
\newcommand{\dgMod}{\mathsf{dgMod}}
\newcommand{\gMod}{\mathsf{gMod}}
\newcommand{\catOrd}{\mathsf{Ord}}
\newcommand{\catBij}{\mathsf{Bij}}
\newcommand{\catSmod}{\mbs\mathsf{mod}}
\newcommand{\EEtw}{\mathcal{E}\text{-}\mathsf{Tw}}
\newcommand{\OpBim}{\mathsf{Op}\text{-}\mathsf{Bim}}

\newcommand{\Palg}{\mathcal{P}-\mathsf{alg}}
\newcommand{\Qalg}{\mathcal{Q}-\mathsf{alg}}
\newcommand{\Pcog}{\mathcal{P}-\mathsf{cog}}
\newcommand{\Qcog}{\mathcal{Q}-\mathsf{cog}}
\newcommand{\Ccog}{\mathcal{C}-\mathsf{cog}}
\newcommand{\Dcog}{\mathcal{D}-\mathsf{cog}}
\newcommand{\uCoCog}{\mathsf{uCoCog}}

\newcommand{\Artinalg}{\mathsf{Artin}-\mathsf{alg}}

\newcommand{\colim}[1]{\underset{#1}{\mathrm{colim}}}
\newcommand{\Map}{\mathrm{Map}}
\newcommand{\Def}{\mathrm{Def}}
\newcommand{\Bij}{\mathrm{Bij}}
\newcommand{\op}{\mathrm{op}}

\newcommand{\undern}{\underline{n}}
\newcommand{\dginterval}{{N{[1]}}}
\newcommand{\dgsimplex}[1]{{N{[#1]}}}

\newcommand{\cofree}{ T^c}
\newcommand{\Tw}{ Tw}
\newcommand{\End}{\mathcal{E}\mathrm{nd}}
\newcommand{\catEnd}{\mathsf{End}}
\newcommand{\coEnd}{\mathrm{co}\End}
\newcommand{\Mult}{\mathrm{Mult}}
\newcommand{\coMult}{\mathrm{coMult}}

\newcommand{\Lie}{\mathcal{L}\mathr{ie}}
\newcommand{\As}{\mathcal{A}\mathrm{s}}
\newcommand{\uAs}{\mathrm{u}\As}
\newcommand{\coAs}{\mathrm{co}\As}
\newcommand{\Com}{\mathcal{C}\mathrm{om}}
\newcommand{\uCom}{\mathrm{u}\Com}
\newcommand{\Perm}{\catoperad{Perm}}
\newcommand{\uBE}{\mathrm{u}\mathcal{BE}}
\newcommand{\uBEs}{{\uBE}^{\mathrm s}}

\newcommand{\comp}{\circ}
\newcommand{\restrictionextension}{\mathrm{RE}}
\newcommand{\extension}{\mathrm{E}}
\newcommand{\extensionone}{\mathrm{E}_1}
\newcommand{\extensiontwo}{\mathrm{E}_2}
\newcommand{\restrictionone}{\mathrm{R}_1}
\newcommand{\restrictiontwo}{\mathrm{R}_2}

\newcommand{\itemt}{\item[$\triangleright$]}

\newcommand{\poubelle}[1]{}

\newcommand{\bbR}{\mathbb{R}}

\definecolor{softblue}{RGB}{60, 100, 210}
\newcommand{\Victor}[1]{\textcolor{softblue}{#1}}
\definecolor{softcherry}{RGB}{200, 60, 80} 
\definecolor{cherryblossom}{RGB}{220, 120, 140}
\newcommand{\Dylan}[1]{\begingroup\color{softcherry}#1\endgroup}

\begin{abstract}
The main goal of this paper is to settle a conceptual framework for cooperative game theory in which the notion of composition/aggregation of games is the defining structure. This is done via the mathematical theory of algebraic operads: we start by endowing the collection of all cooperative games with any number of players with an operad structure, and we show that it generalises all the previous notions of sums, products and compositions of games considered by Owen, Shapley, von Neumann and Morgenstern, and many others. Furthermore, we explicitly compute this operad in terms of generators and relations, showing that the Möbius transform map induces a canonical isomorphism between the operad of cooperative games and the operad that encodes commutative triassociative algebras. In other words, we prove that any cooperative game is a linear combination of iterated compositions of the $2$-player bargaining game and the $2$-player dictator games. We show that many interesting classes of games (simple, balanced, capacities a.k.a fuzzy measures and convex functions, totally monotone, etc) are stable under compositions, and thus form suboperads. In the convex case, this gives by the submodularity theorem a new operad structure on the family of all generalized permutahedra. Finally, we focus on how solution concepts in cooperative game theory behave under composition: we study the core of a composite and describe it in terms of the core of its components, and we give explicit formulas for the Shapley value and the Banzhaf index of a compound game. 
\end{abstract}

\maketitle

\setcounter{tocdepth}{1}

\tableofcontents

\section*{Introduction}

\vspace{2pc}

\subsection*{Cooperative game theory} Social organizations amount to groups of individuals deciding to cooperate amongst themselves; it is nevertheless clear that, in practice, cooperation does not appear amongst every possible group of individuals. A simple model of this complex human behaviour is the notion of a \textit{cooperative game}. Mathematically, a cooperative game on a set of $n$-players is the data of a function $v$ from the set of subsets of $\{1,\cdots,n\}$ to the real numbers $\mathbb{R}$ such that $v(\emptyset) = 0$. One should interpret this definition as follows: $v$ is the function that assigns to every possible coalition the \textit{payoff} that the players in it would get if they form it. For example, a simple majority election can be modeled by a \textit{simple} game where winning coalitions are sent to $1$ and losing coalition to $0$. The theory of cooperative game seeks to understand how, given these payoffs, the players of the game are going to interact, and in particular, whether there are optimal coalitions for them to form. 

\medskip

This subject has a rich history dating back to Emile Borel and to John von Neumann, see \cite{Borel1924-BORAPD, vonNewmann1, leonard2010neumann}. After a major breakthrough by von Neumann and Morgenstern in \cite{vonNeumann2}, the subject grew considerably in the 60's, particularly under the influence of John Forbes Nash and Lloyd Stowell Shapley. Cooperative games predate strategic games, which were later introduced by Nash in the 50's as modeling \textit{competitive} behaviours. Further, understanding and developing the theory of cooperative games is even mentioned as one of the main open problems in mathematics in \cite{Openproblems}, alongside other well-known conjectures.

\medskip

One of the main difficulties in the theory of cooperative games is that, \textit{a priori}, they carry very little structure. The coalition function assigns to any possible coalition a value, but these values are not related to each other. So the only way to fully describe a $n$-player game is to write down a list of $2^{n-1}$ values in $\mathbb{R}$. Therefore, as the number of players increases, the complexity increases very quickly. 

\medskip

One way to decrease the complexity of a cooperative game is to view it as a \textit{composition} of different, simpler games and analyze them individually. For example, many types of elections systems are compositions of simpler voting games. Take, for instance, the formation of a government in a parliamentary system where each representative is elected on a specific constituency, as it is the case in France or in the United Kingdom. Each of these constituencies is a simple majority voting game, and all these voting games are aggregated into a \emph{quotient game}, which is the coalitional game played in the chamber of representatives. In fact, most social situations modeled by cooperative games are composite in nature, since the behavior of a player is very often determined by the strength relations in a subgame. A good example of this is the European Council, which is a voting game made of $27$ players, and where the behavior of each of the players is decided by the subgame of their respective national elections, which determine the political orientation of the choices made by the players. In turn, as we just discussed, the elections that determine the behavior of the player like France are themselves composite games. This shows that one has to consider not only compositions but \textit{iterated compositions} of games in order to analyze some social situation. 

\medskip

As early as the publication of \emph{Theory of Games and Economic Behavior}, von Neumann and Morgenstern~\cite{vonNeumann2} devoted a whole chapter (Chapter~IX) to the question of game decomposition, thus the study of decomposition of games is as old as the study of coalitional games themselves. However, in \textit{op. cit.}, they only treat the case where the larger game into which the subgames are inserted, called the \textit{quotient game} in the literature, is additive. Additive games are the simplest class of games, where the value of a coalition is determined by the sum of the values of each of the players. 

\medskip

A (non-additive) \textit{composition of games} was first mentioned by Shapley during a lecture at Princeton during the academic year of 1953/1954 \cite{shapley1954simple} for simple games, which are games where the values associated to coalitions are either $0$ or $1$. He later devoted several papers to the study of the properties of this composition of simple games, see \cite{ShapleyCompoundIII, ShapleyCompoundI, ShapleyCompoundII, ShapleyComposition}, which also considered in \cite{birnbaum1965modules}. This composition was independently considered and also generalized to normalized non-negative games by Owen in \cite{OwenTensor}. In \cite{ShapleyCompoundII}, Shapley raised the following question, called the \emph{aggregation problem} and formulated as follows: 

\vspace{1pc}

\begin{quoting}[leftmargin=1cm, rightmargin=1cm]
An important question in the application of game theory to economics and the social sciences is the extent to which it is permissible to treat firms, committees, political parties, labor unions, etc., as though they were individual players. Behind every game model played by such aggregates, there lies another, more detailed model: a compound game of which the original is the quotient. Given any solution concept, it is legitimate to ask how well it stands up under aggregation --- or disaggregation. To what extent are the theoretical predictions sensitive to the particular level of refinement adopted in the model?
\end{quoting}

\vspace{1pc}

Although there have been multiple works (see for instance \cite{OwenMultilinear, MegiddoKernel, MegiddoNucleolus, megiddo1975tensor}) on the compositions of games, there has been no systematic algebraic treatment of them. The first goal of this paper is to lay down an algebraic framework for the composition of cooperative games that encompasses all previous notions of compositions and that allows us to treat systematically any type of aggregation problem. This will be done via the theory of (algebraic) operads. 

\subsection*{Operad theory}
Operads are algebraic objects whose main function is to encode the different types of algebraic structures that appear in mathematics. They first appeared in algebraic topology in the 60's in order to encode the natural structure of iterated loop spaces, see the work of Boardman--Vogt \cite{BoardmanVogt73} and of May \cite{May72}. In these first examples, operads were defined as collections of topological spaces with an extra structure. However, it was in the beginning of the 90's that several authors realized that one could define operads in different settings, such as sets or vector spaces, and that the theory of these operads was crucial in understanding many types of algebraic structures. In particular, the theory of \textit{algebraic operads}, which are operads in vector spaces (or more generally, in chain complexes of vector spaces) became increasingly important in homotopical algebra. Since then, this theory has been successfully applied to various areas of mathematics such as algebraic geometry, representation theory, symplectic geometry, mathematical physics or combinatorics. We refer to \cite{Renaissance} for some examples of applications of operads, to \cite{MarklStasheff} for a textbook on operads in different contexts and to \cite{LodayVallette} for a standard reference on algebraic operads. See also \cite{OperadicCalculus} for a recent account of some applications of the theory of operads. 

\medskip

Let us give the unfamiliar reader an idea of what an algebraic operad is and of what it usually does. The basis data is that of a family of vector spaces $\{\mathcal{P}(n)\}_{n \geq 0}$ for all $n \geq 0$, and the structure is given by that of a family of partial composition bilinear maps 
\[
\circ_i: \mathcal{P}(n) \times \mathcal{P}(m) \longrightarrow \mathcal{P}(n+m-1)
\]
for all $i \in [n]$. One should interpret elements in $\mu$ in $\mathcal{P}(n)$ as "abstract operations" with $n$-inputs and one output: these abstract operations are best represented as rooted trees with $n$-leaves and one root. Interpreted in this way, the partial composition maps $\{\circ_i\}_{i \in [n]}$ then correspond to inserting the operation $\nu$ in $\mathcal{P}(m)$ in the $i$-th input of the operation $\mu$ in $\mathcal{P}(n)$ and can be represented as inserting the root of the rooted tree with $m$ leaves into the $i$-th leaf of the rooted tree with $n$ leaves. It is clear that the resulting tree has $(n+m-1)$ leaves, which explains way these operations land in $\mathcal{P}(n+m-1)$. These partial composition maps are subject to some axioms, which amount to saying that "the order in which we compose operations does not matter". They usually incorporate an action of the symmetric groups in order to encode the eventual symmetries of the operations. We refer to Definition \ref{definition: unital partial operad} for the full definition. 

\medskip

Using this abstract definition, one can encode different algebraic structures. For every operad, there is a natural notion of algebra over this operad. And, for every algebraic structure (where operations have many inputs and one output), there is a corresponding operad which encodes it. For instance, an \textit{associative algebra} is a vector space $A$ together with a bilinear binary product $\mu_A: A \times A \longrightarrow A$ such that $\mu_A(\mu_A(a,b),c) = \mu_A(a, \mu_A(b,c))$ for all $a,b,c \in A$. This operation with $2$-inputs and one output can be represented as a binary tree, and the associativity relation amounts to the equality $\mu_A \circ_1 \mu_A = \mu_A \circ_2 \mu_A$. There is an operad, denoted by $\mathcal{A}ss$, which is freely generated by such a binary operation which satisfies the associativity relation and algebras over this operad are associative algebras in the classical sense. Other structures like Lie algebras, commutative algebras, Poisson algebras and so on have their corresponding algebraic operad which encodes them in the same way. See Examples \ref{example: associative operad}, \ref{example: permutative operad} or \cite[Chapter 13]{LodayVallette} for more details on this example and for other examples.

\subsection*{The operad of all cooperative games}
The main new idea of this paper is to use the theory of algebraic operads in order to encode the compositional aspect of cooperative game theory. We define an algebraic operad structure on the collection of all cooperative games in any number of players, which is given by replacing in (an appropriate way) the $i$-th player in a game by the subgame that determines its behaviour. Thus, players and subgames become interchangeable in this framework via this operad structure. 

\medskip

More concretely, let $\mathbb{G}(n)$ be the set of all possible $n$-players game, which is naturally an $\mathbb{R}$-vector space. For any $n$-player game $\Gamma_A = (A, \alpha)$ with value function $\alpha$ and any $m$-player $\Gamma_B = (B,\beta)$ game with the value function $\beta$, we define the $(n+m-1)$-player game $\Gamma_A \circ_i \Gamma_B$ on the player set $A \diamond_i B = (A \setminus \{i\}) \cup B$ (we replace the $i$-th player in $A$ by the player in $B$) whose value function, for any $S \subseteq A \diamond_i B$, is given by: 

\[
\alpha \circ_i \beta(S) \coloneqq \beta(B)\alpha(S_A) + \partial_i \alpha(S_A) \beta(S_B)~, 
\]
\vspace{0.1pc}

where $S_A = S \cap A, S_B = S \cap B$ and $\partial_i \alpha(S_A) = \alpha(S_A \cup \{i\}) - \alpha(S_A)$. Let us explain this formula from a game theoretic point of view. First, the contribution of the players in $B$ which where previously composing the $i$-th player of $\alpha$ should be proportional to what the $i$-th player could provide: this is reflected by the term $\partial_i \alpha(S_A)$, which is the \textit{derivative} at the player $i$ and reflects the marginal contribution of this player to the coalition $S_A$. This term is multiplied by the value of the coalition $S_B$ which is given by the players in $B$. This term is added to the value of the coalition $\alpha(S_A)$ which no longer contains the player $i$. Finally, the term $\beta(B)$ acts as a normalization: what matters in a cooperative game is not the absolute value of each coalition, but the relative strength of the coalitions with respect to each other. 

\begin{theoremintro}[Theorem \ref{thm: gamers operad}]
The partial composition maps $\{\circ_i\}_i$ endow the collection of all cooperative games $\mathbb{G}$ with an operad structure.
\end{theoremintro}

Associated to any operad structure defined in terms of partial composition maps, there are total composition maps, which are obtained by iterating the partial compositions on all the inputs. These total compositions give, for any $k$-player game $\Gamma_0 = (K,\mu_k)$ and any $k$-tuple of games $\Gamma_j = (M_j, \mu_{i_j})$ for $j \in K$, a composite game $\Gamma_0[\Gamma_1,\cdots,\Gamma_k]$ on the player set given by the union of the player sets $\cup_{j = 1}^k M_j$. We show that these total composition maps recover the composition of simple games defined by Shapley in \cite{ShapleyComposition} and the composition of normalized games defined by Owen in \cite{OwenTensor}. So our partial composition maps generalize all the previously considered notions of compositions in cooperative game theory. 

\medskip

Now that we know what kind of algebraic structure the composition of cooperative games is, it is natural to ask whether the operad of cooperative games admits a presentation in terms of generators and relations —like, for instance, the operad that encodes associative algebras or the operad that encodes Lie algebras does. We build, using the Möbius transform (which encodes the Harsanyi dividends of a game), an explicit isomorphism of operads between the operad of all cooperative games and the operad $\mathcal{C}om\mathcal{T}riass$, which encodes a particular type of algebraic structure called \textit{commutative triassociative algebras}, and which was first defined by Vallette in \cite{Vallettepartition}. 

\begin{theoremintro}[Theorem \ref{thm: iso comtriass et operade}]
The Möbius transform defines a canonical isomorphism of operads 
\[
\mathbb{G} \cong \mathcal{C}om\mathcal{T}riass~,
\]
where $\mathbb{G}$ denotes the operad of cooperative games. 
\end{theoremintro}

As a direct consequence, we obtain an explicit presentation of the operad of all cooperative games by generators and relations. Concretely, we show that any cooperative game with $n$-players can be obtained as a sum of iterated compositions of the $2$-player dictator games and the $2$-player bargaining game (in a non-unique way). Let us also mention that this isomorphism restricts to an isomorphism of suboperad between the suboperad of additive games and the permutative operad of Example \ref{example: permutative operad}.

\subsection*{Distinguished suboperads and applications} In practice, the definition of a cooperative game is too large for it to be well behaved, and one often considers cooperative games which satisfy additional conditions. We show that many well studied classes of cooperative games are stable under the partial composition operations and therefore define suboperads of the operad of all cooperative games. Let us mention that, in general, these classes are not linear subspaces, so these suboperads are given by sets together with an operad structure (where the compositions are not linear, only set functions) instead of subvector spaces with an operad structure (which is bilinear). 

\begin{theoremintro}[Theorems \ref{prop: normalized games suboperads},\ref{prop: simple games suboperad},\ref{proposition: capacities are stable under composition},\ref{thm: opérade des convexes},\ref{thm: opérade des k-monotones},\ref{thm: operad of infinity alternating} and \ref{prop: balanced}]\label{thmintro: suboperads}
The following classes of cooperative games are stable under the partial composition operations and thus form suboperads of the operad of cooperative games.
\begin{enumerate}
\item Normalized games and simple monotone games.

\item Non-negative monotone games (capacities).

\item Non-negative convex games (and more generally, $k$-monotone games for all $k \geq 2$). 

\item Non-negative submodular games (and more generally, $k$-alternating games for $k \geq 2$). 

\item Belief functions and plausibility measures.

\item Non-negative and monotone balanced games.
\end{enumerate}
\end{theoremintro}

This is where cooperative game theory intersects with many different areas of mathematics. Indeed, functions from the power set of a finite set to $\mathbb{R}$ appear under different names in many areas. They are also known as non-additive discrete measures, also called \textit{fuzzy measures} or \emph{capacities} when non-negative and monotone, and frequently appear in decision theory~\cite{choquet1954theory, Sugeno, grabisch2010decade}. They are a generalization of classical measures and their integration theories (such as the Choquet or the Sugeno integrals) generalize the classical Lebesgue integral. See also \cite{Grabisch} for a textbook account of this. Thus, our operad structure can also be understood as an operad structure on all fuzzy measures. 

\medskip

Simple monotone games represent combinatorial objects known as \textit{clutters}, see \cite{billera1971composition}, and form a suboperad of our operad. Totally monotone games are also known as \textit{belief functions,} and appear in \emph{evidence theory}, also known as \emph{Dempster-Shafer theory}~\cite{dempster1967upper, shafer1976mathematical, kohlas2013mathematical}, which is a general framework for reasoning with uncertainty, and which generalizes the theory of Bayesian probabilities. \textit{Plausibility measures} are another name for $\infty$-alternating games which also appear in Dempster--Shafer theory, see \cite{shafer1976mathematical, Hohle87}. Balanced games are precisely those with a non-empty core by the Bondareva--Shapley theorem \cite{Bondareva, ShapleyCore}. 

\medskip

Let us point out that, in general, a suboperad is in particular an operad, so we get an operad structure on each of these sets of games. Moreover, we suspect that in many cases, these operad have extra structure and define operads in affine spaces, cones and other, more structured objects. This should be the subject of future research. 

\medskip

\subsection*{Generalized permutahedra and (poly)matroids} Coming back to Theorem \ref{thmintro: suboperads}, non-negative convex games (cooperative games whose payoff function is non-negative and convex) are perhaps the most interesting subset of games from the point of view of algebraic combinatorics and discrete geometry. They are in bijection with \textit{generalized permutahedra} in the sense of Postnikov in \cite{Postnikov} and this bijection is realized, on the one side, by a canonical polytope that can be associated to any cooperative game called the \textit{core}, which is also one of the main solution concepts in cooperative game theory (see below). Points in the core represent distributions of the total value of the game that benefit all players. Convex games are uniquely characterized by their core, which is always a generalized permutahedra.  

\begin{corollaryintro}
The collection of all generalized permutahedra admits an operad structure.
\end{corollaryintro}

This essentially follows from the bijection between generalized permutahedra and non-negative convex functions, see for instance \cite[Section 12]{AguiarArdila}, together with Theorem \ref{thmintro: suboperads}. So far, the formula for the partial compositions is  determined at the level of the associated function by the formula above, but we intend to also determine it on the geometrical side. A first step toward this direction is Theorem \ref{thm: inclusion of the tensor i of the cores}, which gives a partial description. However, the full answer lies beyond the scope of the present paper and will be the subject of future, forthcoming work, see \cite{deuxiemepapier}. Finally, let us mention that yet another name for non-negative convex functions is that of a \textit{polymatroid}, since they can be defined entirely by their rank functions, see \cite[Definition 1.1]{Polymatroids}. Therefore our operad structure induces an operad structure on all polymatroids, which are generalizations of the \textit{matroids} introduced by Whitney in \cite{Whitney35}. 

\subsection*{Aggregation of solutions of cooperative games} In social sciences, one is limited by the restricted accuracy of the measure instruments and the lack of control over the conditions under which the observations are made. This implies that the theoretical tools we apply to these observations should be robust enough to perform well under (dis)aggregation in order to meaningfully describe social interactions. We investigate, for some of the most popular solution concepts in cooperative game theory, the extend to which they behave well under the partial composition operations that we have defined. 

\medskip

It should be noted that expecting a strict compatibility with the partial composition is often too strong. Indeed, most concepts of solutions in cooperative game theory try to describe the relative strengths of the individuals and the coalitions they form, as well as how the eventual gains of a coalition will be shared among its members  depending on their individual strengths. These strengths might vary to some extent in the composite game: when we replace the $i$-th player of $A$ by a set of players $B$, some players in $B$ might gain bargaining strength by joining individually a coalition with other players in $A$, something which was not possible before. So it seems reasonable that while solutions of each of the games induce solutions of the composite game, not every solution of the composite game can be described like this, as in the composite game new non-trivial interactions between the two set of players can appear. However, game theory and especially its cooperative aspects, focuses on power balance, and comparison between wealth, or strength. Hence, it is natural to expect and work with order-preserving structures, especially orders defined with respect to inclusions of sets. 

\medskip

We start our study the solutions of compound games by studying their preimputations and imputations. Let $\Gamma = (N,v)$ be a game. Its affine hyperplane of preimputations $X(\Gamma)$ is given by vectors in $\mathbb{R}^N$ (the vector space generated by all the players of the game) which share the total value $v(N)$ of the grand coalition of the game. Within this affine hyperplane lies the simplex of imputations $I(\Gamma)$, which is given by all the ways to share this value which are individually rational (meaning that a player wins more than what he would alone). A key ingredient of our study of compound solutions is the \textit{partial tensor product} $\otimes_i$ bilinear map
\[
(- \otimes_i -): \mathbb{R}^{A} \otimes \mathbb{R}^{B} \longrightarrow \mathbb{R}^{A \diamond_i B}
\]
which, in fact, corresponds to the composition of additive games in the cooperative game operad. We start by showing that this map restricts to preimputations and imputations, meaning that the partial tensor product of two (pre)imputations of the components of a game gives an imputation of their composite. Moreover, we given conditions under which is map is injective and surjective. 

\medskip

Let us denote by $C(\Gamma)$ the \textit{core} of a game, which is the polytope of all the ways to share the total value of the game $v(N)$, where each coalition does better than it would do alone. Cores are straightforward generalization of deformed permutahedra, and, beyond their combinatorial interest as a natural and interesting class of polytopes, they carry fundamental property about the modeled social situation:  their non-emptiness certifies the possibility of cooperation among the players under consideration. As Shubik~\cite{shubik1982game} stated, ``\emph{a game that has a core has less potential for social
conflict than one without a core}''. Indeed, even when the total value of the game is large, if there exists no payment satisfying every coalition, it is very unlikely and irrational to assume that cooperation arises without any external intervention. Because of its intuitive definition, most of the concept solutions in cooperative game theory are compared to the core in their analysis. 

\medskip 

Notice that this polytope can be empty unless the game is balanced. We show that the core is compatible with the operad structure in the following way. 

\begin{theoremintro}[Theorems \ref{thm: inclusion of the tensor i of the cores} and \ref{thm: fake surjectivity des coeurs}]
Let \(\Gamma_A = (A, \alpha)\) and \(\Gamma_B = (B, \beta)\) be two non-negative games, and consider \(i \in A\). The partial tensor product restricts to a well defined map
    \[
    (- \otimes_i -): C(\Gamma_A) \times C(\Gamma_B) \longrightarrow C(\Gamma_A \circ_i \Gamma_B)
    \]
between the cores of the components and the core of the composite. Furthermore: 

\medskip

\begin{enumerate}
    \item suppose that $\alpha(\{i\})$ and $\beta(B)$ are positive, then this map is injective;

\medskip

    \item suppose that $\alpha(\{i\}) = 0$ and \(\beta(\{l\}) = 0\) for all \(l \in B\),  as well as $\beta(B) > 0$. Then any element $z \in C(\Gamma_A \circ_i \Gamma_B)$ decomposes as $x \otimes_i y$ with $x \in C(\Gamma_A)$ and $y \in I(\Gamma_B)$.
\end{enumerate}
\end{theoremintro}

From the second point of the above theorem we can deduce that of $C(\Gamma_A)$ is empty, so is $C(\Gamma_A \circ_i \Gamma_B).$ This fits well with the game theoretic point of view, since composing games heuristically corresponds to adding details to the social context: one cannot expect to create a solution by such an operation in a context where there is none. 

\medskip

Finally, we turn our attention to well known solution concepts such as the Shapley value and the Banzhaf index. The Shapley value, introduced in \cite{Shapley1953}, is a solution concept for fairly distributing the total gains or costs of a game among its players by weighting their respective contributions. It can be characterized as the unique vector associated to a game that satisfies \textit{efficiency, symmetry and linearity axioms} among others. We give an explicit formula for the Shapley value of a composite game in Proposition \ref{prop: shapley value of a composite game} using its description in terms of the Möbius transform given in \cite{harsanyi1958bargaining}. We then consider Banzhaf index, introduced by \cite{Penrose}, which is a power index which measures the power of a voter in a voting game where voting rights are not distributed equally. It was used by Banzhaf to study the voting power of voters in different states in the U.S.A with respect to the Electoral College system in \cite{banzhaf1968one}. He showed that a voter in New York has 3,312 more voting power that a voter in the district of Columbia, and that in general the Electoral College system favors bigger states. We show that the Banzhaf index is particularly well behaved with respect to the operad structure: its value on a compound game is the partial tensor product of the Banzhaf index of the quotient game and the Shapley value of the component game. 

\begin{theoremintro}[Theorem \ref{thm: Banzhaf value of a composite game}]
Let \(\Gamma_A = (A, \alpha)\) and \(\Gamma_B = (B, \beta)\) be two games, and let \(i \in A\). The Banzhaf index of the composite game is given by 
    \[
    \psi \left(\Gamma_A \circ_i \Gamma_B \right) = \frac{\psi(\Gamma_A) \otimes_i \phi(\Gamma_B)}{2^{\lvert B \rvert - 1}} ~,
    \]
    where these are two equal vectors in $\mathbb{R}^{\lvert A \rvert + \lvert B \rvert -1}$.    
\end{theoremintro}

To conclude, let us mention other solution concepts present in the literature. One of the first solution concepts, introduced in \cite{vonNeumann2}, was the theory of \textit{stable sets} with respect to the domination relation. Unfortunately, stable sets are not compatible with the partial tensor product in general, and the precise extent to which domination is preserved by these operations remains a mystery. There is also the \textit{nucleolus}, introduced in \cite{schmeidler1969nucleolus}, which is an allocation of the value of the grand coalition $v(N)$ which maximizes the smallest excess of coalitions, and which is a vector that always lies in the core (when it is non-empty). Its compatibility with Shapley's and Owen's compositions was studied by Megiddo in \cite{MegiddoKernel,MegiddoNucleolus}, where he shows it is compatible with the composition of simple games but gives counterexamples with respect to Owen's composition. There are also the \textit{kernel}, introduced in \cite{DavisMaschler63} and the \textit{bargaining set}, introduced in \cite{AumannMaschler64}, which are specific subsets of the imputations simplex given by allocations of the total value $v(N)$ which are stable under certain power relations of the game. Although we suspect that they are in general compatible with the operadic composition of cooperative games, this is beyond the scope of the present paper and deserves future research. 

\subsection*{Acknowledgments} We wish to thank the two anonymous referees for their valuable comments. The first author would like to thank ENSTA, Institut Polytechnique de Paris and the Sorbonne Center of Economics (CES) and the second author would like to thank the EPFL and the Université Paris Cité for the excellent working conditions that allowed this project to be carried out. Both authors would like to thank the blackboard in the corridor of the third floor of the math building of the EPFL, where this project began unexpectedly. 

\vspace{1.5pc}

\section{Recollections on cooperative game theory and on algebraic operads}

\vspace{2pc}

In this section, we review basic definitions and examples of both cooperative game theory and algebraic operad theory. 

\subsection{Recollections on cooperative game theory} In this subsection, we recall the definition of a cooperative game as well as some of the related definitions and properties that will be useful for us the following sections. Finally, we recall the definition of the \textit{core} of a game, which is one of the most popular solution concepts in this theory. 

\subsubsection{First definitions} We first recall the original definition of a cooperative/coalitional game introduced by von Neumann and Morgenstern in \cite{vonNeumann2}.

\begin{definition}[Coalitional game]
A \emph{coalitional} $n$\textit{-player game} $\Gamma$ amounts to the data of a pair $(N, v)$, where $N = \{1, \ldots, n\}$  is called the \emph{grand coalition}, composed of \textit{players}, and $v$ is a real-valued function, called the \emph{coalition function}, defined over all the subsets of $N$.
\end{definition}

\begin{notation}
Unless stated otherwise, we will from now on refer to coalitional games simply as games or as cooperative games. By the data of $\Gamma$, we will always mean a pair $(N, v)$. 
\end{notation}

\begin{remark}
Many of the statements and results in this paper work for games where the coalition function takes values in an \textit{ordered ring}, such as the integers $\mathbb{Z}$ or the rationals $\mathbb{Q}$. We will not emphasize this aspect in this work, and we will consider the coalition function to take values in the real numbers $\mathbb{R}$ for simplicity.
\end{remark}

\begin{remark}
A game can also be thought as a non-additive measure on the set $N = \{1, \ldots, n\}$. See the theory of Choquet and Sugeno integrals~\cite{Grabisch} for non-additive measures (also known as \textit{fuzzy measure}), which generalize the Lesbegue measure. 
\end{remark}

\begin{remark}
We consider games defined on an \textit{ordered set} $\{a_1, \cdots,a_n\}$ of players; we may use different labels on the sets of players in order to distinguish them.
\end{remark}

\begin{definition}[Grounded coalition game]
A game $\Gamma = (N,v)$ is \textit{grounded} if it satisfies $v(\emptyset) = 0$.
\end{definition}

It is very common in the literature to focus on \textit{grounded} coalition games, as it is not obvious at first sight how to interpret the value of $v(\emptyset)$. From now on,\textit{ we will assume all the games we consider are grounded unless stated otherwise. }

\begin{remark}
Non-necessarily grounded games have nevertheless been considered, for example by Shapley~\cite{ShapleyCompoundIII}. Some constructions, including Theorem \ref{thm: gamers operad}, do generalize to this setting but we will leave this type of games outside the scope of the paper. 
\end{remark}

We recall the definitions of a few classes of games which are of particular interest in the theory of cooperative games. For a more exhaustive recollection, we refer for instance to the textbook account of Grabisch in \cite{Grabisch}. 

\begin{definition}[Normalized game]
A game $\Gamma = (N, v)$ is \emph{normalized} if $v(N) = 1$. 
\end{definition}

When studying a coalitional game, the focus is on the relative power of coalitions, which can be measured in multiple ways. That is why it is customary to set the value of the grand coalition to be \(1\). 

\begin{definition}[Non-negative game]
A game $\Gamma = (N, v)$ is \emph{non-negative} if $v(S) \geq 0$ for all $S \subseteq N$. 
\end{definition}

\begin{definition}[Monotone game]
A game $\Gamma = (N, v)$ is \emph{monotone} if \(v(S) \leq v(T)\) whenever \(S \subseteq T\).
\end{definition}

\begin{remark}
Notice that a grounded monotone game $\Gamma = (N, v)$ is in particular non-negative, since we have that $0 = v(\emptyset) \leq v(S)$ for any coalition $S$ in $N$.
\end{remark}

\subsubsection{Duality, derivatives and the Möbius transform}\label{Subsubsection: duality and derivatives}

\begin{definition}[Dual of a game]\label{definition: dual game}
Let $\Gamma = (N, v)$ be a game. Its \textit{dual game} $\Gamma^* = (N, v^*)$ is the game determined by the coalition function $v^*$, which is given, for all coalition \(S\), by
\[
v^*(S) \coloneqq v(N) - v(N \setminus S)~.
\]
\end{definition}

\begin{remark}
As already noticed by Fujishige and Murofushi~\cite{fujimoto2007some}, an excellent interpretation of the concept of duality has been formulated by Funaki~\cite{funaki1998dual}, who wrote the following. Usually, the value assigned to a coalition in a cooperative game is interpreted as the maximum that this coalition can obtain in the game given the worst possible context for them. It is the value it can guarantee to itself no matter what. This can be considered the  \emph{pessimistic} interpretation of the value of a game. 

\medskip

According to Funaki, the dual game represents the \emph{optimistic} version of the game: its value at a coalition is what it can get in the best possible situation for them. Given a social context described as a coalition function $v$, the observed value of a coalition $S$ is expected to lie between \(v(S)\) and \(v^*(S)\). 
\end{remark}

Another important concept is the \textit{derivative} of a game at a given player. It can be interpreted as the complete description of the contribution of player $i$ in the game: for each coalition, the derivative provides quantitative information about the marginal contribution of the new player. 

\begin{definition}[Derivative of a game]
Let $\Gamma = (N, v)$ be a $n$-player game and let $i \in N$ be a player. The \textit{derivative} of $\Gamma$ at the player $i$ is defined, for all \(S \subseteq N\), as 
\[
\partial_i v(S) \coloneqq v(S \cup \{i\}) - v(S\setminus \{i\})~. 
\] 
\end{definition}

Alternatively, the derivative of a game at \(i \in N\) can be seen as a game \(\partial_i \Gamma = (N \setminus \{i\}, \partial_i v)\) with \(\partial_i v(S) = v(S \cup \{i\}) - v(S)\) for all \(S \subseteq N \setminus \{i\}\). 

\medskip 

It is possible to define the derivative of a game with respect to a coalition $T \in \mathcal{P}(N)$ by induction. For all $S \in \mathcal{P}(N)$ and for all $i \in T$, we set 
\[
\partial_T v(S) \coloneqq \partial_{T \setminus \{i\}} \left( \partial_i v(S) \right).  
\]
This derivative can be computed directly: for all $S \in \mathcal{P}(N)$ and $T \in N \setminus S$, it is given by 
\[
\partial_T v(S) = \sum_{K \subseteq T} (-1)^{\lvert T \setminus K \rvert} v(S \cup K).
\]

Closely related is the Möbius transform. In general, any function defined on a poset admits a M{\"o}bius inverse, computed in terms of the M{\"o}bius functions of the poset. In our particular case, we are considering the poset of all subsets of $N$, whose M{\"o}bius function on an interval $[T,S]$ is given by $(-1)^{\lvert S \setminus T \rvert}$. See \cite{Rota1} for more details. 

\begin{definition}[Möbius transform]\label{definition: Mobius inversion}
Let $v: \mathcal{P}(N) \longrightarrow \mathbb{R}$ be a function from the poset of subsets of $N = \{1,\cdots,n\}$ to the real numbers. Its \textit{Möbius transform} $\mu^v: \mathcal{P}(N) \longrightarrow \mathbb{R}$ is the function whose value at $S \subseteq N$ is given by
\[
\mu^v(S) \coloneqq \sum_{T \subseteq S} (-1)^{\lvert S \setminus T \rvert} v(T)~.
\]
\end{definition}

This procedure defines an involution on functions, as one can recover the original coalition function $v$ from its Möbius inverse $\mu^v$, through the Zeta transform, given by
\[
\zeta^{\mu^v}(S) = v(S) = \sum_{T \subseteq S} \mu^v(T)~.
\]
Looking at this equation gives some intuition about these transforms. Because the value of a coalition is given by the sum of the values of its subcoalitions of its M{\"o}bius inverse, the latter can be consider as a map associating with any coalition an additional amount, possibly negative, that it is providing to the larger coalition it is a part of. 

\medskip 

One can understand the Möbius transform procedure as a change of basis. A canonical basis of the vector space of functions $v: \mathcal{P}(N) \longrightarrow \mathbb{R}$ is given by the Dirac games $\{\delta_S\}_{S \subseteq N}$~\cite{Shapley1953}. For any $S \subseteq N$, the value of $\delta_S$ is $1$ on $S$ and zero elsewhere. This basis is sent, under the Möbius transform, to the basis given by \textit{unanimity games} $\{u_S\}_{S \subseteq N}$. For any $S \subseteq N$, the value of $u_S$ on a subset $T$ is $1$ if it contains $S$ and zero otherwise.

\begin{lemma}[{\cite[Theorem~2.16, Lemma~2.34]{Grabisch}, \cite[Definition~3]{fujimoto2007some}}]\label{lemma: duality and derivatives}
Let \(\Gamma = (N, v)\) be a game and let $i \in N$ be a player.
\begin{enumerate}
    \item The duality involution and the derivative are compatible in the following sense:
    \[
    \partial_i v^*(S) = \partial_i v(N \setminus S). 
    \]
    \item The duality involution and the M{\"o}bius transform are compatible in the following sense: 
\[
\mu^{v^*}(S) = (-1)^{\lvert S \rvert + 1} \sum_{T \supseteq S} \mu^v(T) = \sum_{T \cap S \neq \emptyset} \mu^v(T). 
\]
\end{enumerate}
\end{lemma}

\subsubsection{The core of a game}\label{subsubsection: the core of a game} Let us recall one of the most popular solution concepts in cooperative game theory. See Section \ref{Section: aggregation of solutions} for more solution concepts. 

\medskip

Let us consider a game \(\Gamma = (N, v)\). We can consider the vector space generated by the players of the game $\bbR^N$, where each element of the canonical basis is a player in $N$. 

\medskip

Let $x = (x_1, \cdots, x_n)$ be a vector in $\bbR^N$, where $N = \{1,\cdots,n\}$. For any $S \subseteq N$, we set
    \[
    x(S) \coloneqq \sum_{i \in S} x_i~. 
    \]
Two particular subsets of $\bbR^N$ are of particular interest from the point of view of cooperative games:

\begin{itemize}
    \item a payment vector $x \in \bbR^N$ is a \emph{preimputation} if $x(N) = v(N)$, i.e., if the whole value $v(N)$ is allocated among the players. The set of preimputations is denoted by $X(\Gamma)$.
    \item a payment vector $x \in X(\Gamma)$ is an \emph{imputation} if $x_i \geq v(\{i\})$ for all \(i \in N\), i.e., if each player gets at least its own value. The set of imputations is denoted by $I(\Gamma)$.
\end{itemize} 

We define the \textit{core} of the game \(\Gamma = (N, v)\) as the set of (pre)imputations $x$ where for any coalition $S$, the payment $x(S)$ is greater than their worth $v(S)$. In other terms, it is the polytope $C(\Gamma)$ given by 
\[
C(\Gamma) = \{x \in X(\Gamma) \mid x(S) \geq v(S), \hspace{2pt} \forall S \in \mathcal{P}(N)\}. 
\]
This polytope might be empty. However, when it is not empty, elements in it correspond to ways of distributing the total wealth of the game $v(N)$ in such a way that any proper coalition is better off than if it stayed alone. In this sense, they correspond to solutions of the game, that is, to all the ways in which all the players in $N$ can cooperate, form the grand coalition and distribute the total value of the game amongst themselves. The nonemptiness of the core is a vital property of the game, and it is commonly admitted that a game with an empty core represents a social situation in which cooperation will not naturally emerge. Moreover, many of the other solution concepts are compared to the core when they are defined or studied, whether they include or are included in it, such as the market equilibria in the Arrow-Debreu model of an economy~\cite{debreu1963limit}, model for which they were awarded the Nobel prize in Economics, the Shapley value for convex games~\cite{shapley1971cores}, the nucleolus~\cite{schmeidler1969nucleolus}, or the bargaining set~\cite{AumannMaschler64}. For a more detailed account on the core of cooperative games, see~\cite{grabisch2013core, Grabisch}.

\subsection{Recollections on algebraic operads}
In this subsection, we recall the definition of an algebraic operad, that is, of an operad defined in terms of vector spaces together with linear composition maps. Roughly speaking, algebraic operads are algebraic structures which, on one hand, generalize associative algebras, and on the other, encode other types of algebraic structures. A standard textbook account of this theory can be found in \cite{LodayVallette}.

\subsubsection{Representations of finite groups} Let us fix $G$ a finite group, we recall the definition of a linear representation of the group $G$. 

\begin{definition}[$G$-module]
A $G$\textit{-module}, also called a \textit{representation} of the group $G$, amounts 
to the data of a pair $(V,\rho)$, where $V$ is a vector space and 
\[
\rho: G \longrightarrow \mathrm{GL}(V)~,
\]
is a group morphism, where $\mathrm{GL}(V)$ is the automorphism group of $V$.
\end{definition}

Morphisms of $G$-modules are given by \textit{equivariant} linear maps. Let $(V,\rho_V)$ and $(W,\rho_W)$ be two $G$-modules, a linear map $f: V \longrightarrow W$ is equivariant if, for all $g$ in $G$, 
\[
f \circ \rho_V(g) = \rho_W(g) \circ f~.
\]

\begin{notation}
Given a $G$-module $(V,\rho)$, we will from now on omit the structural map $\rho$, denoting $\rho(g)(v)$ simply by $g \star v$ (the image of a vector $v$ in $V$ by the action of $G$). 
\end{notation}

\begin{example}\leavevmode
\begin{enumerate}
\item Let $V$ be any vector space and let $g \star v = v$ for any $v \in V$ and any $g \in G$. Then $V$ is a vector space endowed with the \textit{trivial} $G$-module structure. In particular, when $V = \mathbb{R}$, it is called the \textit{trivial representation of} $G$. 

\medskip

\item Consider the vector space $\mathbb{R}[G]$ defined as 
\[
\mathbb{R}[G] \coloneqq \bigoplus_{g \in G}\mathbb{R}e_g~,
\]
that is, the vector space whose basis is indexed by elements in $G$, and consider the $G$-action on it given by $g \star e_h = e_{gh}$. This vector space with this $G$-module structure is called the \textit{regular representation of }$G$. 

\medskip

\item Consider $G$ to be the symmetric group of permutations of $n$ elements $\mathbb{S}_n$ and let $V$ be $\mathbb{R}^n$
with its canonical basis. The action $\sigma \star (x_1,\cdots,x_n) = (x_{\sigma(1)},\cdots,x_{\sigma(n)})$ endows $\mathbb{R}^n$ with an $\mathbb{S}_n$-module structure, called the \textit{standard representation} of $\mathbb{S}_n$. 
\end{enumerate}
\end{example}

\subsubsection{$\mathbb S$-modules} We are going to consider families of $\mathbb{S}_n$-modules, for $n \geq 0$, where $\mathbb{S}_n$ are the symmetric group of permutations of $n$ elements. 

\begin{definition}[$\mathbb{S}$-module]
An $\mathbb{S}$-module $M$ amounts to the data of a collection $\{M(n)\}_{n \geq 0}$, where $M(n)$ is an $\mathbb{S}_n$-module for all $n \geq 0$. 
\end{definition} 

A morphism of $\mathbb{S}$-modules $f: M \longrightarrow N$ is the data of a family of $\mathbb{S}_n$-equivariant maps $f(n): M(n) \longrightarrow N(n)$ for all $n \geq 0$.

\begin{remark}[Pictorial description]
The elements in $M(n)$ are called \textit{arity} $n$ \textit{operations}, and can be depicted as rooted trees with $n$ inputs or leaves and one output or root, where the leaves have labels ranging from $1$ to $n$. The action of $\mathbb{S}_n$ on an element $m_n$ in $M(n)$ can be depicted as the action that permutes the labels of the leaves

\[
\includegraphics[width=75mm,scale=1]{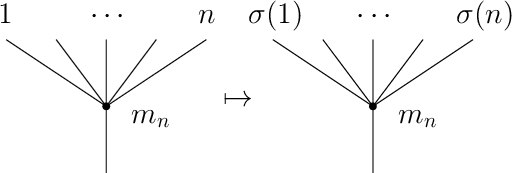}
\]

of the rooted tree $m_n$. This action encodes the \textit{symmetries} of the operation $m_n$. Beware that they might be invariant under this action. For example, an arity $2$ operation $m_2$ might satisfy $(21) \star m_2 = m_2$; in that case it encodes a symmetric operation, like a commutative product.
\end{remark}

\subsubsection{Operads defined in terms of partial compositions} A partial operad is the data of an extra structure on an $\mathbb{S}$-module. Following the description of elements in arity $n$ as rooted trees with $n$ inputs and one output, the partial compositions give operations that allow us to \textit{insert} the root of any given tree into one of the inputs of another tree.

\begin{notation}
When dealing with several sets at the same time, we will often use the notation $[k]$ for the set $\{1, \cdots, k\}$ with $k$ elements.
\end{notation}

\begin{definition}[Unital partial operad]\label{definition: unital partial operad}
A \textit{unital partial operad} $\operad P$ amounts to the data of a triple $(\operad P,\{\circ_i\},\eta)$, where $\operad P$ is an $\mathbb{S}$-module endowed with \textit{partial composition maps}:
\[
\circ_i: \operad P(n) \times \operad P(m) \longrightarrow \operad P(n+m-1)~,
\]
for all $1 \leq i \leq n$, and with a unit map $\eta: \mathbb{R} \longrightarrow \operad P(1)$. The partial composition maps have to be \textit{bilinear} and have to satisfy the following two axioms, for any $\lambda \in \operad P(p), \mu \in \operad P(q)$ et $\nu \in \operad P(r)$, where $p,q,r \geq 0$. 

\begin{enumerate}

\medskip

    \item The \textit{sequential axiom} 
    \[
    (\lambda \circ_i \mu) \circ_{i+j-1} \nu = \lambda \circ_i (\mu \circ_j \nu)
    \]
    for any $i \in [p]$ and any $j \in [q]$.

    \medskip

    \item The \textit{parallel axiom}
    \[
    (\lambda \circ_i \mu) \circ_{k+q-1} \nu = (\lambda \circ_k \nu) \circ_i \mu
    \]
    for any $1 \leq i < k \leq p$.
\end{enumerate}

\medskip 

Furthermore, the partial composition maps are compatible with the action of the symmetric groups. For any $\lambda \in \operad P(p), \mu \in \operad P(q)$:
\begin{enumerate}
\medskip

    \item we have that
    \[
    \lambda \circ_i (\sigma \star \mu) = {\sigma'} \star (\lambda \circ_i \mu)~,
    \]
    where $\sigma$ is in $\mathbb S_q$, and where $\sigma'$ is the unique permutation in $\mathbb S_{p+q-1}$ which acts as $\sigma$ on $\llbracket i, i+q-1 \rrbracket$ and as the identity elsewhere.

\medskip
    \item We have that 
    \[
    (\tau \star \lambda) \circ_{\tau(i)} \mu = \tau' \star (\lambda \circ_i \mu)~,
    \]
    where $\tau$ is in $\mathbb S_p$, and where $\tau'$ is the unique permutation in $\mathbb S_{p + q - 1}$ which acts as $\tau$ on $[p+q-1] \setminus \llbracket i + 1, i + q - 1 \rrbracket$ and which sends $\llbracket i, i + q - 1 \rrbracket$ to $\llbracket \tau(i), \tau(i) + q - 1 \rrbracket$.
\end{enumerate}

\medskip

Finally, the unit map $\eta$ specifies an element $\eta(1)$ in $\operad P(1)$ which acts as a unit with respect to the partial composition maps: for any $\mu$ in $P(n)$, $\mu \circ_i \eta(1) = \mu$ and $\eta(1) \circ_1 \mu = \mu$, for all $n \geq 0$ and $1 \leq i \leq n$. 
\end{definition}

A morphisms of operads $f: \operad P \longrightarrow \operad Q$ amounts to the data of a morphism of $\mathbb{S}$-modules $f$ which commutes with the partial composition maps and the units.

\begin{remark}[Pictorial description]
Remember that one can depict an operation in $\operad P(n)$ and $\operad P(m)$ as, respectively, a rooted trees with $n$ inputs and a rooted tree with $m$ inputs. Given $\mu_1$ in $\operad P(n)$ and $\mu_2$ in $\operad P(m)$, the partial composite $\mu_1 \circ_i \mu_2$ can be depicted as the rooted tree with $(n+m-1)$ inputs obtained by inserting the tree $\mu_2$ in the $i$-th leaf of the tree $\mu_1$
\[
\includegraphics[width=100mm,scale=1]{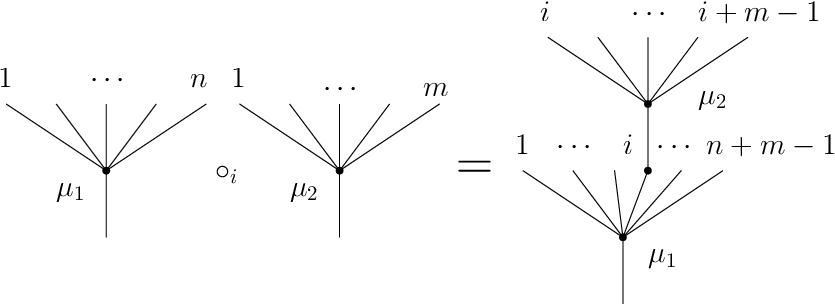}
\]

We invite the unfamiliar reader to represent, as an exercise, the axioms and the compatibilities with respect to the actions of the symmetric groups in terms of rooted trees, as they are easier to conceptualize.
\end{remark}

\begin{notation}\label{notation: linear operads versus other operads}
One can define operads in different contexts (or symmetric monoidal categories), such as sets, polytopes, topological spaces, affine spaces and many more. For example, a \textit{set operad} is a collection of sets $\{\operad P(n)\}$ together with partial composition \textit{set maps} $\{\circ_i\}$ which satisfy the same axioms as in Definition \ref{definition: unital partial operad}. If $\{\operad P(n)\}$ is now a family of polytopes, then the maps $\{\circ_i\}$ are required to satisfy an extra compatibility condition with respect to this polytope structure, see for instance \cite[Definition 7]{DiagonalAssociahedra}. And finally, if  $\{\operad P(n)\}$ is a family of vector spaces, then the maps $\{\circ_i\}$ are required to be \textit{bilinear} and this corresponds precisely to Definition \ref{definition: unital partial operad}. In order to distinguish operads in vector spaces from other types of operads, we will sometimes use the terms \textit{linear operad} or \textit{algebraic operad} to refer to them.
\end{notation}

\subsubsection{Total composition of an operad} 
Operads can also be defined in terms of total composition maps, as it was done in \cite{May72}. These are families of maps 

\[
\gamma(i_1, \cdots, i_k): \PP(k) \otimes \PP(i_1) \otimes \cdots \otimes \PP(i_k) \longrightarrow \PP(n)~,
\]
\vspace{0.1pc}

for any $k \geq 0$ and any $k$-tuple $(i_1,\cdots,i_k)$ such that $i_1 + \cdots i_k =n$, which satisfy explicit equivariance and the associativity conditions, as stated in \cite[Proposition 5.3.1]{LodayVallette} for example. Notice that in the definition of these total composition maps we consider the tensor product of these vector spaces in order to encode the multilinearity conditions that they satisfy. 

\begin{remark}[Pictorial description]
An element in $\PP(k) \otimes \PP(i_1) \otimes \cdots \otimes \PP(i_k) $ can be depicted by a finite sum of \textit{two-levelled rooted trees}, where the bottom tree as $k$-inputs and the top trees have $i_j$-inputs each. We can depict an operation $\mu_k$ in $\operad P(k)$ at the bottom and operations $\mu_{i_j}$ in $\operad P(i_j)$ at the top, for $1 \leq j \leq k$, as follows:
\[
\includegraphics[scale=0.725]{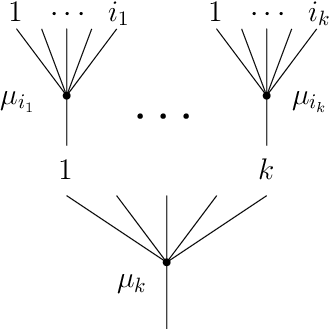}
\] 
where $i_1 + \cdots + i_k = n$. Then, the image of the total composition map $\gamma(i_1, \cdots, i_k)$ is precisely given by inserting every operation at the top into its corresponding leaf at the bottom.
\end{remark}

\begin{proposition}[{\cite[Proposition 5.3.4]{LodayVallette}}]\label{prop: compose total = composee partielle}
The data of an operad defined in terms of total composition maps is equivalent to the data of a unital partial operad.
\end{proposition}

\begin{proof}[Sketch of proof]
Let $(\PP,\{\circ_i\},\eta)$ be a unital partial operad, one defines the maps 
\[
\gamma(i_1, \cdots, i_k): \PP(k) \otimes \PP(i_1) \otimes \cdots \otimes \PP(i_k) \longrightarrow \PP(n)~,
\]
given by iterating $k$ times the partial composition maps $(- \circ_1 (- \circ_2 ( \cdots (\circ_k -)))))$ lands on $\PP(k + i_1 + \cdots + i_k -k)$. Conversely, given the maps $\gamma(i_1, \cdots, i_k)$, we can construct partial composition maps by plugging the unit almost everywhere $\mu \circ_i \nu \coloneqq \gamma(\mu; \mathrm{id}, \cdots,  \nu, \cdots, \mathrm{id})$, except at the $i^{\text{th}}$ place. It can be checked that, under this correspondence, both sets of axioms can be identified.
\end{proof}

\begin{remark}
Yet another equivalent definition of an operad is as a monoid in the monoidal category of $\mathbb{S}$-modules, where the  monoidal product $\circ$ called the \textit{composition product} is defined as 

\[
M \comp L(n) \coloneqq \bigoplus_{k\geq 0} M(k) \otimes_{\mathbb{S}_k} \left( \bigoplus_{i_1 + \cdots + i_k = n} \mathrm{Ind}_{\mathbb{S}_{i_1} \times \cdots \times \mathbb{S}_{i_k}}^{\mathbb{S}_n} (L(i_1) \otimes \cdots \otimes L(i_k))\right)~.
\]
\vspace{0.1pc}

for any two $\mathbb{S}$-modules $M,L$. We refer to \cite[Chapter 5]{LodayVallette} for more details.
\end{remark}

\subsubsection{Free operads, generators and relations} Let us first briefly explain what we mean by generators, relations and presentations. Recall that the \textit{free group} on a set $X = \{x_1,\cdots,x_n\}$, denoted by $\mathbb{F}(X)$, is obtained by considering the set of all possible words using the elements $\{x_i\}$ and their formal inverses $\{x_i^{-1}\}$. And a (finite) presentation of a group $G$ amounts to an isomorphism 
\[
G \cong \mathbb{F}(E)/(R)~,
\]
where $E$ is some finite set, and where we consider the quotient of this free group on $E$ by the normal subgroup generated by a finite number of words $R$. The key ingredient to make sense of this is simply the \textit{existence} of the free group on arbitrary generators.

\medskip

Let $M$ be an $\mathbb{S}$-module. The \textit{free operad} on $M$, denoted by $\mathbb{T}(M)$, is given by all possible partial compositions of the generators in $M$. Concretely, one can represent the $\mathbb{S}$-module of all these possible partial compositions of elements in $M$ as rooted trees with vertices labeled by the elements in $M$. To construct $\mathbb{T}(M)$, one then sums all possible labels by elements in $M$ over all possible rooted trees, where the arity corresponds to the numbers of leaves in the tree. 

\medskip

An explicit construction is as follows. Let $\mathrm{RT}$ denote the set of all rooted trees, and let $t$ be a rooted tree. The \textit{tree tensor} of $M$ is defined as 
\[
M(t) \coloneqq \bigotimes_{v \in \mathrm{vert}(t)} M(|\mathrm{in}(v)|)~,
\]
where $v$ ranges over all the vertices of the rooted tree $t$, and where $|\mathrm{in}(v)|$ denotes the number of incoming edges of the vertex $v$. The precise definition of the above tensor product over the ordered set $\mathrm{vert}(t)$ is given in \cite[Section 5.1.14]{LodayVallette}; roughly, one can think of it as taking the tensor product of the different $M(|\mathrm{in}(v)|)$ over all the elements $v$ in $\mathrm{vert}(t)$, were we have equated the action that permutes the elements $v$ in $\mathrm{vert}(t)$ with the action that permutes the order of the different $M(|\mathrm{in}(v)|)$ that appear in the tensor product. The free operad $\mathbb{T}(M)$ is then given, as an $\mathbb{S}$-module, by 
\[
\mathbb{T}(M)(n) \coloneqq \bigoplus_{t \in \mathrm{RT}_n} M(t)~, 
\]
where the sum ranges over all rooted trees $t$ of arity $n$. The free operad $\mathbb{T}(M)$ on an $\mathbb{S}$-module $M$ admits canonical composition maps, given by grafting rooted trees, and a canonical unit. We refer to \cite[Chapter 5, Section 6]{LodayVallette} for further details. 

\begin{example}
Let $M = (0,0, \mathbb{R}[\mathbb{S}_2].\mu, 0, 0, \cdots)$ be the $\mathbb{S}$-module given by the regular representation of $\mathbb{S}_2$ in arity $2$ and zero elsewhere. This corresponds to a single arity $2$ operation $\mu$ with no symmetries. The basis elements of the free operad $\mathbb{T}(M)$ are pictorially given as follows

\[
\includegraphics[scale=0.675]{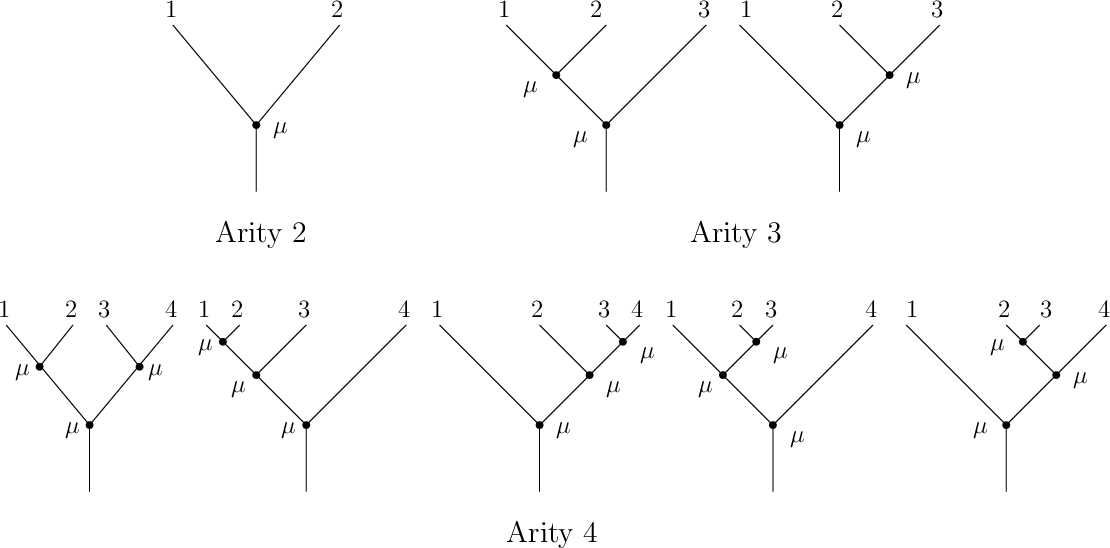}
\] 

in each corresponding arity, up to freely permuting the labels on the leaves of the rooted trees. In fact, in arity $n$, the vector space $\mathbb{T}(M)(n)$ admits a basis given by all binary rooted trees with $n$-leaves.
\end{example}

Let $\mathcal{P}$ be an operad. A \textit{presentation} $(M,R)$ of $\mathcal{P}$ amounts to the data of an $\mathbb{S}$-module $M$, called the \textit{generators}, a sub-$\mathbb{S}$-module $R$ of $\mathbb{T}(M)$, called the relations, and an isomorphism of operads
\[
\mathcal{P} \cong \mathbb{T}(M)/(R)~,
\]

where on the right we consider the free operad on $M$ quotiented by the \textit{operadic ideal} generated by the relations $R$. This essentially means that $\mathcal{P}$ is obtained by taking finite linear sums of all the possible iterated compositions of the operations in $M$ and then identifying terms along the relations specified in $R$. 

\begin{example}[The associative operad]\label{example: associative operad}
The associative operad $\mathcal{A}ss$ is given by the following presentation 
\[
\mathcal{A}ss \coloneqq \mathbb{T}(M_1)/(R_1)~, 
\]
where $M_1 = (0,0, \mathbb{R}[\mathbb{S}_2].\mu, 0, 0, \cdots)$ is the $\mathbb{S}$-module given by the regular representation of $\mathbb{S}_2$ in arity $2$ and zero elsewhere and where $R_1 \subset \mathbb{T}(M_1)(3)$, the generator of the ideal of relations, is the $\mathbb{S}$-module generated by 
\[
\mu \circ_1 \mu = \mu \circ_2 \mu~,
\]
which can depicted as 

\[
\includegraphics[scale=0.5]{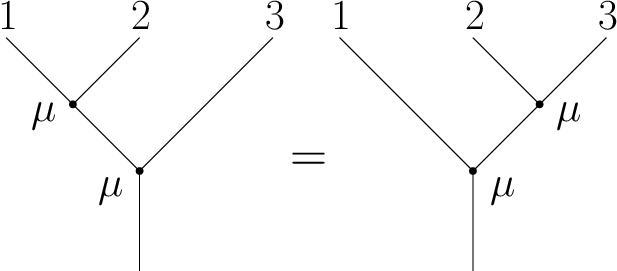}
\] 

It can be computed that $\mathcal{A}ss(n)$ is the regular representation $\mathbb{R}[\mathbb{S}_n]$ of $\mathbb{S}_n$ for all $n \geq 1$, which is of dimension $n!$ over $\mathbb{R}$. Algebras over this operad correspond to classical non-unital associative algebras. For more information on this operad, see \cite[Chapter 9, Section 9.1.3]{LodayVallette}. 
\end{example}

\begin{example}[The commutative operad]
The commutative operad $\mathcal{C}om$ is given by the following presentation 
\[
\mathcal{C}om \coloneqq \mathbb{T}(M_2)/(R_2)~, 
\]
where $M_2 = (0,0, \mathbb{R}.\nu, 0, 0, \cdots)$ is the $\mathbb{S}$-module given by the trivial representation of $\mathbb{S}_2$ in arity $2$ and zero elsewhere. This corresponds to an arity $2$ operad $\nu$ which is \textit{symmetric}. The generator of the ideal of relations $R_2 \subset \mathbb{T}(M_2)(3)$ is again given by 
\[
\mu \circ_1 \mu = \mu \circ_2 \mu~,
\]
which imposes the associativity condition on the product $\nu$. It can be computed that $\mathcal{C}om(n)$ is the trivial representation $\mathbb{R}$ of $\mathbb{S}_n$ for all $n \geq 1$, which is of dimension $1$ over $\mathbb{R}$. Algebras over this operad correspond to classical non-unital commutative algebras. Notice how the symmetry of the commutative product is encoded by the action of the symmetric groups and not by a supplementary relation. For more information on this operad, see \cite[Chapter 13, Section 13.1]{LodayVallette}. 
\end{example}

\begin{example}[The permutative operad]\label{example: permutative operad}
The permutative operad $\mathcal{P}erm$ is given by the following presentation 
\[
\mathcal{P}erm \coloneqq \mathbb{T}(M_3)/(R_3)~, 
\]
where $M_3 = (0,0, \mathbb{R}[\mathbb{S}_2].\rho, 0, 0, \cdots)$ is the $\mathbb{S}$-module given by the regular representation of $\mathbb{S}_2$ in arity $2$ and zero elsewhere and where $R_3 \subset \mathbb{T}(M_3)(3)$, the generator of the ideal of relations, is the $\mathbb{S}$-module generated by 
\[
\rho \circ_1 \rho = \rho \circ_2 \rho = \rho \circ_2 ((21) \star \rho)~, 
\]
which we call the \textit{(right) permutativity relation,} and which can depicted as 

\[
\includegraphics[scale=0.5]{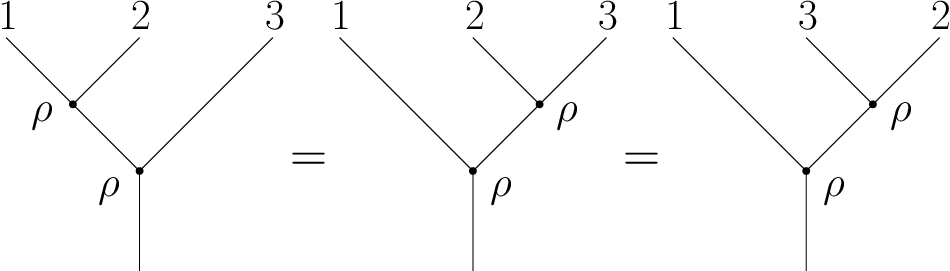}
\] 

It can be computed that $\mathcal{P}erm(n)$ is the standard representation $\mathbb{R}^n$ of $\mathbb{S}_n$ for all $n \geq 1$, which is of dimension $n$ over $\mathbb{R}$. Algebras over this operad correspond to a type of algebraic structure called permutative algebras, see \cite[Chapter 13, Section 13.4.6]{LodayVallette} for more details on this operad. 
\end{example}

\vspace{1.5pc}

\section{The operad of cooperative games}

\vspace{2pc}

We define a linear operad structure on the collection of all cooperative games. The partial composition of two games gives a game where the bottom game is rescaled and where the value of the top game, multiplied by the marginal contribution of the player into which it is inserted, is added. Furthermore, we compare our composition with all the previous notions of compositions of cooperative games present in the literature, such as \cite{ShapleyComposition} or \cite{OwenTensor}, and explain how to recover them as particular cases of our general construction. 

\subsection{The partial composition of cooperative games}\label{subsection: partial composition}
We define \textit{ab initio} partial composition maps for cooperative games and show that they form a linear operad structure. These composition maps can be understood as the universal bilinear composition maps that preserve unanimity games. Recall that we work with \textit{grounded} games, meaning that $v(\emptyset) = 0$.

\medskip

The set $\mathbb{G}(n)$ of all $n$-player games has a natural $\mathbb{R}$-vector space structure. Let $\Gamma_1 =(N,v_1)$ and $\Gamma_2 = (N,v_2)$ be two games, we define 
\[
\lambda \cdot \Gamma_1 \coloneqq (N, \lambda \cdot v_1) \quad \text{and} \quad \Gamma_1 + \Gamma_2 \coloneqq (N, v_1 + v_2)~,
\]
by multiplying and adding the coalition functions that define the games. It is a $2^n-1$-dimensional vector space, since we consider grounded games. It is canonically included in the set of non-necessarily grounded games $\mathbb{G}^\emptyset(n)$, which is a real vector space of dimension $2^n$. 

\medskip

A convenient base of these vector spaces has been identified by Shapley~\cite{Shapley1953}. It is formed of \emph{Dirac games}, one per coalition. The \(n\)-ary Dirac game associated with the coalition \(S\) is denoted by \(\delta_S\) and defined as the simple game satisfying \(\delta_S(T) = 1\) if and only if \(T = S\). In this base, the coalition functions of games are simply written \(v = \sum_{S \subseteq N} v_S \delta_S\). Shapley gave another base in the same paper, formed by \emph{unanimity games}. The unanimity game associated with the coalition \(S\) is denoted by \(u_S\) and defined as the simple game satisfying \(u_S(T) = 1\) if and only if \(T \supseteq S\). In this base, the coalition functions of games are written \(v = \sum_{S \subseteq N} \mu^v_S u_S\), with \(\mu^v\) being the M{\"o}bius inverse of \(v\). Notice that, contrarily to the Dirac games, the unanimity games are monotone. 

\begin{lemma}
The collection $\{\mathbb{G}(n)\}_{n \geq 0}$ has a natural $\mathbb{S}$-module structure, where the $\hspace{1pt}\mathbb{S}_n$-module on $\mathbb{G}(n)$ is constructed as follows. Let $\sigma$ be in $\mathbb{S}_n$ and let $\Gamma = (N,v)$ be a $n$-player game on the set $N = \{1,\cdots,n\}$. The game $\sigma \star \Gamma$ is defined on the set of players $\sigma \star N = \{\sigma(1),\cdots,\sigma(n)\}$ with the coalition function \( \sigma \star v \) defined, for all coalition \(S\), by
\[
\sigma \star v(S) \coloneqq v(\{\sigma^{-1}(i) \mid i \in S\} )~.
\]
\end{lemma}

\begin{proof}
It is straightforward to check that, for all $n \geq 0$, all $\sigma, \tau$ in $\mathbb{S}_n$ and all $\Gamma \in \mathbb{G}(n)$, we have 
\[
(\sigma \tau) \star \Gamma = \sigma \star (\tau \star \Gamma)~.
\]
\end{proof}

\begin{remark}
This $\mathbb{S}_n$-module structure on $\mathbb{G}(n)$ has been considered in \cite{jeuxrepdesgroupessymetriques2,jeuxrepdesgroupessymetriques1}; there the authors study the irreducible components of $\mathbb{G}(n)$ as a representation of $\mathbb{S}_n$, which can be described in terms of $k$\textit{-inessential games}, a generalization of additive games. 
\end{remark}

\textbf{The partial composition of totally ordered sets.} 
Our goal is to first define the partial composition of player sets and coalitions. Let $A = \{a_1, \cdots, a_n\}$ and $B = \{b_1, \cdots, b_m \}$ be two totally ordered sets. We define, for $i \in [n]$, the totally ordered set 
\[
A \diamond_i B \coloneqq \{a_1, \ldots, a_{i-1} \} \cup B \cup \{a_{i+1}, \ldots, a_n\}~, 
\]
where the order is the following: players in $B$ keep their internal order and are placed after $a_{i-1}$ and before $a_{i+1}$. Now let $S \subseteq A$ and $T \subseteq B$ be two coalitions and let $i \in [n]$. We define the partial composition of coalitions as 
\[
S \diamond_i T \coloneqq \begin{cases}
\big( S \cap \{a_1, \ldots, a_{i-1}\} \big) \cup T \cup \big( S \cap \{a_{i+1}, \ldots, a_n\} \big) & \text{if the} ~~ i\text{th player of } A \text{ is in } S, \\
S & \text{otherwise}, 
\end{cases}
\]
and where the total order on $S \diamond_i T$ is inherited from the obvious inclusion $S \diamond_i T \subseteq A \diamond_i B$. In other words, if the $i$th player of $A$ is in $S$, then it gets replaced by $T$ and, otherwise, we keep $S$, now seen as a coalition in $A \diamond_i B$.

\begin{definition}[Partial composition of games]\label{definition: partial composition of games}
Let $n, m$ be two positive integers. Let $\Gamma_A = (A, \alpha) \in \mathbb{G}(n)$ and $\Gamma_B = (B, \beta) \in \mathbb{G}(m)$ be two games and let $i \in A$. Their \textit{partial composition} $\Gamma_A \circ_i \Gamma_B$ is defined as follows: the player set is given by $A \diamond_i B = \left(A \setminus \{i\}\right) \cup B$, and its coalition function $\gamma$ is defined, for all $S \subseteq A \diamond_i B$, by
\[
\gamma(S) = \beta(B)\alpha(S_A) + \partial_i \alpha(S_A) \cdot \beta(S_B), 
\]
where $S_A = S \cap A$ and $S_B = S \cap B$. 
\end{definition}

\begin{remark}\label{rmk: meaning of the composition}
Let us explain the meaning of this formula from a game theoretical point of view. The idea behind is that the composition of the two games is given by the sum of their respective coalition functions on the composite set of players, where the second function is multiplied by the marginal contribution to the first game of the player that has been replaced. However, since both games need to be ``comparable'' in order for this sum to be meaningful, we need to rescale the first game by the total value of the second one. 
\end{remark}

\begin{notation}[Trivial game]
We  denote by $\mathbf{1}$ the unique normalized, grounded $1$-player game.
\end{notation}

\begin{theorem}\label{thm: gamers operad}
Let $\mathbb{G}$ be the $\mathbb{S}$-module of cooperative games. The partial composition maps
\[ 
\begin{tabular}{lccc}
    $\circ_i:$ & $\mathbb{G}(n) \times \mathbb{G}(m)$ & $\xrightarrow{\hspace{1cm}}$ & $\mathbb{G}(n + m - 1)$ \\[0.1cm]
    & $\left( \Gamma_1, \Gamma_2 \right)$ & $\xmapsto{\hspace{1cm}}$ & $\Gamma \coloneqq \Gamma_1 \circ_i \Gamma_2$
\end{tabular} 
\]
and the unit map which sends $\mathbf{1}$ to the trivial game endow $\mathbb{G}$ with a unital partial operad structure. 
\end{theorem}

\proof 
It is immediate to check that the maps $\{\circ_i\}_i$ are bilinear and that the trivial game $\mathbf{1}$ satisfies the unit axiom with respect to these partial compositions. 

\medskip

Let us check the parallel and the sequential axioms. Let $p, q$ and $r$ be three positive integers, let $\Gamma_p = (A, \alpha) \in \mathbb{G}(p)$, $\Gamma_q = (B, \beta) \in \mathbb{G}(q)$ and $\Gamma_r = (C, \gamma) \in \mathbb{G}(r)$ be three games with their player sets being
\[
A = \{a_1, \ldots, a_p\}, \qquad B = \{b_1, \ldots, b_q\} \qquad \text{and} \qquad C = \{c_1, \ldots, c_r\}. 
\]
It is straightforward to check that the sequential and the parallel axioms hold on the player sets. Let us check the coalition functions in both cases. To improve clarity, we will explicitly write down the "name" of the player with respect to whom we take the derivative instead of referring to it by its position in the player set. 

\medskip

$\triangleright$ \emph{Sequential axiom:} 
We have to check that $\left( \Gamma_p \circ_{i} \Gamma_q \right) \circ_{i+j-1} \Gamma_r$ and $\Gamma_p \circ_i \left( \Gamma_q \circ_j \Gamma_r \right)$ have the same coalition functions. 

\medskip

By definition, the coalition function $\delta$ of $\Gamma_p \circ_i \Gamma_q$ is given, for any $S \subseteq D = A \diamond_i B$, by 
\[
\delta(S) = \beta(B)\alpha(S_A) + \partial_{a_i} \alpha(S_A) \cdot \beta(S_B). 
\]

Hence, the coalition function of $\left(\Gamma_p \circ_i \Gamma_q\right) \circ_{i+j-1} \Gamma_r$ is given by 
\[ 
\begin{aligned} 
&\gamma(C)\delta(S_D) + \partial_{b_j} \delta(S_D) \cdot \gamma(S_C) \\
& \hspace{80pt} = \gamma(C)\beta(B)\alpha(S_A) + \gamma(C)\partial_{a_i} \alpha(S_A) \beta(S_B) + \partial_{a_i} \alpha(S_A) \partial_{b_j} \beta(S_B) \gamma(S_C) \\
& \hspace{80pt} = \gamma(C)\beta(B)\alpha(S_A) + \partial_{a_i} \alpha(S_A) \cdot \Big( \gamma(C)\beta(S_B) + \partial_{b_j} \beta(S_B) \cdot \gamma(S_C) \Big)~,
\end{aligned} 
\]
for any subset $S$ of $(A \diamond_i B) \diamond_{i+j-1} C$, using the Lemma \ref{lemma: dérivée de la composée} and the fact that the player $b_j$ is in $B$. On the other hand, the coalition function $\varepsilon$ of $\left( \Gamma_q \circ_j \Gamma_r \right)$ is given by 
\[
\varepsilon(S) = \gamma(C)\beta(S_B) + \partial_{b_j} \beta(S_B) \cdot \gamma(S_C), 
\]
for every subset $S \subseteq E = B \diamond_j C$. The coalition function of $\Gamma_p \circ_i \left( \Gamma_q \circ_j \Gamma_r \right)$ is given by 
\begin{multline*}
\gamma(C)\beta(B)\alpha(S_A) + \partial_{a_i} \alpha(S_A) \cdot \varepsilon(S_E) \\ = \gamma(C)\beta(B)\alpha(S_A) + \partial_{a_i} \alpha(S_A) \cdot \Big( \gamma(C)\beta(S_B) + \partial_{b_j} \beta(S_B) \cdot \gamma(S_C) \Big),  
\end{multline*}
since it can be computed that $\varepsilon(B \diamond_j C) = \gamma(C)\beta(B)$, and thus the sequential axiom holds.

\medskip 

$\triangleright$ \emph{Parallel axiom:} Let $i$ and $k$ be two positive integers such that $i < k \leq p$. Let $\Gamma_p, \Gamma_q$ and $\Gamma_r$ be three games as previously defined. We want to check that the coalition functions of $\left(\Gamma_p \circ_i \Gamma_q \right) \circ_{k+q-1} \Gamma_r$ and of $\left( \Gamma_p \circ_k \Gamma_r \right) \circ_i \Gamma_q$ coincide. 

\medskip

Again, the coalition function $\delta$ of $\Gamma_p \circ_i \Gamma_q$ is given, for any $S \subseteq D = A \diamond_i B$, by 
\[
\delta(S) = \beta(B)\alpha(S_A) + \partial_{a_i} \alpha(S_A) \cdot \beta(S_B). 
\]
Thus, for any subset $S$ of $(A \diamond_i B) \diamond_{k+q-1} C$, the coalition function of $\left(\Gamma_p \circ_i \Gamma_q \right) \circ_{k+q-1} \Gamma_r$ is 
\[
\gamma(C)\delta(S_D) + \partial_{a_k} \delta(S_D) \cdot \gamma(S_C). 
\]
By Lemma \ref{lemma: dérivée de la composée}, we have that
\[ \begin{aligned}
\partial_{a_k} \delta(S_D) & = \beta(B)\partial_{a_k} \alpha(S_A) + \partial_{a_k} \partial_{a_i} \alpha(S_A) \cdot \beta(S_B).
\end{aligned} 
\]
since $a_k$ is in $A \setminus \{i\}$. Thus, the coalition function of $\left( \Gamma_p \circ_i \Gamma_q \right) \circ_{k+q-1} \Gamma_r$ can be rewritten as 
\[ \begin{aligned} 
\gamma(C)\delta(S_D) + \partial_{a_k} & \delta(S_D) \cdot \gamma(S_C) \\ & = \gamma(C)\delta(S_D) + \left( \beta(B)\partial_{a_k} \alpha(S_A) + \partial_{a_k}\partial_{a_i} \alpha(S_A) \cdot \beta(S_B) \right)\cdot \gamma(S_C) \\
& = \gamma(C)\beta(B)\alpha(S_A) + \gamma(C)\partial_{a_i} \alpha(S_A) \cdot \beta(S_B) + \beta(B)\partial_{a_k} \alpha(S_A) \cdot \gamma(S_C) \\ 
& \qquad \qquad \qquad + \partial_{a_k}\partial_{a_i}\alpha(S_A) \cdot \beta(S_B) \cdot \gamma(S_C). 
\end{aligned} \]
In a similar manner, let $\phi$ be the coalition function of $\Gamma_p \circ_k \Gamma_r$, which is given by 
\[
\phi(S) = \gamma(C)\alpha(S_A) + \partial_{a_k} \alpha(S_A) \cdot \gamma(S_C),
\]
for any subset $S \subseteq F = A \diamond_k C$. Thus, the coalition function of $\left( \Gamma_p \circ_k \Gamma_r \right) \circ_i \Gamma_q$ is given, for each coalition in $(A \diamond_k C) \diamond_i B$, by 
\[
\beta(B)\phi(S_F) + \partial_{a_i} \phi(S_F) \cdot \beta(S_B).
\]
By Lemma \ref{lemma: dérivée de la composée}, the derivative of $\phi$ is given by  
\[ \begin{aligned} 
\partial_{a_i} \phi(S_F) & = \gamma(C)\partial_{a_i} \alpha(S_A) + \partial_{a_i} \partial_{a_k} \alpha(S_A ) \cdot \gamma(S_C)
\end{aligned} \]
since the $a_i$ player is in $A \setminus \{k\}.$ Hence, the coalition function of $\left( \Gamma_p \circ_k \Gamma_r \right) \circ_i \Gamma_q$, for any subset $S$, is given by 
\[ \begin{aligned} 
\beta(B)\phi(S_F) + \partial_{a_i} \phi(S_F) \cdot \beta(S_B) & = \beta(B)\phi(S_F) + \left( \partial_{a_i} \alpha(S_A) + \partial_{a_i} \partial_{a_k} \alpha(S_A) \cdot \gamma(S_C) \right) \cdot \beta(S_B) \\
& = \gamma(C)\beta(B)\alpha(S_A) + \beta(B)\partial_{a_k} \alpha(S_A) \cdot \gamma(S_C) ~\\
& \qquad + \gamma(C) \partial_{a_i} \alpha(S_A) \cdot \beta(S_B) + \partial_{a_i} \partial_{a_k} \alpha(S_A) \cdot \gamma(S_C) \cdot \beta(S_B),
\end{aligned} \]
which concludes the proof of the sequential axiom. 

\medskip

$\triangleright$ \emph{Compatibility with the action of the symmetric groups:} Let $p$ and $q$ be two positive integers, let $\Gamma_1 = (A, \alpha) \in \mathbb{G}(p)$ and $\Gamma_2 = (B, \beta) \in \mathbb{G}(q)$ be two games.

\medskip 

Let $\sigma \in \mathbb{S}_q$ and let $\sigma'$ be the unique permutation in $\mathbb{S}_{p+q-1}$ which acts as $\sigma$ on $\llbracket i, i+q-1 \rrbracket$ and the identity elsewhere. 

\medskip

The composition of $\sigma \star \Gamma_2$ at the $i$-th player of $\Gamma_1$ is the game defined on the player set 
\[
C = \{a_1, \ldots, a_{i-1}, b_{\sigma(1)}, \ldots, b_{\sigma(q)}, a_{i+1}, \ldots, a_p\},
\]
and whose coalition function is, for all $S \subseteq C$, given by 
\[
\beta(B)\alpha(S_A) + \partial_{a_i} \alpha(S_A) \cdot (\sigma \star \beta)(S) = \beta(B)\alpha(S_A) + \partial_{a_i} \alpha(S_A) \cdot \beta(\{\sigma^{-1}(i) \mid i \in S_B\}). 
\]
On the other hand, let us compute $\sigma' \star (\Gamma_1 \circ_i \Gamma_2)$. Applying $\sigma'$ on the player set 
\[
\{a_1, \ldots, a_{i-1}, b_1, \ldots, b_q, a_{i+1}, \ldots, a_p\}
\]
gives the player set $C$ above. And the coalition function is given, for any $S \subseteq C$, by 
\[ \begin{aligned} 
(\sigma' \star (\beta(B)\alpha + \partial_{a_i} \alpha \cdot \beta))(S) & = \beta(B)(\sigma' \star \alpha)(S_A) + \partial_{a_i} (\sigma' \star \alpha)(S_A) \cdot (\sigma' \star \beta)(S_B) \\
& = \beta(B)\alpha(S_A) + \partial_{a_i} \alpha(S_A) \cdot \beta(\{\sigma^{-1}(i) \mid i \in S_B\}),
\end{aligned} \]
hence both games coincide. Checking the other compatibility axiom is entirely analogous, and amounts to a straightforward computation. 
\endproof 

\subsection{A characterization of the partial composition of cooperative games}
We show that the partial composition of unanimity games corresponds exactly with the partial composition of their winning coalition. 

\begin{proposition}\label{Proposition: universal property of the partial composition of games}
The partial composition map
\[
\circ_i: \mathbb{G}(n) \times \mathbb{G}(m) \longrightarrow \mathbb{G}(n + m -1)
\]
of Definition \ref{definition: partial composition of games} is the unique bilinear map on the basis of unanimity games such that
\[
u_S \circ_i u_T = u_{S \diamond_i T}~,
\]
holds for any $S \subseteq N$ and $T \subseteq M$. 
\end{proposition}

\begin{proof}
Such bilinear map is fully characterized by the assignment of the basis elements, so there exists a unique bilinear map that satisfies the above-mentioned equality. Let us check that the maps $\{\circ_i\}_i$ satisfy this equality. We have 
    \[
    u_S \circ_i u_T(K) = u_T(M)u_S(K_N) + \partial_i u_S(K_N)u_T(K_M) = u_S(K_N) + \partial_i u_S(K_N)u_T(K_M)~,
    \]
for any $K \subseteq N \diamond_i M$, since $u_T(M)=1$ for any $T \subseteq M$. 
\begin{enumerate}
\item If $i \notin S$, then
    \[
    \partial_i u_S(K_N) = u_S(K_N \cup \{i\}) - u_S(K_N) = 0~,
    \]
    since either both terms are $1$ or both are $0$. Thus we have $u_S \circ_i u_T = u_S = u_{S \diamond_i T}$. 
    
\item If $i \in S$, then $K_N$ can not contain $S$ and thus 
	\[
    u_S \circ_i u_T(K) = u_S(K_N \cup \{i\})u_T(K_M)~,
    \]
	which is $1$ if and only if $S \diamond_i T \subseteq K$. Therefore $u_S \circ_i u_T$ is also equal to $u_{S \diamond_i T}$. 
\end{enumerate} 
\end{proof}

\begin{remark}
By definition, in a unanimity game, a coalition is winning if and only if it includes a set of veto players, without which winning is impossible. Hence, everyone of them is required to participate. Let us now consider that a veto player \(i\) was, in fact, the representative of a set of players, distinct from the player considered originally. Additionally, let assume that the decision process of this set of players is also modeled by a unanimity game. A coalition of the composite game is then winning only if it contains all the veto players, which are now of two kinds. On one hand, there are the original veto players, still active, and on the other hand, the veto players that choose the action of player \(i\). Hence, the new set of veto players is indeed the union of the veto players of the component game which act for \(i\), and the veto players of the original game. On the other hand, if the player $i$ was not a veto player, the winning coalitions remain the same. 
\end{remark}

\subsection{The partial composition of cooperative games in the basis of unanimity games.}\label{subsection: unanimity basis} Let us denote by $\mathbb{G}^{u}(n)$ the vector space of $n$-players cooperative games endowed with the basis given by unanimity games $\{u_S\}_S$. Let 
\[
f = \sum_{S \subseteq [n]} \lambda_S u_S \quad \text{and} \quad g = \sum_{T \subseteq [m]} \rho_T u_T~,
\]
where $f$ is an element in $\mathbb{G}^{u}(n)$ and $g$ is an element in $\mathbb{G}^{u}(m)$. The partial composition of cooperative games can be written down in this basis as follows 
\[
f \circ_i^{u} g = \left(\sum_{T \subseteq [m]} \rho_T\right)\sum_{\substack{S \subseteq [n] \\ S \not \ni i}} \lambda_S u_S + \sum_{\substack{S \subseteq [n] \\ S \ni i}} \sum_{T \subseteq [m]} \lambda_{S}\rho_T u_{S \circ_i T}~,
\]
using bilinearity and Proposition \ref{Proposition: universal property of the partial composition of games}.

\begin{proposition}\label{prop: Mobius changes the basis}
The collection of maps 
\[ \begin{aligned} 
\mu(n): \mathbb{G}(n) & \xrightarrow{\hspace{1.5cm}} \mathbb{G}^u(n) \\
v & \xmapsto{\hspace{1.5cm}} \mu(n)(v) \coloneqq \sum_{S \subseteq [n]} \mu^v(s) u_S, 
\end{aligned} \]
where $\mu^v$ is the M{\"o}bius inverse of $v$, is an isomorphism of $\mathbb{S}$-modules and defines a change of basis on the linear operad of cooperative games. 
\end{proposition}

\begin{proof}
The M{\"o}bius inverse is a linear isomorphism from $\mathbb{G}(n)$ to $\mathbb{G}^{u}(n)$ for every $n \geq 0$, and maps $\delta_T$ to $u_T$ for every $T \subseteq N$. Let us check that it commutes with the action of the symmetric groups. Let $v$ be a $n$-player game, and let $\sigma \in \mathbb{S}_n$. The M{\"o}bius transform of $v_\sigma$ is given by 
\[
\mu^{v_\sigma}(S) = \sum_{T \subseteq S} (-1)^{\lvert S \setminus T \rvert} v(\sigma^{-1}(T)) = \sum_{T \subseteq \sigma^{-1}(S)} (-1)^{\lvert S \setminus T \rvert} v(T) = \mu^v(\sigma^{-1}(S)). 
\]
Hence, the image of $v_\sigma$ by $\mu(n)$ is
\[ 
\mu(n)(v_\sigma) = \sum_{S \subseteq N} \mu^{v_\sigma}(S) u_S = \sum_{S \subseteq N} \mu^v(\sigma^{-1}(S)) u_S = \sum_{S \subseteq N} \mu^v(S) S_{\sigma(S)} ~. 
\]
Finally, using Proposition \ref{Proposition: universal property of the partial composition of games} and the linearity of the M{\"o}bius inverse, it can be checked that
\[
\mu(n+m-1)(v \circ_i w) = \mu(n)(v) \circ_i^{u} \mu(m)(w)~, 
\]
for any two games $\Gamma_1 = (N,v)$ in $\mathbb{G}(n)$ and $\Gamma_2 = (M,w)$ in $\mathbb{G}(m)$. 
\end{proof}

\begin{remark}[The Zeta transform]
The inverse map 
\[ \begin{aligned} 
\zeta(n): \mathbb{G}^u(n) & \xrightarrow{\hspace{1.5cm}} \mathbb{G}(n) \\
f = \sum_{S \subseteq [n]} \lambda_S u_S & \xmapsto{\hspace{1.5cm}} \zeta(n)(f) \coloneqq \sum_{S \subseteq [n]} \zeta^f(S) \delta_S, 
\end{aligned} \]
is given by the Zeta transform, whose value is given, for any $S \subseteq [n]$, by 
\[
\zeta^f(S) = \sum_{T \subseteq S} \lambda_T~. 
\]
\end{remark}

\subsection{Comparison with previous notions of products, sums and compositions for cooperative games}
This subsection is devoted to comparing our notion of partial composition of cooperative games with previous definitions present in the literature: the sum and the product considered by Shapley in \cite{ShapleyCompoundI}, the general composition of monotone simple games defined in \cite{ShapleyCompoundII} (see also \cite{ShapleyComposition}), and its generalization by Owen called the tensor composition in \cite{OwenTensor,OwenMultilinear}. 

\subsubsection{Sum, product and compositions of simple games}\label{subsub: sum and product of simple games} We show that Shapley's notion of sum and product of simple monotone games is recovered by composing two simple monotone games into the $2$-players bargaining game or the $2$-player game which the sum of the two dictator games. 

\begin{definition}[Simple game]
A game $\Gamma = (N, v)$ is \emph{simple} if \(v(S) \in \{0, 1\}\) for all \(S \subseteq N\). 
\end{definition}

In a simple game, a coalition is completely winning, or completely losing. It is thus totally characterized by the set of winning coalitions. If furthermore, if the simple game is monotone, then every superset of a winning coalition is winning, and every subcoalition of a losing coalition is losing. Thus it can be fully characterized by its set of \textit{minimal} winning coalitions. In fact, the data of a simple monotone game is equivalent to the data of a \textit{clutter}, see \cite{billera1971composition}. 

\begin{example}
Simple games are often used to model political games, where power play between coalitions are undergoing, and the outcome is Boolean, like whether a bill pass, someone get elected, etc. For instance, a democratic election where the winner is designated by simple majority is an example of a simple monotone cooperative game, as for example the second round of the French presidential election, which is a simple monotone {\raise.17ex\hbox{$\scriptstyle\sim$}}$50$-million players game. 
\end{example}

\begin{definition}[Shapley's sum and product]
Let $\Gamma_1 = (N_1,W_1)$ and $\Gamma_1 = (N_2,W_2)$ be two simple games, where $W_j$ are the set of winning coalitions. 
\begin{enumerate}
\item The \textit{sum} $\Gamma_1 \oplus \Gamma_2$ is the simple game defined on the set of players $N_1 \times N_2$ where the set of winning coalitions is given by 
\[
W_{\Gamma_1 \oplus \Gamma_2} \coloneqq \left\{S \subseteq N_1 \times N_2 ~ | ~ S \cap N_1 \in W_1 \quad \text{or} \quad S \cap N_2 \in W_2 \right\}~. 
\]
\item The \textit{product} $\Gamma_1 \otimes \Gamma_2$ is the simple game defined on the set of players $N_1 \otimes N_2$ where the set of winning coalitions is given by 
\[
W_{\Gamma_1 \otimes \Gamma_2} \coloneqq \left\{S \subseteq N_1 \times N_2~ | ~ S \cap N_1 \in W_1 \quad \text{and} \quad S \cap N_2 \in W_2 \right\}~. 
\]
\end{enumerate}
\end{definition}

Let us denote by $B$ the $2$-player bargaining game, which is $1$ on $\{1,2\}$ and zero elsewhere. Let us denote by $d_1$ (respectively $d_2$) the $2$-player dictator game where $1$ is a dictator, that is, the game defined by $v(S) = 1$ if $1 \in S$. We denote by $B^*$ the $2$-players game $d_1 + d_2 -B$, given by adding the two $2$-players dictator games and subtracting the bargaining game $B$. It is the dual of the bargaining game $B$ in the sense of Definition \ref{definition: dual game}. 

\begin{proposition}
Let $\Gamma_1 = (N_1,v_1)$ and $\Gamma_1 = (N_2,v_2)$ be two simple games. 

\begin{enumerate}

\item Their sum is given by the composition 
\[
\Gamma_1 \oplus \Gamma_2 = (B^* \circ_1 \Gamma_1) \circ_2 \Gamma_2~,
\]
where $B^*$ is the dual of $2$-players the bargaining game. 

\item Their product is given by the composition 
\[
\Gamma_1 \otimes \Gamma_2 = (B \circ_1 \Gamma_1) \circ_2 \Gamma_2~,
\]
where $B$ is the $2$-player bargaining game. 

\end{enumerate}
\end{proposition}

\begin{proof}
Let us start with product, as it is somewhat easier. The general formula for the composition gives that the coalition function of the term in the right is given by 
\begin{align*}
(B \circ_1 v_1) \circ_2 v_2(S) & = v_1(N_1)v_2(N_2)B(S_{\emptyset}) + v_1(N_1)\partial_1 B(S_{\emptyset}) \cdot v_2(S_{N_2}) \\
& \qquad + v_2(N_2) \partial_2B(S_{\emptyset}) \cdot v_1(S_{N_1}) + \partial_{12} B(S_{\emptyset}) \cdot v_1(S_{N_1}) \cdot v_2(S_{N_2})~,
\end{align*}
for any subset $S \subseteq N_1 \cup N_2$, where $S_{\emptyset} = S \cap \emptyset = \emptyset$. Notice first that $v_1(N_1) = v_2(N_2) = 1$ since these games are simple. Since $\partial_1B(\emptyset) = \partial_2B(\emptyset) = 0$ and that $\partial_{12} B(\emptyset) = 1$, we get that 
\[
(B \circ_1 v_1)(S) \circ_2 v_2(S) = v_1(S_{N_1}) \cdot v_2(S_{N_2})~,
\]
which is $1$ only when $N_1$ is a winning coalition of $\Gamma_1$ and $N_2$ a winning coalition of $\Gamma_2$. 

\medskip

Similar arguments can be applied to recover the sum. One computes that $\partial_1 B^*(\emptyset) = \partial_2 B^*(\emptyset) = 1$ and that $\partial_{12}B^*(\emptyset) = -1$, thus we have that
\[
(B^* \circ_1 v_1) \circ_2 v_2(S) = v_1(N_1) + v_2(N_2) - v_1(S_{N_1}) \cdot v_2(S_{N_2})~,
\]
which is $1$ when either $N_1$ is a winning coalition in $\Gamma_1$ or $N_2$ a winning coalition of $\Gamma_2$. 
\end{proof}

The sum and the product are both generalized by the following construction introduced by Shapley in \cite{ShapleyCompoundII}, which is defined on all simple monotone games. 

\begin{definition}[Shapley's composition of simple games]
Let \(\Gamma_0 = (N_0, \mathcal{W}_0)\) be a simple monotone game and let \(\Gamma_i = (N_i, \mathcal{W}_i)\) be simple games for all \(i \in N_0\). Set \(N = \sqcup_{i \in N_0} N_i\). For any coalition \(S \subseteq N\), we define \(K(S)\) to be the set of players of \(N_0\) that are controlled by \(S\), i.e., 
\[
K(S) \coloneqq \{i \in N_0 \mid S \cap N_i \in \mathcal{W}_i\}. 
\]
The \textit{compound game}, denoted by $\Gamma_0[\Gamma_1, \cdots, \Gamma_n ]$ is defined on the set of players $N$ by the following set of winning coalitions
\[
\mathcal{W} \coloneqq \left\{S \subseteq N ~~|~~ K(S) \in \mathcal{W}_0\right\}. 
\]
\end{definition}

\medskip

Recall that, by Proposition \ref{prop: compose total = composee partielle}, the data of an operad defined in terms of partial compositions is equivalent to the data of an operad defined in terms of total or two-levelled compositions. Informally speaking, the total composition of an operation $\mu_k$ with $k$ inputs with $k$ different operations $\{\mu_{i_j}\}$ with $i_j$ inputs, for $j \in [k]$, is obtained by first plugging $\mu_{i_1}$ in the first input of $\mu_k$, then $\mu_{i_2}$ in the second input of the resulting operation, and so on, until all the inputs of the original operation $\mu_k$ have been filled. Let us denote by $\gamma_{\mathbb{G}}(k;i_1,\cdots,i_k)$ the total composition maps associated to the operad structure defined in Theorem \ref{thm: gamers operad}. 

\begin{proposition}
Let $\Gamma_0 = (N_0, \mathcal{W}_0)$ be a simple game and let $\Gamma_i = (N_i, \mathcal{W}_i)$ be simple games for all \(i \in N_0\). We have the equality:
\[
\Gamma_0[\Gamma_1, \cdots, \Gamma_n] = \gamma_{\mathbb{G}}(n_0;n_1,\cdots,n_{n_0})(\Gamma_0;\Gamma_1,\cdots,\Gamma_n)~,
\]
where $\gamma_{\mathbb{G}}$ denotes the total composition map of the linear operad of cooperative games. 
\end{proposition}

\begin{proof}
Follows from Theorem \ref{thm: comparison with Owen's tensor composition} which compares the total composition with the tensor composition of Owen in \cite[pp. 308-309]{OwenTensor}, since the composition of simple monotone games is a particular case of this. 
\end{proof}

\begin{remark}
This set theorical operad structure is mentioned, as an example, in \cite[Section 4.4.2]{Mendez15}. 
\end{remark}

\subsubsection{Owen's tensor composition} We compare the tensor composition defined by Owen in \cite{OwenTensor} with the total composition of our operad of cooperative games. We also explain how this composition agrees with the usual composition of multilinear polynomials in the normalized case. 

\begin{definition}[\cite{OwenTensor}]\label{def: tensor composition}
Let $\Gamma_0 = (N_0, v_0)$ be a non-negative game and let $\Gamma_i = (N_i, v_i)$, for $i \in N_0 = \{1, \ldots, n\}$ be normalized, non-negative games. The \emph{tensor composition} of the \emph{components} $\{\Gamma_i\}_{i \in N_0}$, with the \emph{quotient} $\Gamma_0$, is defined to be the game
\[
\Gamma \equiv (N, v) = \Gamma_0 \big[\Gamma_1, \ldots, \Gamma_n],
\]
where $N = N_1 \cup \ldots \cup N_n$ and for every $S \subseteq N$,
\[
v(S) = \sum_{T \subseteq N_0} \left( \prod_{i \in T} v_i(S_i) \prod_{i \not \in T} \big( 1 - v_i(S_i) \big) \right) v_0(T),
\]
where $S_i = S \cap N_i$. 
\end{definition}

\begin{theorem}\label{thm: comparison with Owen's tensor composition}
Let $\Gamma_0 = (N_0, v_0)$ be a non-negative game and let $\Gamma_i = (N_i, v_i)$, for $i \in N_0 = \{1, \ldots, n\}$ be normalized, non-negative games. The tensor composition is given by the total composition of our operad, meaning that 
\[
\Gamma_0 \big[\Gamma_1, \ldots, \Gamma_n] = \gamma_{\mathbb{G}}(n;n_1,\cdots,n_{n_0})(\Gamma_0;\Gamma_1,\cdots,\Gamma_n)~,
\]
where $\gamma_{\mathbb{G}}$ denotes the total composition map of the linear operad of cooperative games. 
\end{theorem}

\begin{proof}
Let us begin the proof with the following observation. Let $\Gamma_1 = (N_i, v_i)$ and $\Gamma_2 = (N_i, v_i)$ be two games. If $\Gamma_2$ is normalized, then the partial composition formula simplifies into
\[
v_1 \circ_i v_2(S) = v_1(S_{N_1}) + \partial_i v_1(S_{N_1})\cdot v_2(N_2)~, 
\]
for any player $i \in N_1$ and any subset $S \subseteq N_1 \circ_1 N_2$. Using this formula, it is tedious but straightforward to compute that 
\begin{align*}
\gamma_{\mathbb{G}}(n;n_1,\cdots,n_{n_0})(\Gamma_0;\Gamma_1,\cdots,\Gamma_n) &= (((\Gamma_0 \circ_1 \Gamma_1) \circ_2 \Gamma_2) \cdots \circ_n \Gamma_n)
\end{align*}
is indeed given by the formula of Definition \ref{def: tensor composition}. Alternatively, this result also follows directly from Theorem \ref{th: stability-by-comp} and Proposition \ref{prop: composition of multilinear polynomials}, whose proof is simpler. 
\end{proof}

\subsubsection{Composition of multilinear polynomials and the Möbius transform}
Let $\{e^i\}_{i \in N}$ denote the canonical basis of $\mathbb{R}^N$, and write $e^S$ for $\sum_{i \in S} e^i$. By identifying the coalition $S$ to the vertex of the hypercube $e^S$, the domain of $v$ is now a finite subset of $\mathbb{R}^N$, which we want to extend to the whole hypercube $[0, 1]^N$, such that the extension $f$ of $v$ is continuous. There exist several way of doing so, in particular using the Choquet integral, also called the Lovasz extension. The \emph{multilinear extension} extension was constructed by Owen in \cite{OwenMultilinear}.

\begin{definition}[Multilinear extension]
    Let $\Gamma = (N, v)$ be a game. The \emph{multilinear extension} $f_v$ of $\Gamma$ is given by 
    \[ 
    \begin{tabular}{lccl}
    $f_v:$ & $[0, 1]^N$ & $\xrightarrow{\hspace{1cm}}$ & $\mathbb{R}$ \\[0.1cm]
    & $x$ & $\xmapsto{\hspace{1cm}}$ & $\displaystyle f(x) = \sum_{S \in \mathcal{P}(N)} \left( \prod_{i \in S} x_i \prod_{i \not \in S} (1-x_i) \right) v(S)$
    \end{tabular} 
    \]
    It satisfies that, for all $S \in \mathcal{P}(N)$, $f(e^S) = v(S)$. 
\end{definition}

In \cite[Last Theorem]{OwenMultilinear}, Owen showed that this multilinear extension is unique. If we denote $x_S = \prod_{i \in S} x_i$, then it follows from this uniqueness that for any game $\Gamma = (N, v)$ and for all $x \in [0, 1]^N$, the multilinear extension is given by
    \[
    f_v(x) = \sum_{S \in \mathcal{P}(N)} \mu^v \hspace{-1pt} (S) \hspace{1pt} x_S.
    \]
So it is a polynomial function in variables $x_1,\cdots, x_n$, where the coefficient of $x_S$ is given by the Möbius transform of the game $\mu^v$ evaluated at $S$. The tensor composition of games and the usual composition of functions are related in the following way. 

\begin{theorem}[{\cite[Theorem 6]{OwenMultilinear}}]\label{th: stability-by-comp}
Let $\Gamma_0 = (N_0, v_0)$ be a non-negative game and let $\Gamma_i = (N_i, v_i)$, for $i \in N_0 = \{1, \ldots, n\}$ be normalized, non-negative games. Denote by $f$ the multilinear extension of $\Gamma_0$, and by $g_i$ the multilinear extension of $\Gamma_i$. The multilinear extension of the tensor composition $\Gamma_0[\Gamma_1, \ldots, \Gamma_n]$ is given by the function composition $f(g_1, \ldots, g_n)$.
\end{theorem}

Multilinear polynomials without constants in $n$-variables form a vector space of dimension $2^n -1$, as any such multilinear polynomial $f$ can be written as 
\[
f = \sum_{S \subseteq N} \lambda_S x_S~,  
\]
where $x_S = \prod_{i \in S} x_i$. The partial composition $f \circ_i g$ two multilinear polynomials $f, g$ is simply given by replacing all the instances of $x_i$ in $f$ by $g$. 

\begin{proposition}\label{prop: composition of multilinear polynomials}
The partial composition of the operad $\mathbb{G}^{u}(n)$ (in the basis of unanimity games) agrees with the usual composition of functions for normalized games, under the identification of $u_S$ with $x_S$. 
\end{proposition}

\begin{proof}
Recall from Subsection \ref{subsection: unanimity basis} that the partial composition of 
\[
f = \sum_{S \subseteq [n]} \lambda_S u_S \quad \text{and} \quad g = \sum_{T \subseteq [m]} \rho_T u_T~,
\]
can be written down in the basis of unanimity games as 
\[
f \circ_i^{u} g = \left(\sum_{T \subseteq [m]} \rho_T\right)\sum_{\substack{S \subseteq [n] \\ S \not \ni i}} \lambda_S u_S + \sum_{\substack{S \subseteq [n] \\ S \ni i}} \sum_{T \subseteq [m]} \lambda_{S}\rho_T u_{S \circ_i T}~.
\]
Now observe that a game $\Gamma = (N, v)$ is normalized, meaning $v(N) = 1$, if and only if 
\[
\sum_{S \subset N} \mu^v \hspace{-1pt} (S) = 1~, 
\]
So the partial composition of normalized games, in the basis of unanimity games, is given by 
\[
f \circ_i^{u} g = \sum_{\substack{S \subseteq [n] \\ S \not \ni i}} \lambda_S u_S + \sum_{\substack{S \subseteq [n] \\ S \ni i}} \sum_{T \subseteq [m]} \lambda_{S}\rho_T u_{S \circ_i T}~, 
\]
which under the identification of $u_S$ with $x_S$, exactly corresponds with the usual composition of polynomial functions.
\end{proof}

\begin{remark}
Notice that the partial composition of multilinear polynomials is \textit{not} itself bilinear unless the normalized condition is imposed.
\end{remark}

\vspace{1.5pc}

\section{An explicit presentation of the operad of cooperative games}

\vspace{2pc}

We recall the definition of the operad that encodes commutative trialgebras and show that it is canonically isomorphic to the operad of cooperative games via the Möbius transform. This gives us a presentation in terms of generators and relations of the operad of cooperative games, and thus allows us to show that any cooperative game is a sum of iterated compositions of the $2$-players dictator games and the $2$-players bargaining game.

\subsection{The operad of commutative trialgebras and its presentation}
In this subsection, we recall the notion of a commutative trialgebra, introduced by Vallette in \cite{Vallettepartition}, and give an explicit presentation of the operad that encodes this structure. 

\begin{definition}[Commutative trialgebra]\label{def: commutative trialgebra}
A \textit{commutative trialgebra} amounts to the data of an $\mathbb{R}$-vector space $A$ equipped with two binary operations 
\[
\mu: A \otimes A \longrightarrow A \quad \text{and} \quad \nu: A \otimes A \longrightarrow A~,
\]
which satisfy the following relations. 

\begin{enumerate}
\item The operation $\mu$ is symmetric and associative, meaning $(A,\mu)$ is a commutative $\mathbb{R}$-algebra.

\item The operation $\nu$ satisfies the following relation 
\[
\nu(\nu(a,b),c) = \nu(a,\nu(b,c)) = \nu(a,\nu(c,b))~,
\]
for any $a,b,c$ in $A$, which means that the associator of $\nu$ is \textit{right symmetric.}

\item The two operations satisfy the following relations between them
\[ \begin{cases}
\nu(a,\mu(b,c)) = \nu(a,\nu(b,c))~, \\
\nu(\mu(a,b),c) = \mu(a,\nu(b,c))~,
\end{cases} \]
for any $a,b,c$ in $A$. 
\end{enumerate}
\end{definition}

\begin{remark}
This notion is the \textit{commutative} analogue of the notion of an associative trialgebra introduced by Loday and Ronco in \cite{LodayRonco}.  
\end{remark}

\begin{remark}
The vector space $A$ together with the operation $\nu$ above forms what is called a \textit{permutative algebra}. This type of algebras is encoded by the operad introduced in Example \ref{example: permutative operad}. See \cite{Chapoton01} for more details on this type of algebraic structure.
\end{remark}

\textbf{Generators and relations of the corresponding operad.} The notion of a commutative trialgebra is encoded by an operad, in the sense that there exists an operad such that algebras over this operad are precisely commutative triassociative algebras, see \cite[Section 5.2]{LodayVallette} for the definition of an algebra over an operad. Let us give a presentation of this operad. The $\mathbb{S}$-module of generators $M$ is given by 
    \[
    M \coloneqq (0,0,\kk[\mathbb{S}_2] \cdot \{\nu\} \oplus \kk \cdot \{\mu\}, 0, \cdots)~,
    \]
where $\mu$ is an arity $2$ operation with no symmetries and $\nu$ is a symmetric arity $2$ operation. This follows from the fact that commutative trialgebras are endowed with two binary operations, where one of them is symmetric. The ideal of relations $R$ is generated by the following relations
\[ \begin{aligned}
1) \quad & \mu \circ_1 \mu = \mu \circ_2 \mu, & \quad (\text{associativity of }\mu) \\
2) \quad & \nu \circ_1 \nu = \nu \circ_2 \nu = \nu \circ_2 ((21) \star \nu), & \quad (\text{permutativity of }\nu) \\
3) \quad & \nu \circ_2 \mu = \nu \circ_2 \nu, \\
3') \quad & \nu \circ_1 \mu = \mu \circ_2 \nu. 
\end{aligned} \]
The first two relations have already been depicted in Examples \ref{example: associative operad} and \ref{example: permutative operad}, respectively. The last two relations can be depicted as follows: 
\[
\includegraphics[scale=0.5]{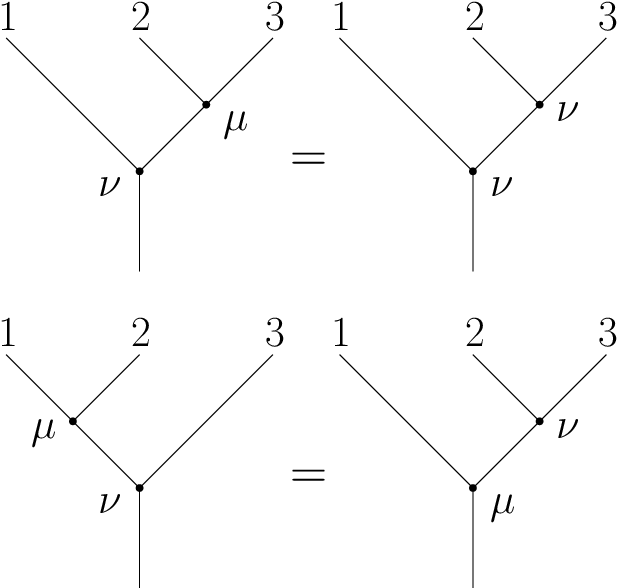}
\] 
We denote by $\mathcal{C}om\mathcal{T}riass$ the operad given by the explicit presentation
\[
\mathcal{C}om\mathcal{T}riass \coloneqq \mathbb{T}(M)/(R)~,
\]
where $M$ and $R$ are the generators and the relations defined above. 

\begin{remark}
There are canonical inclusions $\mathcal{C}om \hookrightarrow \mathcal{C}om\mathcal{T}riass$ and $\mathcal{P}erm \hookrightarrow \mathcal{C}om\mathcal{T}riass$ of operads are determined by, respectively, the inclusion of commutative or the permutative binary product inside the operad $\mathcal{C}om\mathcal{T}riass$. 
\end{remark}

\subsection{The isomorphism with the operad of cooperative games}
In this subsection, we construct an explicit isomorphism of operads between the operad that encodes commutative trialgebras and  the operad of cooperative games. As a result, we obtain an explicit presentation of the operad of cooperative games.

\begin{theorem}\label{thm: iso comtriass et operade}
There is an isomorphism of operads 
\[
\begin{tikzcd}
\varphi: \mathcal{C}om\mathcal{T}riass \xrightarrow{\hspace{1.5cm}} \mathbb{G}~
\end{tikzcd}
\]
determined by the map 
\[ \begin{aligned} 
\varphi(2) : \mathcal{C}om\mathcal{T}riass(2) & \xrightarrow{\hspace{1.5cm}} \mathbb{G}(2) \\
\nu & \xmapsto{\hspace{1.5cm}} d_1 \\
\mu & \xmapsto{\hspace{1.5cm}} B, 
\end{aligned} \]

where $d_1$ stands for the $2$-player dictator game and $B$ stands for the $2$-player bargaining game.
\end{theorem}

\begin{proof}
It follows from a direct computation that $d_1$ and $B$ satisfy the relations that define the operad $\mathcal{C}om\mathcal{T}riass$, hence the morphism of operads $\varphi$ is well-defined. Using \cite[Theorem 14]{Vallettepartition} it can be checked that this morphism is a surjective morphism and that the $\mathbb{R}$-vector space $\mathcal{C}om\mathcal{T}riass(n)$ are of dimension $2^n-1$, which is the same dimension as $\mathbb{G}(n)$. Therefore $\varphi$ is an isomorphism. 
\end{proof}

\begin{remark}
Since the associator of the permutative product needs to be \textit{right} symmetric, sending $\nu$ to the $2$-player dictator game $d_2$ and $\mu$ to $B$ does \textit{not} define a morphism of operads. The isomorphism $\varphi$ is therefore canonical.
\end{remark}

\begin{corollary}\label{cor: generators of the operad of cooperative games}
The operad of cooperative games $\mathbb{G}$ is generated under composition by the two-players bargaining game $\mathrm{B}$ and the two-players dictator games $d_1$ and $d_2$. Thus any $n$-player game is a linear sum of $n-1$ compositions of bargaining and dictator games. 
\end{corollary}

\begin{proof}
Follows directly from Theorem \ref{thm: iso comtriass et operade}. 
\end{proof}

\begin{remark}
There is in particular a morphism of operads $\mathcal{P}erm \hookrightarrow \mathbb{G}$ given by the sending the generator of the permutative product to the $2$-players dictator game $d_1$. The image of this morphism can be identified with the suboperad of $\mathbb{G}$ consisting only in \textit{additive games}. 
\end{remark}

\textbf{The explicit form of the isomorphisms.} The isomorphism of Theorem \ref{thm: iso comtriass et operade} can in fact be factored as the following composition of isomorphisms:

\[
\begin{tikzcd}
\mathcal{C}om\mathcal{T}riass \arrow[r,"\psi"]
&\mathbb{G}^{u} \arrow[r,"\zeta"]
&\mathbb{G}~,
\end{tikzcd}
\]
\vspace{0.1pc}

where $\psi$ sends the basis $\{e_S\}_{S \subseteq N}$ of $\mathcal{C}om\mathcal{T}riass$ constructed in \cite[Theorem 14]{Vallettepartition} to the basis of unanimity games $\{u_S\}_{S \subseteq N}$ of Subsection \ref{subsection: unanimity basis}, and where $\zeta$ is the inverse to the Möbius transform in Proposition \ref{prop: Mobius changes the basis}. Consequently, the inverse of this isomorphism, up to the identification of the linear basis of \cite[Theorem 14]{Vallettepartition} and the linear basis of unanimity games, is given by the Möbius transform of cooperative games: 
\[
\begin{tikzcd}
\mathbb{G} \arrow[r,"\mu"]
&\mathbb{G}^{u} \arrow[r,"\psi^{-1}"]
&\mathcal{C}om\mathcal{T}riass~.
\end{tikzcd}
\]

\vspace{1.5pc}

\section{Distinguished suboperads and properties of the composition}

\vspace{2pc}

The first goal of this section is to first establish a few key properties of the partial composition with respect to derivatives, dual games and marginal vectors. The second goal is to show that almost all the important subclasses of cooperative games that are in practice studied in the literature are stable under partial compositions. All the different suboperads that appear in this section can be depicted in the following inclusion table of suboperads:
\[
\begin{tikzcd}[row sep=2em, column sep=1em, ampersand replacement=\&]
    \boxed{\text{Additive games}} \arrow[d]
    \&
    \& \boxed{k\text{-monotone games}} \arrow[ld]
    \&\boxed{\text{Belief functions}} \arrow[l] \\
    \boxed{\text{Cooperative games}}
    \&\boxed{\text{Capacities}} \arrow[l] 
    \&\boxed{\text{Balanced games}} \arrow[l]
    \&\boxed{\text{Totally balanced games}} \arrow[l,dashed] \\
    \boxed{\text{Normalized games}} \arrow[u]
    \&\boxed{\text{Simple games}} \arrow[l] \arrow[u]
    \&\boxed{k\text{-alternating games}} \arrow[lu]
    \&\boxed{\text{Plausibility measures}} \arrow[l] \\
\end{tikzcd}
\]
Note that the case of totally balanced games is more subtle (see Remark \ref{rmk: totally balanced games}) and will be the subject of future research.

\subsection{The partial tensor product construction and the suboperad of additive games.}\label{subsection: partial tensor product} We start by introducing a useful geometrical construction. We define the  \textit{partial tensor product} bilinear maps
\[
\otimes_i: \mathbb{R}^n \otimes \mathbb{R}^{k} \longrightarrow \mathbb{R}^{n + k -1}~,
\]
for any $n,k \geq 1$ and any element $i \in [n]$ as follows. Let \(x \in \mathbb{R}^n\) and \(y \in \mathbb{R}^k\) be two vectors, let us denote by \( \eta\) the sum of the coefficients of \(y\), i.e., \(\eta = \sum_{i \in [k]} y_i\). The partial tensor product $x \otimes_i y$ of \(x\) and \(y\) at index \(i \in [n]\) is given by 
\[
x \otimes_i y = \underbrace{\eta x_1 \oplus \eta x_2 \oplus \ldots \oplus \eta x_{i-1}}_{\eta x_{\downarrow i}} \oplus \underbrace{x_i y_1 \oplus \ldots \oplus x_i y_k}_{x_i y} \oplus \underbrace{\eta x_{i+1} \oplus \ldots \oplus \eta x_n}_{\eta x_{\uparrow i}},  
\]
with \(x_{\downarrow i} = (x_1, \ldots, x_{i-1})\) and \(x_{\uparrow i} = (x_{i+1}, \ldots, x_n)\). Consider the canonical action of $\mathbb{S}_n$ on $\mathbb{R}^n$ given by permuting the indices of the standard basis. This endows the collection $\mathbb{R}^n$ for all $n \geq 1$ with an $\mathbb{S}$-module structure which we denote by  $\underline{\mathbb{R}}$. It can be checked that $\underline{\mathbb{R}}$, together with the partial tensor product maps $\otimes_i$ forms an linear operad, which we call the \textit{partial tensor operad}.

\begin{proposition}
The partial tensor operad is isomorphic to the suboperad of additive games, and thus in turn isomorphic to the permutative operad. 
\end{proposition}

\begin{proof}
Let $x$ be a vector in $\mathbb{R}^n$, it defines a unique additive game by setting 
\[
v_x(S) = \sum_{i \in S} x_i~,
\]
and any additive game is uniquely characterized by the vector of its values on players. Under this correspondence, it is straightforward to check that for any two vectors \(x \in \mathbb{R}^n\) and \(y \in \mathbb{R}^k\) and any $S \subset [n] \circ_i [m]$, we have that
\[
v_{x \otimes_i y}(S) = v_x \circ_i v_y(S)~. 
\]
\end{proof}

\subsection{Compatibilities of the partial composition formula}
We establish several useful compatibilities of the partial composition formula.

\subsubsection{Compatibility with derivatives with respect to a player} The derivative at a given player of a composite game is given by the following explicit formula, which depends on the location of the player in one of the two player sets. 

\begin{lemma}\label{lemma: dérivée de la composée}
Let $\Gamma_A = (A, \alpha)$ and $\Gamma_B = (B, \beta)$ be two games, and let $i \in A$. We have that 
\[
\partial_j(\alpha \circ_i \beta)(S) = \begin{cases} 
    \beta(B)\partial_j\alpha(S_A) + \partial_j\partial_i\alpha(S_A)\beta(S_B) & \quad \text{if } \  j \in (A \setminus \{i\}), \\ 
    \partial_i\alpha(S_A)\partial_j\beta(S_B)& \quad \text{if } \  j \in B,
\end{cases}
\]
for any $S \subseteq A \diamond_i B$. 
\end{lemma}

\begin{proof}
    Let us compute the derivative of the coalition function \(\alpha \circ_i \beta\) with respect to a given player \(j \in A \diamond_i B\). For all \(S \subseteq A \diamond_i B \setminus \{j\}\), we have two cases. First, assume that \(j \in (A\setminus \{i\})\). We get
    \[ \begin{aligned} 
    \partial_j (\alpha \circ_i \beta)(S) & = \alpha \circ_i \beta(S \cup \{j\}) - \alpha \circ_i \beta(S) \\
    & = \beta(B) \alpha(S_A \cup \{j\}) + \partial_i \alpha(S_A \cup \{j\}) \beta(S_B) - \beta(B) \alpha(S_A) - \partial_i \alpha(S_A) \beta(S_B) \\
    & = \beta(B) \left( \alpha(S_A \cup \{j\}) - \alpha(S_A) \right) + \left( \partial_i \alpha(S_A \cup \{j\}) - \partial_i \alpha(S_A) \right) \beta(S_B) \\
    & = \beta(B) \partial_j \alpha(S_A) + \partial_j \partial_i \alpha(S_A) \beta(S_B). \\
    \end{aligned} \]
    The second case occurs when \(j \in B\), which implies 
    \[ \begin{aligned} 
    \partial_j (\alpha \circ_i \beta)(S) & = \alpha \circ_i \beta(S \cup \{j\}) -\alpha \circ_i \beta(S) \\ 
    & = \beta(B) \alpha(S_A) + \partial_i \alpha(S_A) \beta(S_B \cup \{j\}) - \beta(B) \alpha(S_A) - \partial_i \alpha(S_A) \beta(S_B) \\
    & = \partial_i \alpha(S_A) \left( \beta(S_B \cup \{j\} - \beta(S_B) \right) \\
    & = \partial_i \alpha(S_A) \partial_j \beta(S_B)
    \end{aligned} \]
\end{proof}

More generally, we establish the following formula for the derivative with respect to a coalition, as defined in Paragraph~\ref{Subsubsection: duality and derivatives}.

\begin{proposition}\label{Proposition: dérivée itérée}
Let $\Gamma_A = (A, \alpha)$ and $\Gamma_B = (B, \beta)$ be two games, and let $C \subseteq A \diamond_i B$. We have
\[
\partial_{C}(\alpha \circ_i \beta)(S) = \begin{cases} 
    \beta(B)\partial_{C}\alpha(S_A) + \partial_{C \cup \{i\}}\alpha(S_A)\beta(S_B) & \quad \text{if } \  C \subseteq (A \setminus \{i\}), \\ 
    \partial_{C_A \cup \{i\}}\alpha(S_A)\partial_{C_B}\beta(S_B)& \quad \text{otherwise},
\end{cases}
\]
for any $S \subseteq A \diamond_i B$, where $C_A = C \cap A$ and $C_B = C \cap B$. 
\end{proposition}

\begin{proof}
We will prove it by induction on the size of $C \subseteq A \diamond_i B$. When $C$ is of size one, this is exactly Lemma \ref{lemma: dérivée de la composée}. Let us assume it is true for $C$ of a given size and let us add a player $j$ to $C$. There are four cases to deal with. 

\begin{enumerate}
\item If $C \subseteq A \setminus \{i\}$, and
\begin{enumerate} 
\item $j \in A \setminus \{i\}$, then we have 
\[
\partial_j \big[ \beta(B)\partial_{C}\alpha(S_A) + \partial_{C \cup \{i\}}\alpha(S_A)\beta(S_B) \big] = \beta(B)\partial_{C \cup \{j\}}\alpha(S_A) + \partial_{C \cup \{i,j\}}\alpha(S_A)\beta(S_B)~. 
\]

\item $j \in B$, then we have 
\[
\partial_j(\beta(B)\partial_{C}\alpha(S_A) + \partial_{C \cup \{i\}}\alpha(S_A)\beta(S_B)) = \partial_{C \cup \{i\}}\alpha(S_A)\partial_j\beta(S_B)~. 
\]
\end{enumerate}

\item If $C \nsubseteq A \setminus \{i\}$, and 
\begin{enumerate} 
\item $j \in A \setminus \{i\}$, then we have
\[
\partial_j(\partial_{C_A \cup \{i\}}\alpha(S_A)\partial_{C_B}\beta(S_B)) = \partial_{C_A \cup \{i,j\}}\alpha(S_A)\partial_{C_B}\beta(S_B)~. 
\]

\item $j \in B$, then we have 
\[
\partial_j(\partial_{C_A \cup \{i\}}\alpha(S_A)\partial_{C_B}\beta(S_B)) = \partial_{C_A \cup \{i\}}\alpha(S_A)\partial_{C_B \cup \{j\}}\beta(S_B)~. 
\]
\end{enumerate}
\end{enumerate}
Thus, for any of these cases, the result holds.
\end{proof}

\subsubsection{Compatibility with duality} The duality involution on cooperative games is an automorphism of the operad of cooperative games, that is compatible with the composition in the following sense.

\begin{proposition}\label{prop: compatibility of the composition and duality}
Let $\Gamma_A = (A, \alpha)$ and $\Gamma_B = (B, \beta)$ be two games, and let $i \in A$. We have the following equality of value functions
\[
(\alpha \circ_i \beta)^* = \alpha^* \circ_i \beta^*~, 
\]
and therefore the duality involution commutes with the operadic composition, meaning that
\[
(\Gamma_A \circ_i \Gamma_B)^* = \Gamma_A^* \circ_i \Gamma_B^*.
\]
\end{proposition}

\begin{proof}
Let $S \subseteq A \diamond_i B$. One the one hand, we compute that
\[
\begin{aligned} 
   (\alpha \circ_i \beta)^*(S) & = \alpha \circ_i \beta(A \diamond_i B) - \alpha \circ_i \beta(A \diamond_i B \setminus S) \\
   & = \alpha(A)\beta(B) - \beta(B)\alpha((A \diamond_i B \setminus S) \cap A) \\ & \qquad \qquad - \partial_i \alpha((A \diamond_i B \setminus S) \cap A)\beta((A \diamond_i B \setminus S) \cap B) \\
   & = \alpha(A)\beta(B) - \beta(B)\alpha(A \setminus (S_A \cup \{i\})) - \partial_i \alpha(A \setminus (S_A \cup \{i\}))\beta(B \setminus S_B). \\ 
\end{aligned} 
\]
On the other hand, we have that
\[
\begin{aligned} 
   \alpha^* \circ_i \beta^*(S) & = \beta^*(B)\alpha^*(S_A) + \partial_i\alpha^*(S_A)\beta^*(S_B) \\
   &= \alpha(A)\beta(B) - \beta(B)\alpha(A \setminus S_A) + \partial_i\alpha(A \setminus (\{i\} \cup S_A))[\beta(B) - \beta(B \setminus S_B)] \\
   &= \alpha(A)\beta(B) - \beta(B)\alpha(A \setminus (S_A \cup \{i\})) - \partial_i\alpha(A \setminus (\{i\} \cup S_A))\beta(B \setminus S_B)~,
\end{aligned} 
\]
hence we can observe that both sides are indeed equal. 
\end{proof}

\subsubsection{Compatibility with marginal vectors}
Here, we study the compatibility of the partial composition with marginal vectors. 

\begin{definition}[Marginal vectors]
Let $\Gamma = (N,v)$ be an $n$-player game, and let us choose an identification $N = [n]$. For any $\sigma$ in $\mathbb{S}_n$, the \textit{marginal vector} $m^\sigma(v)$ is a vector in $\mathbb{R}^n$ whose $i$-th coordinate is given by 
\[
m^\sigma_i(v) \coloneqq v([i, \sigma]) - v((i, \sigma))~,
\]
where
\[
[i,\sigma] = \left\{j \in N ~~|~~ \sigma^{-1}(j) \leq \sigma^{-1}(i) \right\} \quad \text{and} \quad (i, \sigma) = \left\{j \in N ~~|~~ \sigma^{-1}(j) < \sigma^{-1}(i) \right\}. 
\]
\end{definition}

\begin{example}
For example, if $\sigma = \mathrm{id}_n$, then the $i$-th coordinate of $m^{\mathrm{id}_n}(v)$ is given by
\[
m^\sigma_i(v) \coloneqq v(\{1,\cdots,i\}) - v(\{1,\cdots,i-1\})~.
\]
\end{example}

\begin{remark}
    The set of all marginal vectors can equivalently be indexed by all the possible total orders on the set of players. For a given order $\{i_1,\cdots,i_n\}$, the marginal vector is given by 
    \[
    m^\sigma_{i_j}(v) = \partial_{i_j} v(\{i_1, \ldots, i_{j-1}\}).
    \]
\end{remark}

\begin{proposition}\label{proposition: composition of marginal vectors}
Let $\Gamma_A = (A, \alpha)$ and $\Gamma_B = (B, \beta)$ be two games, and let $i \in A$. Let $m^\sigma(\alpha)$ and $m^\tau(\beta)$ be marginal vectors of $\Gamma_A$ and $\Gamma_B$ respectively. Then $m^\sigma(\alpha) \otimes_i m^\tau(\beta)$ is a marginal vector of the compound game $\Gamma_A \circ_i \Gamma_B$. 
\end{proposition}

\begin{proof}
We identify $A$ with $[n]$ and $B$ with $[m]$. Let us chose two orders $\{a_1,\cdots,a_n\}$ and $\{b_1,\cdots,b_m\}$ of $[n]$ and $[m]$, respectively. If $n$ and $m$ are the respective marginal vectors, we have that
\[
(n \otimes_i m)_{p} = 
\begin{cases}
\beta(M)\partial_{a_p}\alpha(\{a_1,\cdots,a_{p-1}\}) & \text{if} \quad p<i~, \\
\partial_{a_i}\alpha(\{a_1,\cdots,a_{i-1}\})\partial_{b_p}\beta(\{b_1,\cdots,b_{p-1}\}) & \text{if} \quad i \leq p \leq i+m-1~, \\
\beta(M)\partial_{a_{p+m-1}}\alpha(\{a_1,\cdots,a_{p-1}\}) & \text{if} \quad p>m+i-1~,
\end{cases}
\]
Let us show that this is precisely the marginal vector of $\alpha \circ_i \beta$ associated to the composite order $\{a_1,\cdots, a_{i-1}, b_1, \cdots, b_m, a_{i+1}, \cdots, a_n\}$ on $[n] \circ_i [m]$. This follows from applying Lemma \ref{lemma: dérivée de la composée} in each of the possible cases.

\begin{enumerate}
    \item When $p<i$, the marginal vector associated to this order is given by 
    \[
    \partial_{a_p}(\alpha \circ_i \beta)(\{a_1,\cdots, a_{p-1}\}) = \beta(M)\partial_{a_p}\alpha(\{a_1,\cdots, a_{p-1}\}) + \partial_{a_p}\partial_{a_i}\alpha(\{a_1,\cdots, a_{p-1}\})\beta(\emptyset) 
    \]
    so it coincides with the previous vector since $\beta(\emptyset) = 0$.

    \medskip
    
    \item When $i \leq p \leq i+m-1$, the marginal vector associated to this order is given by 
    \[
    \partial_{b_p}(\alpha \circ_i \beta)(\{a_1,\cdots, a_{i-1}, b_1, \cdots, b_{p-1}\}) = \partial_{a_i}\alpha(\{a_1,\cdots, a_{i-1}\})\partial_{b_p}\beta(\{b_1, \cdots, b_{p-1}\})~,
    \]
    which also coincides with the previous description. 

    \medskip
    
    \item When $p<m+i-1$, the marginal vector associated to this order is given by 
    \[
    \begin{aligned}
        & \partial_{a_{p+m-1}}(\alpha \circ_i \beta)(\{a_1,\cdots, a_{i-1}, b_1, \cdots, b_{m},a_{i+1}, \cdots, a_{p+m-1}\})  \\ & \quad = \beta(M)\partial_{a_{p+m-1}}\alpha(\{a_1,\cdots, a_{i-1}, a_{i+1}, \cdots, a_{p+m-1}\}) \\ & \qquad + \partial_{a_{p+m-1}}\partial_{a_i}\alpha(\{a_1,\cdots, a_{i-1}, a_{i+1}, \cdots, a_{p+m-1}\})\beta(M) \\
        & \quad = \beta(M)\partial_{a_{p+m-1}}\alpha(\{a_1,\cdots, a_{i-1}, a_i, a_{i+1}, \cdots, a_{p+m-1}\})~,
    \end{aligned}
    \]
    so it does coincide again with the vector obtained by the tensor construction. 
\end{enumerate}
\end{proof}

\subsection{The suboperad of normalized games}
Normalized games, that is, games $\Gamma =(N,v)$ which satisfy that $v(N)= 1$, are stable under compositions. 

\begin{proposition}\label{prop: normalized games suboperads}
Normalized games are stable under composition and thus form a suboperad.
\end{proposition}

\begin{proof}
Let us choose \(\Gamma_A = (A, \alpha)\) and \(\Gamma_B = (B, \beta)\) be two normalized games and let \(i \in A\). Then their composition is again normalized since \(\alpha \circ_i \beta(A \diamond_i B) = \alpha(A)\beta(B) = 1\).
\end{proof}

\begin{remark}[An operad in the category of affine spaces]
Given a vector space $V$ together with a linear map $\pi: V \longrightarrow \mathbb{R}$, the fiber $\pi^{-1}(1)$ naturally acquires the structure of an affine space. In fact, the category of affine spaces is equivalent to the category of vector spaces with a non-zero map to $\mathbb{R}$, and this equivalence respects the monoidal structures. One can obtain normalized games as the fiber of the map $\mathbb{G}(n) \longrightarrow \mathbb{R}$ which sends $v$ to $v(N)$, and since this map is compatible with the operad structure, the resulting operad of normalized games is an operad in the category of affine spaces. 
\end{remark}

\subsection{The suboperad of simple games}
Cooperative games which have values in $\{0,1\}$ are called simple games. We show directly that simple monotone games are stable under the composition, thus recovering the suboperad corresponding to the original composition defined by Shapley in \cite{ShapleyComposition}. See also Paragraph \ref{subsub: sum and product of simple games} for how to recover Shapley's product and sum using the operadic composition. This induces an operad structure on a class of combinatorial objects known as \textit{clutters}, see \cite{billera1971composition} for more information.  

\begin{proposition}\label{prop: simple games suboperad}
The simple monotone games are stable under the composition and thus form a suboperad. 
\end{proposition}

\begin{proof}
    Let \(\Gamma_A = (A, \alpha) \) and \(\Gamma_B = (B, \beta)\) be two simple monotone games. Let \(i \in A\) and let \(\Gamma_C = (C, \gamma)\) be the compound game defined by \(\Gamma_C = \Gamma_A \circ_i \Gamma_B\). Notice that since $\alpha$ is simple and monotone, the value of $\partial_i \alpha$ is also either $0$  or $1$. For all \(S \subseteq C\), we have 
    \[
    \gamma(S) = \beta(B) \alpha(S_A) + \partial_i \alpha(S_A) \beta(S_B).
    \]
    If \(\beta(B) = 0\) or \(\beta(S_B) = 0\), then \(\gamma(S)\) is null or a product of \(1\) and \(0\), hence belongs to \(\{0, 1\}\). Henceforth, we assume that \(\beta(B) = \beta(S_B) = 1\). This implies that 
    \[
    \gamma(S) = \alpha(S_A) + \partial_i \alpha(S_A) = \alpha(S_A \cup \{i\}), 
    \]
    which, since $\alpha$ is simple, must either be $0$ or $1$.
\end{proof}

\subsection{The suboperad of capacities}
Capacities, which are defined as games $\Gamma = (N,v)$ whose set function is non-negative and monotone, were introduced by Choquet in \cite{choquet1954theory}. They are ubiquitous in decision theory, in particular in multi-criteria decision making. They also correspond to the original definition of a \textit{fuzzy measure} given by Sugeno in \cite{Sugeno}. We refer to \cite[Section 2.3 and Chapter 6]{Grabisch} for a full textbook account. 

\begin{definition}[Capacity]\label{def: capacity}
A \textit{capacity} is the data of a game $\Gamma = (N,v)$ such that its coalition function $v$ is non-negatively valued and monotone. 
\end{definition}

\begin{remark}
A game $\Gamma = (N,v)$ is monotone if and only if, for all players $i \in N$, the derivative $\partial_i v(S)$ is a non-negative function for all $S \subseteq N \setminus \{i\}$. 
\end{remark}

\begin{remark}
Since we are working with \textit{grounded} games, the monotonicity assumption implies the non-negativeness, since $v(S) \geq v(\emptyset) = 0$. However, we will not use this in the proof, and instead write arguments which generalize trivially also to the non-grounded case. 
\end{remark}

\begin{theorem}\label{proposition: capacities are stable under composition}
Capacities are stable under composition and thus form a suboperad. 
\end{theorem}

\begin{proof}
Let \(\Gamma_A = (A, \alpha)\) and \(\Gamma_B = (B, \beta)\) be two capacities. We consider their composition 
\[
\alpha \circ_i \beta(S) = \beta(B) \alpha(S_A) + \partial_i \alpha(S_A) \beta(S_B)~,
\]
at a coalition $S \subseteq N \setminus \{i\}$. Since $\Gamma_A$ is monotone, $\partial_i \alpha(S_A) \geq 0$, and as the product and the sum of non-negative numbers, the composition is non-negative.

\medskip

Let us now prove that the partial derivative of the composition at any player \(j \in N \setminus \{i\}\) is also non-negative. By Lemma ~\ref{lemma: dérivée de la composée}, there are two cases to consider. First, if \(j \in B\), the derivative of the composite at the player $j$ is given by \(\partial_j (\alpha \circ_i \beta)(S) = \partial_i \alpha(S_A) \partial_j \beta(S_B)\), which is non-negative as the product of two non-negative functions.  Now, if \(j \in A \setminus \{i\} \), we have that 
\[
\partial_j (\alpha \circ_i \beta)(S) = \beta(B) \partial_j \alpha(S_A) + \partial_i \partial_j \alpha(S_A) \beta(S_B). 
\]
Observe that since \(\Gamma_B\) is monotone, we have that \(\beta(B) \geq \beta(S_B)\). Hence, 
    \[ \begin{aligned} 
    \partial_j (\alpha \circ_i \beta)(S) & \geq \left( \partial_j \alpha(S_A) + \partial_i \partial_j \alpha(S_A) \right) \beta(S_B) \\ 
    &\geq \left( \partial_j \alpha(S_A) + \partial_j \alpha(S_A \cup \{i\}) - \partial_j \alpha(S_A) \right) \beta(S_B) \\
    &\geq \partial_j \alpha(S_A \cup \{i\}) \beta(S_B) \\
    & \geq 0~, 
    \end{aligned} \]
which is indeed non-negative, since $\Gamma_A$ is monotone and $\Gamma_B$ is non-negative. 
\end{proof}

\begin{remark}[On non-negative games]
Observe that the monotonicity assumption is, in fact, crucial for Theorem \ref{proposition: capacities are stable under composition} to hold. Indeed, the composition of two non-negative games is not, in general, non-negative. The reason is that if the partial derivative $\partial_i \alpha(S_A) = [\alpha(S_A \cup \{i\}) - \alpha(S_A)] \ll 0$, then the value of the composite 
\[
\alpha \circ_i \beta(S) = \beta(B) \alpha(S_A) + \partial_i \alpha(S_A) \beta(S_B)~,
\]
can in fact be negative. 

\medskip

Let us give a concrete example of this. Consider the game $\alpha$ on $\{1, 2\}$ given by $\alpha(\{1\}) = 0$, $\alpha(\{2\}) = 1000$ and $\alpha(\{1, 2\}) = 0$ and the game $\beta$ on $\{a, b\}$ defined by $\beta(\{a\}) = 0$, $\beta(\{b\}) = 1$ and $\beta(\{a, b\}) = \frac{1}{2}$. Then, it can be computed that
\[
\alpha \circ_1 \beta(\{b, 2\}) = 500 - 1000 = -500~,
\]
which is negative, while both games are indeed non-negative. 
\end{remark}

Finally, we show that the duality involution of Definition \ref{definition: dual game} descends to the suboperad of capacities. This will be important later on, when studying the suboperads of the operad of capacities given by $k$-monotone and $k$-alternating games, which are exchanged by the dual game involution. 

\begin{proposition}\label{prop: the dual of a capacity is a capacity}
The dual of a capacity is again a capacity. Therefore duality defines an involution of the operad of capacities.
\end{proposition}

\begin{proof}
Let $\Gamma = (N,v)$ be a capacity. Recall that its dual game is given by 
\[
v^*(S) \coloneqq v(N) - v(N \setminus S)~,
\]
for all $S \subseteq N$. Since $v$ is monotone, $v(N) \geq v(N \setminus S)$. Hence, since it is non-negative, we have that $v^*(S) = v(N) - v(N \setminus S) \geq 0$ for all $S \subseteq N$, so the dual game is non-negative. By Lemma \ref{lemma: duality and derivatives}, the derivative of the dual game is given by 
\[
\partial_jv^*(S) = \partial_jv(N \setminus S) \geq 0~, 
\]
for any $j \in N$, so the dual game $\Gamma = (N,v^*)$ is also monotone, thus capacities are stable under the duality involution. By Proposition \ref{prop: compatibility of the composition and duality}, duality is compatible with the operadic composition, so it defines an involution of the suboperad of capacities. 
\end{proof}

\subsection{Convex games and the hierarchy of $k$-monotone suboperads}
Convex non-negative games are very significant objects in combinatorics. They are equivalent to \textit{polymatroids}, as defined for instance in \cite[Definition 1.1]{Polymatroids}. They are also completely characterized by their \textit{cores}, as defined in Paragraph~\ref{subsubsection: the core of a game}. Their cores are \textit{generalized permutahedra} in the sense of Postnikov \cite{Postnikov} and any generalized permutahedra is uniquely realized by a non-negative convex games, as explained in \cite[Section 12]{AguiarArdila}.

\medskip

We show that non-negative convex games are stable under partial compositions, and thus define a (sub)operad. This endows this family of objects with a new operad structure, whose applications shall be explored in \cite{deuxiemepapier}. Furthermore, non-negative convex games are included in the family of $k$-monotone games (where they correspond to the case $k =2$). We show that, in general, $k$-monotone games form a suboperad for any $k \geq 2$. The case of $k = \infty$ corresponds to totally monotone games, which are also known as \textit{belief functions} and play an important role in the Dempster-Shafer theory of evidence, see \cite{dempster1967upper, shafer1976mathematical, kohlas2013mathematical}. 

\subsubsection{Convex games or $2$-monotone games}
We show that convex games with non-negative values are stable under the composition. 

\begin{definition}[Convex game]\label{definition: convex game}
Let $\Gamma = (N,v)$ be a game. It is \textit{convex} if for all $S,T \subseteq N$, the following inequality holds
\[
v(S) + v(T) \leq v(S \cap T) + v(S \cup T)~. 
\]
\end{definition}

One can characterize convex (also known as supermodular) games purely in terms of the partial derivatives of size $2$ of the game. 

\begin{proposition}[{\cite[Corollary~2.23]{Grabisch}}]\label{prop: convex-characterization}
Let $\Gamma = (N,v)$ be a game. It is convex if and only if $\partial_{ij} v \geq 0$, or, equivalently, for every player $i \in N$, the partial derivative $\partial_i v$ is a monotone function. 
\end{proposition}

Furthermore, by adding the non-negativeness assumption, we get that convex, non-negative games are also monotone. See \cite[Theorem~2.20, Point iii)]{Grabisch} for more details. These two conditions are required for these games to be stable under the partial composition.

\begin{theorem}\label{thm: opérade des convexes}
    The convex, non-negative games are stable under composition and thus form a suboperad. 
\end{theorem}

\begin{proof}
Let \(\Gamma_A = (A, \alpha) \) and \(\Gamma_B = (B, \beta)\) be two non-negative, convex games and let \(\Gamma_C = \Gamma_A \circ_i \Gamma_B\) be their composite along a player \(i \in A\). Our goal is to show that for any \(S \subseteq A \diamond_i B\) and any two players \(j, k \in (A \diamond_i B) \setminus S\), the derivative $\partial_{jk} \gamma(S) \geq 0$ and thus by Proposition \ref{prop: convex-characterization} the composite game will be convex. By Proposition \ref{Proposition: dérivée itérée}, there are three cases. When both players are in $B$, we have that the partial derivative $\partial_{jk} \gamma(S) = \partial_i\alpha(S_A)\partial_{jk}\beta(S_B)$ is non-negative since it is the product of two non-negative functions (since $\alpha$ is monotone and $\beta$ is convex). When one player is in $B$ and one in $A \setminus \{i\}$, we have that the partial derivative $\partial_{jk} \gamma(S) = \partial_{ij}\alpha(S_A)\partial_{k}\beta(S_B)$ is non-negative since it is the product of two non-negative functions (since $\beta$ is monotone and $\alpha$ is convex). Finally, when both $j,k$ are in $A \setminus \{i\}$, we have that
\[
\partial_{jk} \gamma(S) = \beta(B) \partial_k \partial_j \alpha(S_A) + \partial_i \partial_k \partial_j \alpha(S_A) \beta(S_B)~, 
\]
and now, like in the proof of Theorem \ref{proposition: capacities are stable under composition}, we use that \(\beta(B) \geq \beta(S_B) \geq 0\) to show that 
\[
\partial_{jk} \gamma(S) \geq \partial_{jk} \alpha(S_A \cup \{i\}) \beta(S_B)~,
\]
which is greater than $0$ by our previous convexity and non-negativity assumptions; thus it concludes the proof. 
\end{proof}

\begin{corollary}
There is an operad structure on the set of all generalized permutahedra.     
\end{corollary}

\begin{proof}
Follows directly from Theorem \ref{thm: opérade des convexes}, under the correspondence explained in \cite[Section 12]{AguiarArdila}. 
\end{proof}

\begin{remark}[About the geometric side of this operad structure] Let $\Gamma = (N,v)$ be a non-negative convex game. Its associated \textit{core} $C(\Gamma)$ is a generalized permutahedra; it corresponds to the \textit{base polytope} associted to the convex function $v$ in the terminology of \cite[Section 12.3]{AguiarArdila}. 

\medskip

Let $\mathfrak{p}_A$ and $\mathfrak{p}_B$ be two generalized permutahedra, and let \(\Gamma_A = (A, \alpha) \) and \(\Gamma_B = (B, \beta)\) be the corresponding convex games to these polytopes. The operad structure induced by Theorem \ref{thm: opérade des convexes} is given, a priori, by taking the core of the composite game
\[
\mathfrak{p}_A \circ_i \mathfrak{p}_B \coloneqq C(\Gamma_A \circ_i \Gamma_B)~,
\]
and hence is so far not explicitly given by a geometrical construction. We can describe a subset of $C(\Gamma_A \circ_i \Gamma_B)$ in terms of the partial tensor product of $C(\Gamma_A)$ and $C(\Gamma_B)$, see Theorem \ref{thm: inclusion of the tensor i of the cores}, but in general this construction is not sufficient to obtain the full polytope $C(\Gamma_A \circ_i \Gamma_B)$ from the other two polytopes. Exploring if there exists a geometrical construction of this operad structure which does not involve the correspondence with convex games should be the subject of future research.  
\end{remark}

\subsubsection{General $k$-monotone games} Being a convex game is a special case of a more general condition called $k$-monotonicity, which we show is also stable under the partial compositions of our operad structure. 

\begin{definition}[$k$-monotone games]
Let $\Gamma = (N,v)$ be a game. Let $k \geq 2$. The game is $k$\textit{-monotone} if for any family of $k$ sets $A_1, \cdots, A_k$, the following inequality holds: 
\[
v\left(\bigcup_{i=1}^k A_i \right) \geq \sum_{\substack{I \subseteq \{1,\cdots,k\}\\ I \neq \emptyset}} (-1)^{|I|+1} v\left(\bigcap_{i\in I} A_i\right)~. 
\]
It is $\infty$-monotone if this holds for all $k \geq 2$. 
\end{definition}

\begin{remark}
For $k =2$, this definition is exactly the definition of a convex game, see Definition \ref{definition: convex game}. More generally, being $k$-monotone implies being $k'$-monotone for any $2 \leq k' \leq k$. 
\end{remark}

Like convex games, $k$-monotone games can be characterized by the positivity of the iterated partial derivatives. 

\begin{theorem}[{\cite[Theorem~2.21]{Grabisch}}]
Let $\Gamma = (N,v)$ be a game. It is $k$-monotone for $k \geq 2$ if and only if for any $K$ of size $2 \leq |K| \leq k$, the iterated partial derivative $\partial_K v$ is a non-negative function. 
\end{theorem}

\begin{remark}
There is also a characterization of $k$-monotonicity in terms of the Möbius transform which is as follows. Let $k \geq 2$, a game $(N,v)$ is $k$-monotone if and only 
\[
\sum_{L \in [A,B]} \mu^v(L) \geq 0~,
\]
for any $A \subseteq N$ such that $2 \leq |A| \leq k$ and any $B \subseteq N$ such that $A \subseteq B$. Here $[A,B]$ stands for the set of all subsets $L$ of $N$ such that $A \subseteq L \subseteq B$. 
\end{remark}

Since any $k$-monotone game is $2$-monotone, this implies that if it is non-negative, then it is also monotone. This corresponds to the case where all partial derivatives of any length $\leq k$ are non-negative functions. So they are in particular capacities, see Definition \ref{def: capacity}.

\begin{theorem}\label{thm: opérade des k-monotones}
For any $k \geq 2$, $k$-monotone, non-negative games are stable under composition and thus form a suboperad. 
\end{theorem}

\begin{proof}
Let \(\Gamma_A = (A, \alpha) \) and \(\Gamma_B = (B, \beta)\) be two non-negative, $k$-monotone games and let \(\Gamma_C = \Gamma_A \circ_i \Gamma_B\) be their composite along a player \(i \in A\). Let us show that their composite is $k$-monotone. We work by induction of $k \geq 2$. For $k = 2$, this is Theorem \ref{thm: opérade des convexes}. Assume $(k-1)$-monotonicity is preserved, let us show that $k$-monotonicity is also preserved. Let $K$ be a subset of size $k-1$ and let $j$ be in $A \diamond_i B \setminus K$. We want to show that $\partial_{K \cup \{j\}}(\alpha \circ_i \beta)$ is non-negative. Using Proposition \ref{Proposition: dérivée itérée}, the only non-trivial case is when both $K$ and $j$ are in $A \setminus \{i\}$, in which case we apply the same type of arguments as in Theorem \ref{proposition: capacities are stable under composition} and Theorem \ref{thm: opérade des convexes}. 
\end{proof}

As a direct corollary, we get that \textit{belief functions}, which correspond to totally monotone games in the Dempster--Shafer theory of evidence \cite{dempster1967upper, shafer1976mathematical, kohlas2013mathematical}, are also also stable under the partial compositions and thus form a (sub)operad. 

\begin{corollary}\label{cor: totally monotone}
Non-negative, totally monotone (also known as $\infty$-monotone) games are stable under composition and thus form a suboperad.
\end{corollary}

\begin{proof}
Follows from Theorem \ref{thm: opérade des k-monotones}, since totally monotone are $k$-monotone games for all $k \geq 2$.
\end{proof}

\begin{remark}
This result also follows from the description of the partial composition in the basis of unanimity games given in Subsection \ref{subsection: unanimity basis}. Indeed, a non-negative game is totally monotone if and only if its Möbius transform has non-negative coefficients (see \cite[Equation~(2.1)]{shafer1976mathematical}). 
\end{remark}

\subsection{The dual case of submodular games and $k$-alternating games} The convexity condition on a game (also called supermodularity) is sent, under the dual game construction of Definition \ref{definition: dual game}, to a condition called submodularity. The family of $k$-alternating games for $k \geq 2$ generalizes this submodularity condition. We show that non-negative monotone $k$-alternating games are stable under partial compositions, and thus form a suboperad of the operad of capacities. These operads are the duals of the operads of non-negative $k$-monotone games constructed in the previous subsection. It also follows that $\infty$-alternating functions also form a suboperad. These games are known as a \textit{plausibility measures} in the Dempster--Shafer theory of evidence, see \cite{shafer1976mathematical, Hohle87} for more details. 

\begin{definition}[$k$-alternating games]
Let $\Gamma = (N,v)$ be a game let $k \geq 2$. The game is $k$\textit{-alternating} if for any family of $k$ sets $A_1, \cdots, A_k$, the following inequality holds: 
\[
v\left(\bigcup_{i=1}^k A_i \right) \leq \sum_{\substack{I \subseteq \{1,\cdots,k\}\\ I \neq \emptyset}} (-1)^{|I|+1} v\left(\bigcap_{i\in I} A_i\right)~. 
\]
It is $\infty$-alternating if this holds for all $k \geq 2$. 
\end{definition}

\begin{remark}
A game $(N,v)$ is $2$-alternating if for all $S,T \subseteq N$, the following inequality holds
\[
v(S) + v(T) \geq  v(S \cap T) + v(S \cup T)~. 
\]
This type of functions are also known as \textit{submodular functions} (or submodular games). 
\end{remark}

We are going to show that $k$-alternating capicities also define suboperads, for any $k \geq 2$. The main ingredient is the duality that relates $k$-alternating and $k$-monotone games. 

\begin{theorem}[{\cite[Theorem~2.20, Point ii)]{Grabisch}}]\label{thm: duality between k-monotone and k-alternating}
Let $\Gamma = (N,v)$ be a game. It is $k$-monotone (resp. $k$-alternating) if and only if its dual $(N,v^*)$ is $k$-alternating (resp. $k$-monotone). 
\end{theorem}

\begin{theorem}\label{thm: operad of infinity alternating}
For any $k \geq 2$, $k$-alternating, non-negative monotone games are stable under composition and thus form a suboperads.
\end{theorem}

\begin{proof}
Let \(\Gamma_A = (A, \alpha) \) and \(\Gamma_B = (B, \beta)\) be two non-negative monotone $k$-alternating games and let \(\Gamma_C = \Gamma_A \circ_i \Gamma_B\) be their composite along a player \(i \in A\). By Theorem \ref{proposition: capacities are stable under composition} and Proposition \ref{prop: the dual of a capacity is a capacity}, the composite $\Gamma_C^*$ is again a capacity. By Theorem \ref{thm: opérade des k-monotones}, it is also $k$-monotone. Therefore $\Gamma_C$, as the dual of $\Gamma_C^*$, is again a capacity and by Theorem \ref{thm: duality between k-monotone and k-alternating}, it is $k$-alternating. 
\end{proof}

\begin{corollary}
Non-negative monotone $\infty$-alternating games are stable under composition and thus form a suboperad.
\end{corollary}

\begin{proof}
Follows from Theorem \ref{thm: operad of infinity alternating}, since $\infty$-alternating games are $k$-alternating games for all $k \geq 2$.    
\end{proof} 

\subsection{Balanced games}
Balanced games are games who are well behaved when one evaluates them at any balanced collection. By the famous Bondareva--Shapley theorem, they are precisely the games whose core is non-empty. We show that they are stable under partial compositions and thus form a suboperad. Therefore the core of a compound game is non-empty if the core of the two components are non-empty. We refer to Subsection \ref{Subsection: the core of a composite game} for a more detailed study of the core of a compound game, in particular to Theorem \ref{thm: inclusion of the tensor i of the cores} which gives another proof of the result of this subsection.

\begin{definition}[Balanced collections]
    A collection of coalitions \(\calB \subseteq \mathcal{P}(N)\) is \emph{balanced} over \(N\) if there exists a map \(\lambda: \calB \to \hspace{2pt} ]0, 1]\), called \textit{its weights}, such that 
    \[
    \sum_{S \in \calB} \lambda_S \mathbf{1}^S = \mathbf{1}^N.
    \]
\end{definition}

\begin{definition}[Balanced game]
A game \(\Gamma = (N, v)\) is \emph{balanced} if, for all balanced collections \(\calB\) over \(N\), the following inequality holds
    \[
    \sum_{S \in \calB} \lambda_S v(S) \leq v(N). 
    \]
\end{definition}

\begin{theorem}[Bondareva--Shapley]\label{thm: bondareva-shapley}
A game \(\Gamma = (N, v)\) is balanced if its core $C(\Gamma)$ is non-empty. 
\end{theorem}

\begin{theorem}\label{prop: balanced}
    Balanced, non-negative monotone games are stable under composition and form a suboperad. 
\end{theorem}

\begin{proof}
    Let \(\Gamma_A = (A, \alpha) \) and \(\Gamma_B = (B, \beta)\) be two non-negative monotone balanced games. Let \(i \in A\) and let \(\Gamma_C = (C, \gamma)\) be the compound game defined by \(\Gamma_C = \Gamma_A \circ_i \Gamma_B\). By Theorem \ref{proposition: capacities are stable under composition}, it is again non-negative and monotone. Let us check that it is balanced.

    \medskip
    
    Let \(\calB\) be a balanced collection over \(C\) together with a compatible weight \(\lambda\). We have 
    \[ \begin{aligned} 
    \sum_{S \in \calB} \lambda_S \gamma(S) & = \sum_{S \in \calB} \lambda_S \left( \beta(B) \alpha(S_A) + \partial_i \alpha(S_A) \beta(S_B) \right) \\
    & = \beta(B) \sum_{S \in \calB} \lambda_S \alpha(S_A) + \sum_{S \in \calB} \lambda_S \partial_i \alpha(S_A) \beta(S_B). 
    \end{aligned} \]
    Since \(\Gamma_B\) is balanced and \(\{S_B, B \setminus S_B\}\) is a balanced collection, we have that 
    \[
    \beta(S_B) + \beta(B \setminus S_B) \leq \beta(B). 
    \]
    And since \(\Gamma_B\) is non-negative, we finally have that \(\beta(S_B) \leq \beta(B)\), implying that
    \[ \begin{aligned} 
    \sum_{S \in \calB} \lambda_S \gamma(S) & \leq \beta(B) \sum_{S \in \calB} \lambda_S \alpha(S_A) + \sum_{S \in \calB} \lambda_S \partial_i \alpha(S_A) \beta(B) \\
    & = \beta(B) \Big[ \sum_{\substack{S \in \calB \\ S_B = \emptyset}} \lambda_S \alpha(S_A) + \sum_{\substack{S \in \calB \\ S_B \neq \emptyset}} \lambda_S \left( \alpha(S_A) + \partial_i \alpha(S_A) \right) \Big] 
    \end{aligned} \]
    From previous computations, we know that \(\alpha(S_A) + \partial_i \alpha(S_A) = \alpha(S_A \cup \{i\}) \), therefore
    \[
    \sum_{S \in \calB} \lambda_S \gamma(S) \leq \beta(B) \sum_{S \in \calB_A} \lambda_S \alpha(S), 
    \]
    with \(\calB_A\) is obtained by replacing \(S_B\), if nonempty, by \(\{i\}\) in every coalition of \(\calB\). Because the condition of balancedness expresses itself player by player, the collection \(\calB_A\) is balanced on \(A\). Now, by balancedness of \(\Gamma_A\), we have 
    \[
    \sum_{S \in \calB} \lambda_S \gamma(S) \leq \beta(B) \sum_{S \in \calB_A} \lambda_S \alpha(S) \leq \beta(B) \alpha(A). 
    \]
We conclude using the fact that \(\gamma(C) = \beta(B) \alpha(A)\).  
\end{proof}

\begin{remark}[About totally balanced games]\label{rmk: totally balanced games}
Totally balanced games are those whose restrictions to subcoalitions are balanced for every subcoalition. They appear in combinatorial optimization problems~\cite{curiel1988cooperative}, such as linear programming games~\cite{debreu1963limit, owen1975core}, flow games~\cite{kalai1982totally}, market games~\cite{shapley1969market}, assignment games~\cite{gale1962college, shapley1971assignment} and permutation games~\cite{shapley1974cores}. 

\medskip

Using the same proof strategy as in Theorem \ref{prop: balanced}, it is possible to show that if $\Gamma_A = (A, \alpha)$ is a non-negative game whose derivative $\partial_i \alpha$ is totally balanced (in an appropriate sense) and $\Gamma_B = (B,\beta)$ is a non-negative totally balanced game, their composite $\Gamma_A \circ_i \Gamma_B$ is again a totally balanced game. This happens, for instance, when $\Gamma_A$ is a non-negative $3$-monotone game. Therefore partial compositions endow the subspecies of non-negative totally balanced games with a left module structure over the suboperad of non-negative $3$-monotone game. Since this situation is more involved than with the other suboperads, it is beyond the scope of the present paper and will be the subject of future research. 
\end{remark}

\vspace{1.5pc}

\section{The aggregation of solutions for cooperative games}\label{Section: aggregation of solutions}

\vspace{2pc}

In this section, we study how solution concepts such as the core, the Shapley value and the Banzhaf index behave with respect to the operadic composition. We start by showing that, using the partial tensor product of Subsection \ref{subsection: partial tensor product}, we can construct (pre)imputations of a composite game from (pre)imputations of its components. Moreover, we show that under fairly light hypothesis, this map is surjective on imputations and thus any imputation of the compound game is a partial tensor product of imputations of its components. We study the core of a compound game, which lies in the simplex of imputations. We show that the partial tensor product of two elements in the core of the components lies in the core of the composite, and we give conditions under which this map is injective and also when it is close to being surjective. Finally, we give an explicit formula for the Shapley value of a composite game and show that the Banzhaf index of a composite game can be described as the partial tensor product of the Banzhaf index of the quotient and of the Shapley value of the component. 

\subsection{The (pre)imputations of a composite game}
Let \(\Gamma = (N, v)\) be a game. We consider the vector space $\mathbb{R}^N$ spanned by the player set of the game. For a given vector $x = (x_1, \cdots, x_n)$, we view the value $x_i$ as the gain of the $i$-th player in $N$, and the vector space $\mathbb{R}^N$ as the space of all allocations.

\medskip

\textbf{Preimputations.} Usually, all solution concepts of cooperative game theory live in the affine subspace of \emph{preimputations} of the game $\Gamma$, which is defined by
\[
X(\Gamma) = \{x \in \mathbb{R}^N \mid x(N) = v(N)\}. 
\]
One can think of vectors in this affine hyperplane as possible ways to share the value $v(N)$ of the grand coalition of the game among the players. 
The core, the stable sets, the Banzhaf and Shapley values all lie in this affine subspace. 

\medskip

\textbf{Imputations.} The simplex of imputations \(I(\Gamma)\) is defined as 
\[
I(\Gamma) \coloneqq \{x \in X(\Gamma) \mid x_i \geq v(\{i\})\}. 
\]
Thus vectors in $I(\Gamma)$ are all the possible ways to share the value of the grand coalition among the players in which each player, individually, is better off than if it stayed alone. When a game is normalized and the value of the singletons are all $0$, the set of imputations is simply the standard \((n-1)\)-dimensional simplex, i.e., the convex combination of the canonical basis of \(\mathbb{R}^N\). 

\medskip

Let $\Gamma = (N,v)$ be a non-negative game. We define the \textit{cooperative surplus} as 
\[
v_{\mathrm{sp}} \coloneqq v(N) - \sum_{i \in N} v(\{i\})~. 
\]
In this case, the vertices of $I(\Gamma)$ are given by 
\[
v^{(i)} \coloneqq v_{\mathrm{sp}}\mathbf{1}^{\{i\}} + \sum_{j \in N} v(\{j\})\mathbf{1}^{\{j\}}~, 
\]
and $I(\Gamma)$ is the convex hull of these vectors. We show that the partial tensor product constructed in Subsection \ref{subsection: partial tensor product} allows us to construct (pre)imputations of the composite of two games. 

\begin{lemma}\label{lemma: inclusion of preimputations}
    Let \(\Gamma_A = (A, \alpha)\) and \(\Gamma_B = (B, \beta)\) be two games, and let \(i \in A\). The partial tensor product restricts to a well defined map
    \[
    (- \otimes_i -): X(\Gamma_A) \times X(\Gamma_B) \longrightarrow X(\Gamma_A \circ_i \Gamma_B)~.
    \]
\end{lemma}

\begin{proof}
Let $x$ be in $X(\Gamma_A)$ and let $y$ be in $X(\Gamma_B)$, and let $z = x \otimes_i y$. We can compute that 
\[
\sum_{j \in A \diamond_i B} z_j = \left(\sum_{j \in B} y_j \right) \sum_{l \in A \setminus \{i\}} x_l + x_i \sum_{j \in B} y_j = \alpha(A)\beta(B) = \gamma(A \diamond_i B)~, 
\]
where $\sum_{l \in A\setminus \{i\}} x_l = \alpha(A)$ and $\sum_{j \in B} y_j = \beta(B)$ since we assumed that \(x \in X(\Gamma_A)\) and \(y \in X(\Gamma_B)\). So the map 
\[
(- \otimes_i -): X(\Gamma_A) \times X(\Gamma_B) \longrightarrow X(\Gamma_A \circ_i \Gamma_B)
\]
is well defined.
\end{proof}

\begin{remark}
    Notice that the partial tensor product construction, as a linear map 
    \[
    (- \otimes_i -): \mathbb{R}^{n} \otimes \mathbb{R}^{m} \longrightarrow \mathbb{R}^{n+m-1}
    \]
    is clearly not injective. It is, however, surjective since any element in the canonical basis of $\mathbb{R}^{n+m-1}$ is in its image. 
\end{remark}

\begin{proposition}\label{prop: tenseur i des imputations}
    Let \(\Gamma_A = (A, \alpha)\) and \(\Gamma_B = (B, \beta)\) be two non-negative games, and let \(i \in A\). The partial tensor product restricts to a well defined map
    \[
    (- \otimes_i -): I(\Gamma_A) \times I(\Gamma_B) \longrightarrow I(\Gamma_A \circ_i \Gamma_B)~, 
    \]
    so the partial tensor product of imputations is an imputation in the composite game. Furthermore, if $\alpha(\{i\})$ and  $\beta(B)$ are positive, this map is injective. 
\end{proposition}

\begin{proof}
Let $x$ be in $I(\Gamma_A)$ and let $y$ be in $I(\Gamma_B)$. By Lemma \ref{lemma: inclusion of preimputations}, we know that their image via the partial tensor product is a preimputation. Let us check that $x \otimes_i y$ lies in $I(\Gamma_A \circ_i \Gamma_B)$, meaning we have 
\[
(x \otimes_i y)_j \geq \alpha \circ_i \beta(\{j\})~,
\]
for all $j \in A \diamond_i B$. If $j$ is in $A \setminus \{i\}$, then we have that 
\[
(x \otimes_i y)_j = \beta(B)x_j \geq \beta(B)\alpha(\{j\}) = \alpha \circ_i \beta(\{j\})~,
\]
since $\beta(B) \geq 0$ and since $x_j \geq \alpha(\{j\})$. If $j$ is in $B$, then 
\[
(x \otimes_i y)_j = x_iy_j \geq \alpha(\{i\})\beta(\{j\}) = \alpha \circ_i \beta(\{j\})~, 
\]
since $x_i \geq \alpha(\{i\}) \geq 0$ and $y_j \geq \beta(\{j\}) \geq 0$. 

\medskip

Let us show that it is injective: suppose that there exists $x'$ in $X(\Gamma_A)$ and let $y'$ in $X(\Gamma_B)$ such that $x' \otimes_i y' = x \otimes_i y$. Then clearly $\beta(B)x_j = \beta(B)x_j'$ for all $j \in A \setminus\{i\}$, which implies that $x_j = x_j'$ for all $j \in A \setminus\{i\}$ since $\beta(B) \neq 0$. Since $x$ and $x'$ are in $X(\Gamma_A)$, then 
\[
x_i = \beta(B) - \sum_{j \in B \setminus \{i\}} x_j = \beta(B) - \sum_{j \in B \setminus \{i\}} x_j' = x_i'~,  
\]
and therefore $x = x'$. This, in turn, implies that $y = y'$ since for all $j \in B$ we have $x_i y_j = x_i y_j'$ with $x_i \geq \alpha(\{i\}) > 0.$
\end{proof}

\begin{example}
    Let us provide a counterexample for the surjectivity of the partial tensor product map on imputations. Let \(\Gamma_A = (\{a_1, a_2\}, \alpha)\) and \(\Gamma_B = (\{b_1, b_2\}, \beta)\) be two games defined by 
    \[ \begin{aligned} 
    \alpha(a_1) = 1, \quad \alpha(a_2) = 1, \quad & \text{and} \quad \alpha(a_1 a_2) = 3 \\
    \beta(b_1) = 0, \quad \beta(b_2) = 2, \quad & \text{and} \quad \beta(b_1 b_2) = 3. 
    \end{aligned} \]
    The imputations simplices are given by 
    \[
    I(\Gamma_A) = \operatorname{conv} \left\{ \begin{pmatrix} 2 \\ 1 \end{pmatrix}, \begin{pmatrix} 1 \\ 2 \end{pmatrix} \right\}, \qquad \text{and} \qquad I(\Gamma_B) = \operatorname{conv} \left\{ \begin{pmatrix} 1 \\ 2 \end{pmatrix}, \begin{pmatrix} 0 \\ 3 \end{pmatrix} \right\}. 
    \]
    Hence, we have 
    \[
    I(\Gamma_A) \otimes_1 I(\Gamma_B) = \operatorname{conv} \left\{ v_1 = \begin{pmatrix} 2 \\ 4 \\ 3 \end{pmatrix}, v_2 = \begin{pmatrix} 1 \\ 2 \\ 6 \end{pmatrix}, v_3 = \begin{pmatrix} 0 \\ 6 \\ 3 \end{pmatrix}, v_4 = \begin{pmatrix} 0 \\ 3 \\ 6 \end{pmatrix} \right\}. 
    \]
    The composition \(\Gamma_C = \Gamma_A \circ_{a_1} \Gamma_B\) satisfies in particular 
    \[
    \gamma(a_2) = 3, \quad \gamma(b_1) = 0, \quad \gamma(b_2) = 2, \quad \text{and} \quad \gamma(b_1 b_2 a_2) = 9. 
    \]
    This yields the following imputation simplex
    \[
    I(\Gamma_A \circ_1 \Gamma_B) = \operatorname{conv} \left\{ w_1 = \begin{pmatrix} 0 \\ 2 \\ 7 \end{pmatrix}, w_2 = \begin{pmatrix} 0 \\ 6 \\ 3 \end{pmatrix}, w_3 = \begin{pmatrix} 4 \\ 2 \\ 3 \end{pmatrix} \right\}. 
    \]
    By counting the number of vertices, we see that \(I(\Gamma_A) \otimes_1 I(\Gamma_B)\) is distinct from \(I(\Gamma_A \circ_1 \Gamma_B)\). However, the image of $I(\Gamma_A) \times I(\Gamma_B)$ does lie in \(I(\Gamma_A \circ_1 \Gamma_B)\), as we have that: 
    \[
    v_1 = \frac{1}{2} w_2 + \frac{1}{2} w_3, \qquad v_2 = \frac{3}{4} w_1 + \frac{1}{4} w_3, \qquad v_3 = w_2, \qquad \text{and} \qquad v_4 = \frac{3}{4} w_1 + \frac{1}{4} w_2. 
    \]
\end{example}

\begin{remark}
    In the following, we use a regular abuse of notation concerning the name of some vectors and their dimension. We are only interested in the specific nonzero entries of vectors, which corresponds to specific players. So, we when write an equation such as \(\mathbf{1}^{\{p\}} \otimes_i \mathbf{1}^{\{q\}} = \mathbf{1}^{\{p\}}\) with \(p \neq i\), it is understood that the dimension of the vector on the right-hand side makes the equation valid, and the same player \(j\) has a payment of \(1\), while everyone else has \(0\). Finally, vectors must be more understood as linear games rather than actual vector expressed in a specific basis. 
\end{remark}

\begin{proposition}\label{prop: tensor of imputations is surjective}
Let \(\Gamma_A = (A, \alpha)\) and \(\Gamma_B =(B, \beta)\) be two non-negative games, and let \(i \in A\). If $\alpha(\{i\}) = 0$ and \(\beta(\{l\}) = 0\) for all \(l \in B\), the map
\[
(- \otimes_i -): I(\Gamma_A) \times I(\Gamma_B) \longrightarrow I(\Gamma_A \circ_i \Gamma_B)~, 
\]
is surjective. 
\end{proposition}

\begin{proof}
The strategy of the proof is to first show that all the vertices of $I(\Gamma_A \circ_i \Gamma_B)$ can be reached by the partial tensor construction of the vertices of \(I(\Gamma_A)\) and \(I(\Gamma_B)\) and then deduce the general case. The vertices of \(I(\Gamma_A)\) and \(I(\Gamma_B)\) are respectively given by
    \[
    a^{(p)} \coloneqq \alpha_{\text{sp}} \mathbf{1}^{\{p\}} + \sum_{k \in A \setminus \{i\}} \alpha(\{k\}) \mathbf{1}^{\{k\}} \qquad \text{and} \qquad b^{(q)} \coloneqq \beta_{\text{sp}} \mathbf{1}^{\{q\}} = \beta(B) \mathbf{1}^{\{q\}},
    \]
since, by assumption,  \(\beta(\{l\}) = 0\) for all \(l \in B\). We first assume that \(p \neq i\). So, because \(a^{(p)}_i = 0\), the partial tensor product gives
    \[
    a^{(p)} \otimes_i b^{(q)} = \beta(B) a^{(p)}. 
    \]
    On the other hand, whenever \(p \in A \setminus \{i\}\), we have 
    \[
    c^{(p)} = \gamma_{\text{sp}} \mathbf{1}^{\{p\}} + \sum_{m \in A \diamond_i B} \gamma_m \mathbf{1}^{\{m\}} = \gamma_{\text{sp}} \mathbf{1}^{\{p\}} + \sum_{k \in A \setminus \{i\}} \gamma(\{k\}) \mathbf{1}^{\{k\}} + \sum_{l \in B} \gamma(\{l\}) \mathbf{1}^{\{l\}}. 
    \]
    Let us evaluate the coefficients in this sum. If \(m \in A \setminus \{i\}\), we have \(\gamma(\{m\}) = \beta(B) \alpha(\{m\})\), and if \(m \in B\), we have \(\gamma(\{m\}) = \alpha(\{i\}) \beta(\{m\}) = 0\). From this, we compute 
    \[ \begin{aligned}
    \gamma_{\text{sp}} = \gamma_C - \sum_{m \in A \diamond_i B} \gamma_m & = \alpha(A) \beta(B) - \sum_{k \in A\setminus \{i\}} \gamma(\{k\}) - \sum_{l \in B} \gamma(\{l\}) \\
    & = \alpha(A) \beta(B) - \beta(B) \sum_{k \in A \setminus \{i\}} \alpha(\{k\}) - \alpha(\{i\}) \sum_{l \in B} \beta(\{l\}) \\
    & = \beta(B) \alpha_{\text{sp}}. 
    \end{aligned} \]
    Then, 
    \[
    c^{(p)} = \beta(B) \alpha_{\text{sp}} \mathbf{1}^{\{p\}} + \beta(B) \sum_{k \in A \setminus \{i\}} \alpha(\{k\}) \mathbf{1}^{\{k\}} = \beta(B) a^{(p)} = a^{(p)} \otimes_i b^{(q)}. 
    \]
    Notice that in this case, \(c^{(p)}\) is the partial tensor product between \(a^{(p)}\) and any imputation of \(I(\Gamma_B)\). We finish the first part of the proof by looking at the tensor product
    \[ \begin{aligned} 
    a^{(i)} \otimes_i b^{(q)} & = \alpha_{\text{sp}} \mathbf{1}^{\{i\}} \otimes_i b^{(q)} + \sum_{k \in A \setminus \{i\}} \alpha(\{k\}) \mathbf{1}^{\{k\}} \otimes_i b^{(q)} \\
    & = \alpha_{\text{sp}} \beta(B) \mathbf{1}^{\{q\}} + \beta(B) \sum_{k \in A \setminus \{i\}} \alpha(\{k\})
    \mathbf{1}^{\{k\}}. 
    \end{aligned} \]
    On the other hand, for \(q \in B\), we have 
    \[
    c^{(q)} = \beta(B) \alpha_{\text{sp}} \mathbf{1}^{\{q\}} + \beta(B) \sum_{k \in A \setminus \{i\}} \alpha(\{k\}) \mathbf{1}^{\{k\}} = a^{(i)} \otimes_i b^{(q)}, 
    \]
    and so any vertex of \(I(\Gamma_A \circ_i \Gamma_B)\) is a partial tensor product of vertices of \(I(\Gamma_A)\) and \(I(\Gamma_B)\). 

    \medskip

    By convexity, any imputation \(z \in I(\Gamma_A \circ_i \Gamma_B)\) can be written as a convex combination
    \[
    z = \sum_{m \in A \diamond_i B} \lambda_m c^{(m)} = \sum_{k \in A \setminus \{i\}} \lambda_k c^{(k)} + \sum_{l \in B} \lambda_l c^{(l)}, 
    \]
    with \(\lambda_m, \lambda_k, \lambda_l\) all non-negative and \(\sum_{m \in A \diamond_i B} \lambda_m = \sum_{k \in A \setminus \{i\}} \lambda_k + \sum_{l \in B} \lambda_l = 1\). Set \(\lambda_A = \sum_{k \in A \setminus \{i\}} \lambda_k\) and \(\lambda_B = \sum_{l \in B} \lambda_l\). Define, for all \(k \in A \setminus \{i\}\) and \(l \in B\), the new coefficients \(\lambda'_k = \lambda_A^{-1} \lambda_k\) and \(\lambda'_l = \lambda_B^{-1} \lambda_l\). Hence, 
    \[
    \sum_{k \in A \setminus \{i\}} \lambda'_k = \lambda_A^{-1} \sum_{k \in A \setminus \{i\}} \lambda_k = 1, \qquad \text{and} \qquad \sum_{l \in B} \lambda'_l = \lambda_B^{-1} \sum_{l \in B} \lambda = 1. 
    \]
    We can finally write \(z\) as 
    \[
    z = \lambda_A \sum_{k \in A \setminus \{i\}} \lambda'_k c^{(k)} + \lambda_B \sum_{l \in B} \lambda'_l c^{(l)}, \qquad \text{with} \quad \lambda_A + \lambda_B = 1, \quad \lambda_A, \lambda_B \geq 0. 
    \]
    Using the fact that the vertices can split into partial tensor products, we have 
    \[ \begin{aligned} 
    z & = \sum_{k \in A \setminus \{i\}} \lambda_k a^{(k)} \otimes_i b^{(q_k)} + \sum_{l \in B} \lambda_l a^{(i)} \otimes_i b^{(l)} \\
    & = \sum_{k \in A \setminus \{i\}} \lambda_k a^{(k)} \otimes_i b^{(q_k)} + \lambda_B a^{(i)} \otimes_i \sum_{l \in B} \lambda'_l b^{(l)}. 
    \end{aligned} \]
    For \(k \in A \setminus \{i\}\), the vertex \(c^{(k)}\) decomposes using \(a^{(k)}\) and any imputation we want from \(I(\Gamma_B)\). Because \(\sum_{l \in B}  \lambda'_l = 1\) and \(\lambda'_l \geq 0\) for all \(l \in B\), we have that \(y = \sum_{l \in B} \lambda'_l b^{(l)}\) is a convex combination of the vertices of \(I(\Gamma_B)\), and by convexity, is itself a imputation, i.e., \(y \in I(\Gamma_B)\). Thus, 
    \[
    z = \sum_{k \in A \setminus \{i\}} \lambda_k a^{(k)} \otimes_i y + \lambda_B a^{(i)} \otimes_i y = \left( \lambda_B a^{(i)} +  \sum_{k \in A \setminus \{i\}} \lambda_k a^{(k)} \right) \otimes_i y. 
    \]
    Define \(x = \lambda_B a^{(i)} + \sum_{k \in A \setminus \{i\}} \lambda_k a^{(k)}\). By definition of \(\lambda_B\), we have that 
    \[
    \lambda_B + \sum_{k \in A \setminus \{i\}} \lambda_k = \lambda_B + \lambda_A = 1. 
    \]
    Moreover, \(\lambda_k \geq 0\) for all \(k \in A \setminus \{i\}\). So, \(x\) is a convex combination of the vertices of \(I(\Gamma_A)\), and therefore \(x \in I(\Gamma_A)\) by convexity. So, we have that 
    \[
    z = x \otimes_i y, \qquad \text{with} \qquad x \in I(\Gamma_A) \quad \text{and} \quad y \in I(\Gamma_B), 
    \]
    which concludes the proof. 
\end{proof}

\subsection{The core of a composite game}\label{Subsection: the core of a composite game}
As we explained in Paragraph~\ref{subsubsection: the core of a game}, the core is one of the most used solution concepts in cooperative game theory. Its elements correspond to ways of distributing the total wealth of the game $v(N)$ in such a way that any proper coalition is better off than if it stayed alone. Recall that, for a game \(\Gamma = (N, v)\), it is the polytope $C(\Gamma)$ given by 
\[
C(\Gamma) = \{x \in X(\Gamma) \mid x(S) \geq v(S), \hspace{2pt} \forall S \subseteq N\}. 
\]
It might be empty, and is non-empty if and only if the game $\Gamma$ is balanced by Theorem \ref{thm: bondareva-shapley}. 

\begin{theorem}\label{thm: inclusion of the tensor i of the cores}
Let \(\Gamma_A = (A, \alpha)\) and \(\Gamma_B = (B, \beta)\) be two non-negative games, and consider \(i \in A\). The partial tensor product restricts to a well defined map
    \[
    (- \otimes_i -): C(\Gamma_A) \times C(\Gamma_B) \longrightarrow C(\Gamma_A \circ_i \Gamma_B)~, 
    \]
which is furthermore injective whenever $\alpha(\{i\})$ and $\beta(B)$ are positive. 
\end{theorem}

\begin{proof}
If the cores are empty, then the proposition is trivial. So let $x \in C(\Gamma_A)$ and $y \in C(\Gamma_B)$. The element $z= x \otimes_i y$ is in core $C(\Gamma_A \circ_i \Gamma_B)$ if the following two conditions are satisfied.
\begin{enumerate}
    \item It is a preimputation, meaning that \(\sum_{j \in A \diamond_i B} z_j = (\alpha \circ_i \beta)(A \diamond_i B) = \alpha(A)\beta(B)\).
    
    \item It satisfies the inequality \(\sum_{j \in S} z_j \geq \alpha \circ_i \beta(S)\), for any $S \subseteq A \diamond_i B$. 
\end{enumerate}

First, notice that both $x$ and $y$ are non-negative vectors. Indeed, we have $x_i \geq v(i) \geq 0$ for all players $i$ in $A$ and likewise for $y$. The first point follows form Lemma \ref{lemma: inclusion of preimputations} since $x \in X(\Gamma_A)$ and $y \in X(\Gamma_B).$

\medskip

Now, we must show that \(z(S) \geq \alpha \circ_i \beta(S)\), for all \(S \subseteq A \diamond_i B\). Because \(x \in C(\Gamma_A)\), we have that
\[
\begin{aligned}
\alpha \circ_i \beta(S) & = \left[ \beta(B) - \beta(S_B) \right] \alpha(S_A) + \beta(S_B) \alpha(S_A \cup \{i\}) \\
& \leq \left[ \beta(B) - \beta(S_B) \right] \sum_{j \in S_A} x_j + \beta(S_B) \sum_{j \in S_A \cup \{i\}} x_j.
\end{aligned} \]
Now, we can factorize by \(\beta(S_B)\) and using the fact that \(y \in C(\Gamma_B)\) to get
\[
\alpha \circ_i \beta(S) \leq \beta(B) \sum_{j \in S_A} x_j + x_i \beta(S_B) \leq \left(\sum_{k \in B} y_k \right) \sum_{j \in S_A} x_j + x_i \sum_{k \in S_B} y_k = x \otimes_i y(S)~.
\]
Hence the element $z= x \otimes_i y$ is in the core of the composite. The injectivity then follows from Proposition \ref{prop: tenseur i des imputations} since $C(\Gamma_A \circ_i \Gamma_B) \subseteq I(\Gamma_A \circ_i \Gamma_B)$, any element in $C(\Gamma_A \circ_i \Gamma_B)$ has at most one pre-image by the partial tensor product which lies in $I(\Gamma_A) \times I(\Gamma_B)$, hence in $C(\Gamma_A) \times C(\Gamma_B).$
\end{proof}

\begin{remark}
    The fact that this map is well defined in the above theorem gives back \cite[Theorem 4]{OwenTensor} in the specific case of the total composition of normalized games
\end{remark}

\begin{example}
    We provide a small counterexample with two totally monotone, therefore convex, games for which the partial tensor product of the cores does not coincide with the core of the composite game. Let \(\Gamma_A = (A, \alpha)\) and \(\Gamma_B = (B, \beta)\) be the games defined in the previous example by 
    \[ \begin{aligned} 
    \alpha(a_1) = 1, \quad \alpha(a_2) = 1, \quad & \text{and} \quad \alpha(a_1 a_2) = 3 \\
    \beta(b_1) = 0, \quad \beta(b_2) = 2, \quad & \text{and} \quad \beta(b_1 b_2) = 3. 
    \end{aligned} \]
    The imputations simplices are given by 
    \[
    I(\Gamma_A) = \operatorname{conv} \left\{\begin{pmatrix} 2 \\ 1 \end{pmatrix}, \begin{pmatrix} 1 \\ 2 \end{pmatrix} \right\}, \quad \text{and} \quad I(\Gamma_B) = \operatorname{conv} \left\{ \begin{pmatrix} 1 \\ 2 \end{pmatrix}, \begin{pmatrix} 0 \\ 3 \end{pmatrix} \right\}. 
    \]
    Because they are \(2\)-player games, their cores coincide with their imputation simplices. Moreover, each of the vertices corresponds to a marginal vector, i.e., 
    \[
    m^{(12)}_\alpha = \begin{pmatrix} 2 \\ 1 \end{pmatrix}, \quad m^{(21)}_\alpha = \begin{pmatrix} 1 \\ 2 \end{pmatrix}, \quad m^{(12)}_\beta = \begin{pmatrix} 0 \\ 3 \end{pmatrix}, \quad \text{and} \quad m^{(21)}_\beta = \begin{pmatrix} 1 \\ 2 \end{pmatrix}.
    \]
    Hence, we have 
    \[
    C(\Gamma_A) \otimes_1 C(\Gamma_B) = \operatorname{conv} \left\{\begin{pmatrix} 0 \\ 6 \\ 3 \end{pmatrix}, \begin{pmatrix} 2 \\ 4 \\ 3 \end{pmatrix}, \begin{pmatrix} 0 \\ 3 \\ 6 \end{pmatrix},  \begin{pmatrix} 1 \\ 2 \\ 6 \end{pmatrix} \right\}. 
    \]
    The composition \(\Gamma_C = \Gamma_A \circ_i \Gamma_B\) is given by 
    \[ \begin{aligned} 
    \gamma(b_1) = 0, \quad & \gamma(b_2) = 2, \quad \gamma(a_2) = 3, \\
    \gamma(b_1b_2) = 3, \quad & \gamma(b_1 a_2) = 3, \quad \gamma(b_2 a_2) = 5, \\
    & \gamma(b_1 b_2 a_2) = 9. 
    \end{aligned} \]
    The marginal vectors of \(\Gamma_C\) are 
    \[ \begin{aligned} 
    m^{(123)}_\gamma = \begin{pmatrix} 0 \\ 3 \\ 6 \end{pmatrix} = m^{(21)}_\alpha \otimes_1 m^{(21)}_\beta, \qquad & m^{(132)}_\gamma = \begin{pmatrix} 0 \\ 6 \\ 3 \end{pmatrix} = m^{(12)}_\alpha \otimes_1 m^{(21)}_\beta, \\
    m^{(213)}_\gamma = \begin{pmatrix} 1 \\ 2 \\ 6 \end{pmatrix} = m^{(21)}_\alpha \otimes_1 m^{(12)}_\beta, \qquad & m^{(231)}_\gamma = \begin{pmatrix} 4 \\ 2 \\ 3 \end{pmatrix} \not \in C(\Gamma_A) \otimes_1 C(\Gamma_B), \\
    m^{(312)}_\gamma = \begin{pmatrix} 0 \\ 6 \\ 3 \end{pmatrix} = m^{(12)}_\alpha \otimes_1 m^{(21)}_\beta, \qquad & m^{(321)}_\gamma = \begin{pmatrix} 4 \\ 2 \\ 3 \end{pmatrix} \not \in C(\Gamma_A) \otimes_1 C(\Gamma_B). 
    \end{aligned} \]
\end{example}

\begin{theorem}\label{thm: fake surjectivity des coeurs}
Let \(\Gamma_A = (A, \alpha)\) and \(\Gamma_B = (B, \beta)\) be two non-negative games, and let \(i \in A\). Suppose that $\alpha(\{i\}) = 0$ and \(\beta(\{l\}) = 0\) for all \(l \in B\),  as well as $\beta(B) > 0$. Then any core element $z \in C(\Gamma_A \circ_i \Gamma_B)$ decomposes as $x \otimes_i y$ with $x \in C(\Gamma_A)$ and $y \in I(\Gamma_B)$. 
\end{theorem}

\begin{proof}
Let $z$ be an element in $C(\Gamma_A \circ_i \Gamma_B)$. Since $C(\Gamma_A \circ_i \Gamma_B) \subseteq I(\Gamma_A \circ_i \Gamma_B)$, then $z = x \otimes_i y$ where $x \in I(\Gamma_A)$ and $y \in I(\Gamma_B)$ by Proposition \ref{prop: tensor of imputations is surjective}. Let us show that $x$ is in $C(\Gamma_A)$, that is, that for all $S \subseteq A$, we have that \(x(S) \geq \alpha(S)\). For $S \subseteq A$, we consider the two following cases.

\medskip

\begin{enumerate}
    \item Suppose that $i \notin S$. Since $x \otimes_i y$ is in $C(\Gamma_A \circ_i \Gamma_B)$, we have that
    \[
    \beta(B)x(S) = x \otimes_i y(S) \geq \alpha \circ_i \beta(S) = \beta(B)\alpha(S)~,  
    \]
    which implies that \(x(S) \geq \alpha(S)\), since $\beta(B) > 0$. So $x$ satisfies the condition with respect to the coalition $S$. 

\medskip

    \item Suppose that $i \in S$. We denote $S_A = S \cap (A \setminus \{i\})$. Notice that, on the one side, we have 
    \[
    \beta(B)x(S) = \beta(B)x(S_A) + x_i\beta(B) = \beta(B)x(S_A) + x_i y(B) = x \otimes_i y(S \cup B)~,
    \]
    and on the other side, 
    \[
    \alpha \circ_i \beta(S \cup B) = \beta(B) \left( \alpha(S_A) + \partial_i \alpha(S_A) \right) = \beta(B)\alpha(S_A \cup \{i\}) = \beta(B) \alpha(S)~. 
    \]
    So, since  $x \otimes_i y$ is in $C(\Gamma_A \circ_i \Gamma_B)$, this gives that \(\beta(B)x(S) \geq \beta(B)\alpha(S)\), which implies that \(x(S) \geq \alpha(S)\), since $\beta(B) > 0$. So $x$ satisfies the condition with respect to $S$. 
\end{enumerate}
By combining the two different cases, we see that $x$ satisfies the condition with respect to any $S \subseteq A$ and therefore that $x$ is in $C(\Gamma_A).$
\end{proof}

\begin{remark}
    In particular, under the same assumptions as in Theorem~\ref{thm: fake surjectivity des coeurs}, if $C(\Gamma_A)$ is empty, so is $C(\Gamma_A \circ_i \Gamma_B)$. This means that, when the quotient game is not balanced, no matter how many details we get about the decision made by each player, it will not change this fact. Hence, the balancedness of a game, or the social context it represents, is robust under the level of details under study, and actually characterizes the situation modeled, rather than the model itself. 
\end{remark}

\begin{remark}
Notice that, up to \textit{strategic equivalence}, the value at every player of a game can be supposed to be zero, thus the hypothesis of Theorem \ref{thm: fake surjectivity des coeurs} are satisfied already if $\beta(B)$ is positive. For more details on the notion of strategic equivalence, see \cite[Pages 245-246]{vonNeumann2} or \cite{McKinsey}. 
\end{remark}

\subsection{Shapley value of composite cooperative games} The Shapley value associates to a cooperative game a unique vector. It satisfies many natural properties such as \emph{symmetry}, meaning that the Shapley value of two players contributing the same amount to the same coalitions is the same, \emph{efficiency}, the sum of the payment equals the value of the grand coalition, \emph{linearity} with respect to the game~\cite{Shapley1953}. See \cite{algaba2019handbook} for a detailed account. One can define it in terms of derivatives, as it was done originally in ~\cite{Shapley1953}, or in terms of the M{\"o}bius transform of the game, see \cite{harsanyi1958bargaining}. 

\begin{definition}[Shapley value]
Let \(\Gamma = (N, v)\) be a game. The Shapley value \(\phi_i\) of player \(i \in N\) is given by
\[
\phi_i(\Gamma) = \sum_{S \subseteq N \setminus \{i\}} \gamma_S^{-1} \partial_iv(S) = \sum_{\substack{S \subseteq N \\ S \ni i}} \frac{\mu^v(S)}{\lvert S \rvert},
\]
with \(\gamma_S\) being the binomial coefficient \(\binom{\lvert N \rvert - 1}{\lvert S \rvert}\) divided by \(\lvert N \rvert\). By considering the Shapley values with respect to all players $i \in N$, we get the Shapley value $\phi(\Gamma)$ which is a vector in $\mathbb{R}^N$.
\end{definition}

\begin{remark}
The Shapley value can be generalized to any coalitions (and not only to single players), as it was done by Grabisch~\cite{grabisch1997alternative} under the name of \emph{(Shapley) interaction value}. 
\end{remark}

As already noticed by Shapley himself, the Shapley value is not compatible with the composition, see \cite{ShapleyComposition}. It is in fact for this reason that Owen introduced a variant of the Shapley value which is compatible with the composition in \cite{Owen77}. Using the partial composition formula on the basis of unanimity games of Subsection \ref{subsection: unanimity basis}, we obtain an explicit formula for the Shapley value of a composite game. 

\medskip

There is a simple way to define the Shapley value using the basis of unanimity games: it is the unique linear function $\phi: \mathbb{G}(n) \longrightarrow \mathbb{R}^n$ which is given on this basis by 
\[
\phi_j \left( \mathbf{1}^S \right) = \begin{cases}
        \frac{1}{ \lvert S \rvert }, & \text{if } j \in S, \\
        0, & \text{otherwise}. 
    \end{cases}
\]
It follows from this that the Shapley value of a generic game $\Gamma = (N, v) \in \mathbb{G}(n)$, which written in this basis is given by 
\[
v = \sum_{S \subseteq [n]} \lambda_S u_S~,
\]
is the vector in $\mathbb{R}^n$ whose $j$-th coordinate is given by the formula
\[
\phi(\Gamma)_j = \sum_{\substack{S \subseteq N \\ S \ni j}}  \frac{\lambda_S}{|S|}~,
\]
for all $1 \leq j \leq n$. 

\begin{proposition}\label{prop: shapley value of a composite game}
    Let \(\Gamma_A = (A, \alpha)\) and \(\Gamma_B = (B, \beta)\) be two normalized games, let \(i \in A\), and let \(\Gamma_C = \Gamma_A \circ_i \Gamma_B\). Then, we have 
    \[ 
    \phi_j \left( \Gamma_C \right) = \begin{cases}
        \displaystyle \sum_{\substack{S \subseteq A \\ S \ni i}} \sum_{\substack{T \subseteq B \\ T \ni j}} \frac{\mu^\alpha_S \mu^\beta_T}{\lvert S \diamond_i T \rvert}, & \quad \text{if } j \in B, \\
        \displaystyle \phi_j \left( \Gamma_{A \setminus \{i\}} \right) + \sum_{\substack{S \subseteq A \\ S \supseteq \{i, j\}}} \sum_{T \subseteq B} \frac{\mu^\alpha_S \mu^\beta_T}{\lvert S \diamond_i T \rvert}, & \quad \text{if } j \in A \setminus \{i\}, 
    \end{cases}
    \]
    where \(\Gamma_{A \setminus \{i\}}\) denotes the restriction of \(\Gamma_A\) to the coalition \(A \setminus \{i\}\). 
\end{proposition}

\begin{proof}
Recall that if $\alpha$ and $\beta$ are written in the basis of unanimity games, then their partial composition is given by the following formula:
\[
\alpha \circ_i^{u} \beta = \sum_{\substack{S \subseteq A \\ S \not \ni i}} \mu^\alpha_S u_S + \sum_{\substack{S \subseteq A \\ S \ni i}} \sum_{T \subseteq B} \mu^\alpha_S \mu^\beta_T u_{S \diamond_i T}~.
\]
Using this formula, we can directly compute what happens in each of these cases since the Shapley value is linear.  

\begin{enumerate}
    \item Let $j \in B$, we have that:
    	\[
    	\phi_j \left( \Gamma_C \right) = \phi_j \Bigg( \sum_{\substack{S \subseteq A \\ S \ni i}} \sum_{\substack{T \subseteq B \\ T \ni j}} \mu^\alpha_S \mu^\beta_T u_{S \diamond_i T} \Bigg) = \sum_{\substack{S \subseteq A \\ S \ni i}} \sum_{\substack{T \subseteq B \\ T \ni j}} \mu^\alpha_S \mu^\beta_T \phi_j \left( u_{S \diamond_i T}\right), 
    	\] 
    	where $\phi_j\left(u_{S \diamond_i T} \right) = 1/ \lvert S \diamond_i T \rvert$ whenever \(j \in S \diamond_i T\).
    	
    	\medskip
    	
    	\item Let $j \in A \setminus \{i \}$, we have that:
    \[
    	\phi_j \left( \Gamma_C \right) = \phi_j \Bigg( \sum_{\substack{S \subseteq A \setminus \{i\} \\ S \ni j}} \lambda_S u_S \Bigg) + \phi_j \Bigg( \sum_{\substack{S \subseteq A \\ S \supseteq \{i,j\}}} \sum_{\substack{T \subseteq B}} \mu^\alpha_S \mu^\beta_T u_{S \diamond_i T} \Bigg),
    	\] 
    	where the second sum can be easily identified with the second term in the formula of the proposition by the same arguments as above. It remains to identify the first term, which can be seen to correspond to the Shapley value of the restriction of $\Gamma_A$ to $A \setminus \{i\}$.   	
\end{enumerate}
\end{proof}

\subsection{Banzhaf value of composite cooperative games} The Banzhaf index, originally introduced by Penrose in \cite{Penrose}, and later reintroduced by Banzhaf in \cite{banzhaf1968one}, is a power index which measures the power of a voter in a voting game where voting rights are not equally distributed among players. A player's power in a voting game is here measured by the likelihood of a situation in which his individual vote is decisive. We refer to \cite{banzhaf1968one, OwenMultilinear, coleman2012control} for more details on the Banzhaf index. 

\begin{definition}[Banzhaf index]
    Let \(\Gamma = (N, v)\) be a cooperative game, and \(i \in N\) be a player. Its \emph{Banzhaf index} is defined~by 
    \[
    \psi_i(\Gamma) = \frac{1}{2^{n-1}} \sum_{S \subseteq N \setminus \{i\}} \partial_i v(S) = \frac{1}{2^{n-1}} \sum_{\substack{S \subseteq N \\ i \in S}} \mu^v_S~. 
    \]
    The Banzhaf indices give a single vector $\psi(\Gamma)$ in $\mathbb{R}^n$ whose $i$-th coordinate is $\psi_i(\Gamma)$. 
\end{definition}

Notice that the Banzhaf index resembles the Shapley value, and differs only on the weights of the linear combination of the M{\"o}bius coefficients. A mathematical advantage of the Banzhaf index compared to the Shapley value is that these weights do not depend on the associated coalition, and corresponds to the inverse of the number of coalitions (counting the empty set, as \(\partial_i v\) is not necessarily grounded) on which the partial derivative of the game is defined. 

\medskip 

Some authors call this index the \emph{normalized} Banzhaf index, and the non-normalized is obtained by dropping the \(\frac{1}{2^{n-1}}\) factor. We do not follow this convention here. The Banzhaf index can also be defined as the unique linear map whose image on this unanimity basis is given by 

\[ \begin{aligned} 
\psi: \mathbb{G}(n) & \xrightarrow{\hspace{1.5cm}} \mathbb{R}^n \\
u_S & \xmapsto{\hspace{1.5cm}} \frac{1}{2^{n-1}} \mathbf{1}^S. 
\end{aligned} 
\]
\vspace{0.1pc}

We show that the Banzhaf index of a composite game is fully determined by the partial tensor product of the Banzhaf value of the quotient game and the Shapley value of the inserted game.

\begin{theorem}\label{thm: Banzhaf value of a composite game}
    Let \(\Gamma_A = (A, \alpha)\) and \(\Gamma_B = (B, \beta)\) be two games, and let \(i \in A\). The Banzhaf index of the composite game is given by 
    \[
    \psi \left(\Gamma_A \circ_i \Gamma_B \right) = \frac{\psi(\Gamma_A) \otimes_i \phi(\Gamma_B)}{2^{\lvert B \rvert - 1}} ~,
    \]
    where these are two equal vectors in $\mathbb{R}^{\lvert A \rvert + \lvert B \rvert -1}$.
\end{theorem}

\begin{proof}
The main idea is to write the two games \(\Gamma_A = (A, \alpha)\) and \(\Gamma_B = (B, \beta)\) in the basis of unanimity games and to use the linearity of the Banzhaf index with respect to this basis to compute its value. If $\alpha$ and $\beta$ are written in the basis of unanimity games, we can compute that

    \[ \begin{aligned} 
    \psi( \Gamma_A \circ_i \Gamma_B) & = \psi \left( \sum_{S \subseteq A} \mu^\alpha_S u^A_S \circ_i \sum_{T \subseteq B} \mu^\beta_T u^B_T \right) \\
    & = \psi \left( \sum_{S \subseteq A} \sum_{T \subseteq B} \mu^\alpha_S \mu^\beta_T \left( u^A_S \circ_i u^B_T \right) \right) \\
    & = \sum_{S \subseteq A} \sum_{T \subseteq B} \mu^\alpha_S \mu^\beta_T \psi \left( u^{A \diamond_i B}_{S \diamond_i T} \right) \\
    & = \sum_{S \subseteq A} \sum_{T \subseteq B} \mu^\alpha_S \mu^\beta_T \frac{1}{2^{\lvert A \rvert + \lvert B \rvert - 2}} \mathbf{1}^{S \diamond_i T}~.  \\
    \end{aligned} \]
Now, the key input in this computation is the following equality:
\[
\mathbf{1}^{S \diamond_i T} = \mathbf{1}^S \otimes_i \frac{1}{\lvert T \rvert} \mathbf{1}^T~,
\]
which holds for all $S \subseteq A$ and $T \subseteq B$. Combining it with the equality 
\[
\frac{1}{2^{\lvert A \rvert + \lvert B \rvert - 2}} = \frac{1}{2^{\lvert A \rvert - 1}} \frac{1}{2^{\lvert B \rvert - 1}}~,
\]
we finally get that 
    \[ \begin{aligned} 
    \psi(\Gamma_A \circ_i \Gamma_B) & = \sum_{S \subseteq A} \sum_{T \subseteq B} \mu^\alpha_S \mu^\beta_T \frac{1}{2^{\lvert A \rvert - 1}} \frac{1}{2^{\lvert B \rvert - 1}} \mathbf{1}^S \otimes_i \frac{1}{\lvert T \rvert} \mathbf{1}^T \\
    & = \left( \sum_{S \subseteq A} \mu^\alpha_S \frac{1}{2^{\lvert A \rvert - 1}} \mathbf{1}^S \right) \otimes_i \left( \sum_{T \subseteq B} \mu^\beta_T \frac{1}{2^{\lvert B \rvert - 1}} \frac{1}{\lvert T \rvert} \mathbf{1}^T \right) \\
    & = \frac{1}{2^{\lvert B \rvert - 1}}\psi(\Gamma_A) \otimes_i \phi(\Gamma_B)~, \\
    \end{aligned}
    \]
where we also used the linearity of the partial tensor product $\otimes_i$ with respect to each variable. 
\end{proof}

\bibliographystyle{alpha}
\bibliography{bibabel}

\end{document}